\documentclass[11pt,a4paper,reqno]{amsart}
 \usepackage{amsfonts,amsmath,amscd,amssymb,amsbsy,amsthm,amstext,amsopn,amsxtra}
 \usepackage{color,mathrsfs,verbatim}
 \usepackage{stmaryrd}
 \usepackage{extarrows}
 \usepackage[all]{xy}
 \usepackage{bm} 

\usepackage{todonotes}
\usepackage{hyperref}
\hypersetup{colorlinks,linkcolor=[rgb]{0,0,1},citecolor=[rgb]{1,0,0}}

\usepackage{paralist}

\usepackage[tmargin=1in, bmargin=1in, lmargin=1.05in, rmargin=1.05in]{geometry}

 \numberwithin{equation}{section}

\newtheorem{theorem}{Theorem}[subsection]
\newtheorem{lemma}[theorem]{Lemma}
\newtheorem{proposition}[theorem]{Proposition}
\newtheorem{corollary}[theorem]{Corollary}

\newtheorem{assumption}[theorem]{Assumption}

\theoremstyle{definition}
\newtheorem{definition}[theorem]{Definition}
\newtheorem{example}[theorem]{Example}
\newtheorem{examples}[theorem]{Examples}

\theoremstyle{remark}
\newtheorem{remark}[theorem]{Remark}
\newtheorem{remarks}[theorem]{Remarks}

\renewcommand{\geq}{\geqslant}
\renewcommand{\leq}{\leqslant}
\newcommand{\ra}{\rightarrow}
\newcommand{\lra}{\longrightarrow}
\newcommand*{\longhookrightarrow}{\ensuremath{\lhook\joinrel\relbar\joinrel\rightarrow}}
\renewcommand{\tilde}{\widetilde}
\renewcommand{\bar}{\overline}

\newcommand{\cA}{\mathcal{A}}

\newcommand{\cD}{D}
\newcommand{\cE}{\mathcal{E}}
\newcommand{\cF}{\mathcal{F}}

\newcommand{\cH}{\mathcal{H}}

\newcommand{\cM}{\mathcal{M}}
\newcommand{\cO}{\mathcal{O}}
\newcommand{\cP}{\mathcal{P}}

\newcommand{\kk}{\ensuremath{\Bbbk}} 
\renewcommand{\AA}{\ensuremath{\mathbb{A}}}
\newcommand{\CC}{\ensuremath{\mathbb{C}}}
\newcommand{\DD}{\ensuremath{\mathbb{D}}}
\newcommand{\NN}{\ensuremath{\mathbb{N}}} 
\newcommand{\PP}{\ensuremath{\mathbb{P}}}
 
\newcommand{\RR}{\ensuremath{\mathbb{R}}} 

\newcommand{\ZZ}{\ensuremath{\mathbb{Z}}}

\newcommand{\Spec}{\ensuremath{\mathrm{Spec}}} 

\newcommand{\cdc}{D_{c}}   
\newcommand{\cdlc}{D_{\mathrm{lc}}} 

\newcommand{\cdcp}{D_c^{\geq 0}}  
\newcommand{\cdlcp}{D_{\mathrm{lc}}^{\geq 0}} 
\newcommand{\cdcn}{D_c^{\leq 0}}  
\newcommand{\cdlcn}{D_{\mathrm{lc}}^{\leq 0}} 
\newcommand{\ypr}{Y \times \mathbb{P}^r} 
  
\newcommand{\dabc}{D^{[a,b]}_c}  
\newcommand{\dlabc}{D_{\mathrm{lc}}^{[a,b]}}

\newcommand{\git}{\ensuremath{/\!\!/_{\!\theta}\,}}

\newcommand{\one}{\ensuremath{(\mathrm{i})}}
\newcommand{\two}{\ensuremath{(\mathrm{ii})}}
\newcommand{\three}{\ensuremath{(\mathrm{iii})}}
\newcommand{\four}{\ensuremath{(\mathrm{iv})}}

\newcommand{\bt}{\,\boxtimes\,}   

\DeclareMathOperator{\Coh}{Coh}
\DeclareMathOperator{\cok}{cok}
\DeclareMathOperator{\cone}{cone}
\newcommand{\Dfin}{D_{\operatorname{fin}}}
\newcommand{\Dperf}{D_{\operatorname{perf}}}
\DeclareMathOperator{\End}{End} 
\DeclareMathOperator{\Ext}{Ext} 
\DeclareMathOperator{\Hom}{Hom}
\DeclareMathOperator{\id}{id}

\newcommand{\Knum}{K^{\operatorname{num}}}
\newcommand{\Knumc}{K^{\operatorname{num}}_c}
\newcommand{\Knumperf}{K_{\operatorname{perf}}^{\operatorname{num}}}

\DeclareMathOperator{\pa}{p-a}
\DeclareMathOperator{\Pic}{Pic}

\DeclareMathOperator{\Qcoh}{Qcoh}

\DeclareMathOperator{\rk}{rk}

\DeclareMathOperator{\shom}{\mathcal{H}\!{\it om}}

\DeclareMathOperator{\spec}{Spec}
\DeclareMathOperator{\Stab}{Stab}
\DeclareMathOperator{\supp}{Supp}

\newcommand{\vv}{\bm{\mathrm v}}   

\begin{document}

\title{Nef divisors for moduli spaces of complexes with compact support}
\author{Arend Bayer}
\address{School of Mathematics and Maxwell Institute,
The University of Edinburgh,
James Clerk Maxwell Building,
Peter Guthrie Tait Road, Edinburgh, Scotland EH9 3FD,
United Kingdom}
\email{arend.bayer@ed.ac.uk}
\urladdr{http://www.maths.ed.ac.uk/~abayer/}

\author{Alastair Craw} 
\address{Department of Mathematical Sciences, 
University of Bath, 
Claverton Down, 
Bath BA2 7AY, 
United Kingdom.}
\email{a.craw@bath.ac.uk}
\urladdr{http://people.bath.ac.uk/ac886/}

\author{Ziyu Zhang}
\address{Department of Mathematical Sciences, 
University of Bath, 
Claverton Down, 
Bath BA2 7AY, 
United Kingdom.}
\email{zz505@bath.ac.uk}
\urladdr{http://people.bath.ac.uk/zz505/}
\curraddr{Institute for Algebraic Geometry, Leibniz University Hannover, Welfengarten 1, 30167 Hannover, Germany}

\dedicatory{In memory of Johan Louis Dupont}


\begin{abstract}
In \cite{BM12}, the first author and Macr{\`{\i}} constructed a family of nef divisors on any moduli space of Bridgeland-stable objects on a smooth projective variety $X$. In this article, we extend this construction to the setting of any separated scheme $Y$ of finite type over a field, where we consider moduli spaces of Bridgeland-stable objects on $Y$ with compact support.  
We also show that the nef divisor is compatible with the polarising ample line bundle coming from the GIT construction of the moduli space in the special case when $Y$ admits a tilting bundle and the stability condition arises from a $\theta$-stability condition for the endomorphism algebra. 

Our main tool generalises the work of Abramovich--Polishchuk \cite{AP06} and
Polishchuk \cite{P07}: given a t-structure on the derived category $D_c(Y)$ on $Y$ of objects with compact
support and a base scheme $S$, we construct a constant family of t-structures on a category of objects on
$Y \times S$ with compact support relative to $S$. 
\end{abstract}

\maketitle
\tableofcontents
\setdefaultenum{(i)}{a)}{}{}

\section{Introduction}

\subsection{Motivation}
In recent years, a number of authors have applied wall-crossing techniques for Bridgeland stability conditions in order to systematically study the birational geometry of moduli spaces; see section \ref{subsec:moremotivation} for more background. The Positivity Lemma of ~\cite{BM12} provides a clear, geometric link between the stability manifold and the moveable cone of the moduli space by producing a family of nef divisors on any moduli space of Bridgeland-stable objects on a smooth projective variety $X$.

 Rich and interesting wall-crossing structures have also been observed in semi-local settings, including, for example, the resolution of singularities $Y \to \Spec R$ of an affine singularity, with many interesting examples coming from geometric representation theory or the study of algebras that are finite over their centre. The main goal of this paper is to extend the machinery of \cite{BM12} to such settings. In fact our approach works more generally, replacing $X$ by any separated scheme $Y$ of finite type over $\kk$.

\subsection{The main result}
Let $\kk$ be an algebraically closed field, and let $Y$ be a separated scheme of finite type over $\kk$.  Let $D(Y)$ denote the bounded derived category of coherent sheaves on $Y$, and $\cdc(Y)$ the full subcategory of objects with proper support.  As we explain in more detail at the beginning of Section \ref{sect:stabilityconditions}, the usual notion of \emph{numerical stability conditions} is not well-suited for the category $D_c(Y)$ (and the frequently used replacement, the category $D_Z(Y)$ of 
 complexes supported on a proper subvariety $Z \subset Y$ does not lead to nice moduli spaces). 

We therefore propose to use a variant of the definition of numerical $K$-group: we define $\Knumc(Y)$ as the quotient of the Grothendieck group of $\cdc(Y)$ by the radical of the Euler pairing with \emph{perfect complexes} on $Y$. We prove in Lemma \ref{lem:finiterank} that this group has finite rank under very mild assumptions. Accordingly, a numerical Bridgeland stability condition for compact support on $Y$ is a pair $\sigma = (Z_\sigma, \cP_\sigma)$, where $Z_\sigma\colon \Knumc(Y)\to \CC$ is a group homomorphism and $\mathcal{P}_\sigma$ is a slicing of $\cdc(Y)$.

Let $S$ be a separated scheme of finite type over $\kk$; we do not assume that $S$ is proper. We write $N^1(S)$ for the group of Cartier divisor up to numerical equivalence with respect to proper curves in $S$, and $N_1(S)$ for the dual group of curve classes. For $\vv\in \Knumc(Y)$ and $\sigma\in \Stab(\cdc(Y))$, let $\cE\in D(Y\times S)$ be a family of $\sigma$-semistable objects of class $\vv$ over $S$. This means in particular that the derived restriction of $\cE$ to the fibre in $Y\times S$ over each closed point $s\in S$ is a $\sigma$-semistable object that has numerical class $\vv$ (see Section~\ref{sec:linearisation}). 

 Our main result generalises \cite[Theorem~1.1]{BM12}:
 
\begin{theorem}
\label{thm:positivity}
 Let $Y$ be a normal, quasi-projective scheme of finite type over an algebraically closed field $\kk$ of characteristic zero, and let $\sigma$ be a numerical Bridgeland stability condition for compact support on $Y$. For any family $\cE\in D(Y\times S)$ of $\sigma$-semistable objects of class $\vv$ whose support is proper over $S$, we obtain a nef numerical Cartier divisor class $\ell_{\cE,\sigma}\in N^1(S)=\Hom(N_1(S),\RR)$, defined dually by setting
 \begin{equation}
 \label{eqn:ellCintro}
 \ell_{\cE,\sigma}\big([C]\big) = \ell_{\cE,\sigma}\cdot C := \Im\left(\frac{Z_{\sigma}(\Phi_{\cE}(\cO_C)\big)}{-Z_\sigma(\vv)}\right)
 \end{equation}
 for every proper curve $C\subseteq S$, where $\Phi_{\cE}\colon D(S)\to D(Y)$ is the integral functor with kernel $\cE$. Moreover, $\ell_{\cE,\sigma}\cdot C>0$ if and only if for two general closed points $c, c^\prime\in C$, the corresponding objects $\cE_c, \cE_{c'}\in \cdc(Y)$ are not $S$-equivalent. 
\end{theorem}
 
More generally, Theorem~\ref{thm:positivity} holds for any scheme $Y$ that is separated and of finite type over an algebraically closed field such that $\Knumc(Y)$ has finite rank (see Theorem~\ref{thm:linearisation}, or Remark \ref{remark:finiterank} for an alternative assumption). We also provide a geometric condition which ensures that the family $\cE$ has proper support over $S$ (see Proposition~\ref{prop:simpleimpliesproper}): it suffices to assume that $Y$ is proper over an affine scheme and the fibre $\cE_s$ of the family over every closed point $s\in S$ is simple, in the sense that $\Hom(\cE_s,\cE_s)=\kk$.

\begin{corollary} 
\label{cor:main} 
Assume that there exists a fine moduli space $\cM_\sigma(\vv)$ of $\sigma$-stable 
objects of class $\vv$. Then $\cM_\sigma(\vv)$ comes equipped with a family of numerically positive divisors.
\end{corollary}

Note that the moduli space $\cM_\sigma(\vv)$ will not be proper in general.

\subsection{Families of t-structures for compact support} The proof of Theorem~\ref{thm:positivity} relies on extending the work of Abramovich--Polishchuk \cite{AP06} and Polishchuk~\cite{P07}
to the setting of objects with compact support. More precisely, given separated schemes $S$ and $Y$ of finite type, we define what it means for an object of $D(Y\times S)$ to have \emph{left-compact support}, see Definition \ref{def:compact}. This rather ad-hoc definition is more restrictive than requiring an object to have proper support over $S$, but it is better behaved under derived restriction along an open immersion (see Proposition~\ref{prop:localisation} and Remark~\ref{rem:notesssurj}). Given a t-structure on the category $D_c(Y)$ of objects with compact support, we construct a constant family of t-structures in the derived category of objects on $Y\times S$ with left-compact support (Theorem \ref{thm:t-str-any}) and show that it satisfies the open heart property (Proposition \ref{prop:open-heart}). We follow the approach of \cite{AP06,P07} very closely, but for completeness we provide full proofs of most statements. 

 Returning to the proof of Theorem~\ref{thm:positivity}, the restriction of the family $\cE$ of semistable objects to $Y\times C$ has left-compact support for any proper curve $C$ in $S$. The positivity statements from Theorem~\ref{thm:positivity} now follow as in \cite{BM12}, where a key step invokes the open heart property for the newly constructed t-structure for objects on $Y\times C$ with left-compact support. 

\subsection{Comparison with $\theta$-stability}
 One situation where moduli spaces $\cM_\sigma(\vv)$ are known to exist is when $Y$ is a smooth scheme that is projective over an affine scheme, and that carries a tilting bundle $E$. Under an assumption on the endomorphism algebra $A$ of $E^\vee$ (see Lemma~\ref{lem:perfectpair}), we obtain stability conditions on $\Dfin(A)$ of the form $\sigma_{\theta,\lambda,\xi}$, where $\theta$ is a stability parameter for $A$-modules in the sense of King~\cite{King94}, and where $\lambda, \xi$ are parameters (see Lemma~\ref{lem:stab-cond}). 
 
  In this setting, we obtain stability conditions on $\cdc(Y)$ from those on $\Dfin(A)$ by way of the tilting equivalence, and for any such $\sigma:= \sigma_{\theta,\lambda,\xi}$ and any class $\vv\in \Knumc(Y)$, the coarse moduli space of $\sigma$-semistable objects in $\cdc(Y)$ of class $\vv$ coincides with the coarse moduli space $\overline{\cM_A}(\theta,\vv)$ of $\theta$-semistable $A$-modules of dimension vector $\vv$ that is constructed by GIT. It is then natural to compare the numerical line bundle from Theorem~\ref{thm:positivity} with the polarising ample line bundle on the moduli space given by the GIT construction. The following result is Theorem~\ref{thm:polarisingbundle} (compare Proposition~\ref{prop:simpleimpliesproper}) in the special case when $\xi\in \RR$ is chosen to satisfy $\lambda(\vv)=1/(\xi^2+1)$.
  
 \begin{theorem}
 \label{thm:pullbackample} 
 Let $S$ be a separated scheme of finite type, and suppose that a family $\cE \in D(Y \times S)$ of $\sigma_{\theta,\lambda,\xi}$-semistable objects of class $\vv$ has proper support over $S$. Then the numerical divisor class $\ell_{\cE}(\sigma_{\theta,\lambda,\xi})$ on $S$ is equal to the pullback of the polarising ample line bundle on $\overline{\cM_A}(\theta,\vv)$ along the classifying morphism $f\colon S \to \overline{\cM_A}(\theta,\vv)$.
 
 \end{theorem}
Note that when $\vv$ is primitive and $\theta$ generic, then $\overline{\cM_A}(\vv, \theta)$ is actually a fine moduli space.
 
 An important ingredient in the proof of Theorem~\ref{thm:pullbackample} is the correspondence between flat families $\cE\in D(Y\times S)$ of $\sigma_{\theta,\lambda,\xi}$-(semi)stable objects of class $\vv$ with respect to the heart $\cA$, and flat families $F$ of $\theta$-(semi)stable $A$-modules of dimension vector $\vv$ over $S$ (see Proposition~\ref{prop:flatfamAmods}). In particular, when $\vv$ is primitive and $\theta$ is generic,  the universal family of $\sigma_{\theta,\lambda,\xi}$-stable objects of class $\vv$ over $\cM_A(\vv, \theta)$ is given explicitly by $\cE= E\otimes_A T$, where $T$ is the universal bundle on $\cM_A(\vv, \theta)$; and conversely, the universal bundle satisfies $T=\Psi_{\cE}(E^\vee)$, where $\Psi_{\cE}$ is an integral functor with kernel $\cE$ (see Proposition~\ref{prop:functorpair}).

\subsection{Additional background and outlook} \label{subsec:moremotivation}
In the projective setting, the link between wall-crossing for stability conditions and birational geometry of moduli spaces has led to a large number of results over the last five years. This was initiated with striking examples for abelian surfaces \cite{AB13} and $\PP^2$
\cite{ArBeCo13a}, and then exploited systematically for abelian \cite{MYY, MYY2,Yoshioka:positivecone} and K3 surfaces (see \cite{BM13} in the smooth case, and \cite{MZ14} for singular O'Grady-type moduli spaces; see also \cite{Hassett-Tschinkel:survey} for a survey and more applications), for Enriques surfaces \cite{Nuer}, for $\PP^2$ \cite{Matthew:torsion, ChoiChung, Izzet-Jack-Matthew,  Izzet-Jack:P2,Izzet-Jack:interpolation,Izzet-Jack:ample,Chunyi-Xiaolei:MMP,Aaron-Cristian-Jie}, with the story now essentially completed in \cite{Chunyi-Xiaolei:birational}, and for other rational surfaces \cite{BC13}; it has also led to results for general surfaces \cite{TooManyAuthors,Izzet-Jack:generalsurface}. In many cases, there is a complete description of the movable cones of the moduli spaces, along with its chamber decomposition coming from associated minimal models.

On the other hand, a number of authors have studied stability conditions on quasi-projective (`local') Calabi-Yau varieties $Y$, see \cite{Richard:An, Bridgeland09, Brav-HThomas:ADE, IUU}  and \cite{Toda08,Toda:CY-fibrations,Bridgeland:stab-CY, BM11} for crepant resolutions of two- and three-dimensional canonical singularities, respectively, and
\cite{AnnoBezruMirko} for higher-dimensional symplectic resolutions of singularities naturally associated to  algebraic groups. Our goal is to provide in this context the machinery that is used in the projective setting.

 Even in the case where $Y$ is a projective crepant resolution of $\CC^3/G$ for a finite abelian subgroup $G \subset \mathrm{SL}_3(\CC)$, a rich wall-crossing picture emerges by considering $Y$ itself as a moduli space parametrising skyscraper sheaves of points. Indeed, a simple reinterpretation of \cite{CI04} (along the lines of our Section \ref{sec:tilting}) says that any (projective) birational model of $Y$ appears as a moduli space of Bridgeland-stable objects; more generally, this result holds for any projective crepant resolution of a Gorenstein, affine toric 3-fold by \cite{IshiiUeda13}. We anticipate that this result can be generalised significantly, both by allowing for more general $Y$, and by considering different moduli spaces on $Y$. We also hope that it will simplify the study of the space of stability conditions itself: typically, one of the crucial steps is the systematic understanding of walls of the geometric chamber in $\Stab(\cdc(Y))$, where skyscraper sheaves of points are all semistable, some of them being strictly semistable. Our results provide a nef line bundle on $Y$, whose associated contraction should govern the wall-crossing behaviour to a large extent.

\subsection{Running assumptions and notation}
All our schemes are assumed to be Noetherian schemes over an algebraically closed field $\kk$. In addition, we assume from Section~\ref{sec:integralfunctors} onwards that our schemes are separated and of finite type over $\kk$. Given a product of schemes $Y\times S$, we write $p\colon Y\times S\to Y$ and $q\colon Y\times S\to S$ for the projections to the first and second factor.

For any scheme $X$, let $\cD(X)$ denote the bounded derived category of coherent sheaves on $X$, let $\Dperf(X)$ denote the full subcategory of perfect complexes on $X$, and write $\DD(\Qcoh(X))$ for the unbounded derived category of quasi-coherent sheaves on $X$. To avoid a proliferation of $\mathbf{R}$ and $\mathbf{L}$, we omit these symbols in our derived functors except in writing $\mathbf{R}\Gamma$.

\subsection*{Acknowledgements} We thank the anonymous referee for a very careful reading of the paper and for suggesting numerous improvements. We are indebted to Joseph Karmazyn who found a counterexample to a false claim in a previous version. The first author is grateful to Matthew Woolf for a number of helpful conversations, as well as to Emanuele Macr{\`\i} for very many of them. The second author thanks Alastair King for a helpful conversation. The first author was supported by ERC starting grant WallXBirGeom 337039, while the second and third authors were supported in part by EPSRC grant EP/J019410/1. The third author also acknowledges the support of Leibniz University Hannover during the final stage of this work.

\section{Derived category with left-compact support}
In this section, we define what it means for a complex of coherent sheaves on a product $Y\times S$ to have `left-compact support'. We also study the basic properties, and compare this notion to that of an object on $Y\times S$ having proper support over $S$. The latter part of this section follows closely the work of Abramovich--Polishchuk~\cite{AP06} and Polishchuk~\cite{P07} in defining sheaves of t-structures over a base. 

\subsection{Compact and left-compact support}
 For any Noetherian scheme $Y$ over $\kk$, we may identify $\cD(Y)$ with the full subcategory of $\DD(\Qcoh(Y))$ of bounded complexes with coherent cohomology. 

\begin{definition}
 The support of a quasi-coherent sheaf $G$ is the locus $\supp(G)=\{y\in Y \mid G_y\neq 0\}$ of points with non-zero stalk. The support of an object $F\in\DD(\Qcoh(Y))$ is the union of the supports of its cohomology sheaves.
 \end{definition}
Since localisation is exact, we could equivalently define 
\begin{equation} \label{eq:alternatesupportqc}
\supp(F) = \{y \in Y \mid F_y \neq 0\}
\end{equation}
 where $F_y$ is the complex of stalks of $F$ at the local ring at $y$. In addition, if $F \in D(Y)$, then by Nakayama's Lemma
 \begin{equation} \label{eq:alternatsupportcoh}
 \supp(F) = \{ y \in Y \mid i_y^* F \neq 0\},
 \end{equation}
 where $i_y$ is the inclusion of the spectrum $\Spec\ \kk(y)$ of the residue field of $y$. Also note that for $F \in D(Y)$, the support $\supp(F)$ is closed. Write $\cdc(Y) \subset D(Y)$ for the full subcategory of objects that have proper support.  Following convention, we refer to such objects as having `compact support'. We note the following easy properties of support.
 \begin{lemma}\label{lem:easycompact}
 Let $F, E \in D(Y)$, and let $f \colon Y \to Y'$ and $g \colon Y'' \to Y$ be morphisms between Noetherian schemes. Then:
 \begin{enumerate}
 \item[\one] $\supp (F \otimes E) \subset \supp (F)$;
 \item[\two] $\supp (g^* F) = g^{-1}(\supp(F))$; 
 \item[\three] $\supp(f_*F) \subseteq \overline{f(\supp(F))}$. 
 \end{enumerate}
 \end{lemma}
 \begin{proof}
 The first part is immediate from \eqref{eq:alternatesupportqc}, the second from \eqref{eq:alternatsupportcoh}. For the third part, assume $y \notin \overline{f(\supp(F))}$; 
then there is an open neighbourhood $y \in U \subset Y'$  with $F|_{f^{-1}(U)} = 0$, and the claim follows from flat base change.
\end{proof}

\begin{definition}
\label{def:compact}
Let $Y$ and $S$ be Noetherian schemes. An object $F \in \cD(Y \times S)$ is said to have \emph{left-compact support} if $\supp(F) \subseteq Z\times S$ for some proper subscheme $Z \subseteq Y$. Write $\cdlc(Y \times S)$ for the full subcategory of complexes in $\cD(Y \times S)$ that have left-compact support.
\end{definition}

\begin{lemma}
\label{lem:cpt-supp-new}
 Let $Y$ be a separated scheme of finite type and let $S$ be a proper scheme. Then $\cdc(Y\times S) = \cdlc(Y\times S)$.
\end{lemma}
\begin{proof}
 One inclusion is obvious. For the other inclusion, let $F\in D_c(Y\times S)$. It follows from \cite[Tag~03GN]{stacks-project} that $Z:=p(\supp(F))$ is proper and hence $\supp(F) \subseteq Z \times S$. 
\end{proof}

By Lemma \ref{lem:easycompact}, having left-compact support is preserved under some standard operations:

\begin{lemma}
\label{lem:cpt-supp}
Let $F\in \cdlc(Y \times S)$. Then:
\begin{enumerate}
\item[\one] for any object $E\in \cD(Y\times S)$, we have $F\otimes E\in\cdlc(Y\times S);$
\item[\two] for any morphism $g\colon S'\ra S$, we have $(\id_Y \times g)^*F \in \cdlc(Y\times S');$
\item[\three] for any proper morphism $g\colon S\ra S'$, we have $(\id_Y \times g)_*F \in \cdlc(Y\times S')$.
\end{enumerate}
\end{lemma}

\begin{corollary} 
\label{cor:cpt-supp}
Let $Y$ be a separated scheme of finite type and let $S$ be proper. Pullback and pushforward  along the projection to the first factor $p\colon Y\times S \to Y$ induce exact functors $p^*\colon \cdc(Y)\to \cdc(Y\times S)$ and $p_*\colon \cdc(Y\times S)\to \cdc(Y)$.
\end{corollary}
\begin{proof}
 This follows from Lemma~\ref{lem:cpt-supp-new} and Lemma~\ref{lem:cpt-supp}\two-\three. 
\end{proof}

\begin{lemma}
\label{lem:Flcsupp1}
 Let $F\in \DD(\Qcoh(Y \times S))$ and let $g\colon S\ra S'$ be an affine morphism of Noetherian schemes. If we have $(\id_Y \times g)_*F\in \cdlc(Y\times S^\prime)$, then $F\in \cdlc(Y \times S)$.
 \end{lemma}
 \begin{proof}
Since pushforward along $f:= \id_Y \times g$ is exact, we obtain that $F$ is a bounded complex with coherent
cohomology; see \cite[Tag 08I8]{stacks-project}.
To show $F$ has left-compact support, it suffices to show that $\supp(F) \subseteq f^{-1}(\supp( f_*F))$. The complement $U$ of $\supp (f_*F)$ in $Y \times S'$ is an open subscheme such that $(f_*F)|_U = 0$. Since $f$ is an affine morphism, we have  $F|_{f^{-1}(U)}=0$ by \cite[Tag 08I8]{stacks-project}, and the claim follows.
 \end{proof}

 \begin{lemma}
 \label{lem:Flcsupp2}
 Let $F\in D(Y \times S)$. If $S = \bigcup_i U_i$ is a finite open cover such that $F|_{Y \times U_i}$ has left-compact support for each $i$, then $F\in \cdlc(Y\times S)$.
\end{lemma}
\begin{proof}
The cover of $S$ is finite, so we can find a proper $Z \subseteq Y$ such that $\supp (F|_{Y \times U_i}) \subseteq Z \times U_i$ for all $i$.
Lemma \ref{lem:easycompact}\two{} completes the proof.
\end{proof}

\subsection{Localisation}
We now study how $\cdlc(Y\times S)$ behaves under restriction to certain open subsets in $Y\times S$. These results extend observations of Polishchuk~\cite[Lemma 2.3.1]{P07} (see also Arinkin--Bezrukavnikov~\cite[\S2.2]{AB2010}) to the left-compactly supported case.

\begin{lemma}[{\cite[Lemma 2.3.1]{P07}}]
Let $X$ be a Noetherian scheme, $j \colon U \to X$ be an open immersion, and let $F$ be a bounded complex of coherent sheaves
on $U$. Then there exists a complex $F_c$ on $X$ consisting of coherent subsheaves $F^i_c \subset j_* F^i$ such that
$j^* F_c = F$ as objects in $D(U)$, and such that $H^i\left(F_c\right) \subset H^i\left(j_* F\right)$ for all $i \in \ZZ$.
\end{lemma}
\begin{proof}
All statements follow directly from the proof of \cite[Lemma 2.3.1]{P07}. The only claim that is not made explicitly in [ibid.] is the inclusion
of the cohomology sheaves; this follows from property (i) stated in the proof and a simple diagram chase.
\end{proof}

\begin{corollary}
\label{cor:essential-surjective}
Let $S, Y$ be Noetherian schemes.
For any open subset $U\subseteq S$, let $j\colon Y\times U \longhookrightarrow Y\times S$ be the open immersion.  The restriction functor $j^*\colon\cdlc(Y\times S) \ra \cdlc(Y\times U)$ is essentially surjective. 
\end{corollary}

\begin{proof}
Given $F \in \cdlc(Y \times U)$, assume that it is represented by a bounded complex of coherent sheaves. Let $F_c \subset j_*F$ be the subcomplex with $F = j^*F_c$ by the previous Lemma. By the assertions in the Lemma above,
we obtain
\[ \supp(F_c) = \bigcup_i \supp\left(H^i(F_c)\right) \subseteq \bigcup_i \supp\left(H^i j_* F\right) = \supp(j_* F) \subseteq \overline{j(\supp(F))}
\]
 where the last inclusion is a case of Lemma \ref{lem:easycompact}\three. 
By assumption, $\supp(F) \subseteq Z \times U$ for some proper subscheme $Z \subseteq Y$; hence
$\supp(F_c) \subset Z \times S$, which implies the claim.
\end{proof}

\begin{proposition}
\label{prop:localisation}
Let $S, Y$ be Noetherian schemes. Let $U\subseteq S$ be open and set $T:=S\setminus U$. The triangulated category $\cdlc(Y\times U)$ is equivalent to the localisation of $\cdlc(Y\times S)$ with respect to the localising class of morphisms $f\colon F \ra F'$ whose cone has support contained in $Y\times T$.
\end{proposition}
\begin{proof}
 Following Rouquier~\cite[Remark~3.14]{Rouquier2010}, the pullback $j^*\colon D(Y\times S)\to D(Y\times U)$ is essentially surjective and has kernel given by the subcategory of complexes supported in $Y\times T$. It follows that $\cD(Y\times U)$ is equivalent as a triangulated category to the quotient of $\cD(Y\times S)$ by the subcategory of complexes whose support is contained in $Y\times T$. Now restrict $j^*$ to $\cdlc(Y\times S)$ and apply Corollary~\ref{cor:essential-surjective} to see that $j^*\colon\cdlc(Y\times S) \ra \cdlc(Y\times U)$ is essentially surjective and has kernel given by the subcategory of complexes (with left-compact support) whose support is contained in $Y\times T$. It follows that $\cdlc(Y\times U)$ is equivalent to the quotient of $\cdlc(Y\times S)$ by the subcategory of complexes supported in $Y\times T$.
\end{proof}

\subsection{Objects with proper support over the base}
 Let $S,Y$ be Noetherian schemes. Recall that $p\colon Y\times S\to Y$ and $q\colon Y\times S\to S$ are the first and second projection morphisms. 

 \begin{definition}
 An object $F\in D(Y\times S)$ is said to have \emph{proper support over $S$} if the morphism $q|_{\supp(F)}\colon \supp(F) \to S$ is proper with respect to the reduced scheme structure on the closed subset $\supp(F)\subseteq Y\times S$. 
\end{definition}

\begin{lemma}
\label{lem:lcvsproper}
Let $S,Y$ be separated schemes of finite type. Let $F\in D(Y\times S)$. If $F$ has left-compact support, then $F$ has proper support over $S$. If $S$ is proper, the converse also holds.
\end{lemma}
\begin{proof}
The first statement follows because $\supp(F)$ is a closed subset of $Z \times S$ for some proper subscheme $Z\subseteq Y$. For the second statement, the composition of the proper morphisms $q|_{\supp(F)}$ and $S\to \spec \kk$ is proper, so $\supp(F)$ is proper. It follows from \cite[Tag~03GN]{stacks-project} that $Z:=p(\supp(F))$ is a proper subscheme of $Y$, so $\supp(F) \subseteq Z \times S$ as required.
\end{proof}

\begin{example}
For $Y=\mathbb{A}^1$, the object $\cO_\Delta\in D(Y\times \mathbb{A}^1)$ has proper support over $\mathbb{A}^1$, but it does not have left-compact support. 
\end{example}

\begin{remark}
\label{rem:notesssurj}
The previous example can be used to show that the analogue of the important Corollary~\ref{cor:essential-surjective} is false if we replace the notion of left-compact support with that of proper support over $S$. Indeed, \cite[Lemma~2.3.1]{P07} implies that for an open immersion  $j\colon \mathbb{A}^1\hookrightarrow \mathbb{P}^1$, the restriction functor $(\id_Y\times j)^*\colon D(Y\times \mathbb{P}^1)\to D(Y\times \mathbb{A}^1)$ is essentially surjective. However, $\cO_\Delta\in D(Y\times \mathbb{A}^1)$ is not quasi-isomorphic to the restriction of an object $F\in D(Y\times \mathbb{P}^1)$ that has proper support over $\mathbb{P}^1$, because otherwise Lemma~\ref{lem:cpt-supp} and Lemma~\ref{lem:lcvsproper} would imply that $\cO_\Delta \cong (\id_Y\times j)^*F$ has left-compact support which is absurd.
\end{remark}

\subsection{Integral functors}
\label{sec:integralfunctors}
It is well known that when $S$ and $Y$ are smooth projective varieties, an object $E\in D(Y\times S)$ is the kernel for a pair of integral functors, sometimes  denoted
\[
\Phi_{E}^{S\to Y}\colon D(S)\to D(Y)
\quad\text{and}\quad 
\Phi_E^{Y\to S}\colon D(Y)\to D(S).
\]
We present here the natural extension of this statement to complexes with compact support. 

\begin{lemma} 
\label{lem:pushforwardcompact}
 Let $f\colon Y\to Y^\prime$ be a morphism of separated schemes of finite type and let $F\in D(Y)$. 
 \begin{enumerate}
 \item[\one] If $\supp(F)$ is proper over $Y'$, then $f_*F\in D(Y')$. 
 \item[\two] If $\supp(F)$ is proper, then $f_*F\in D_c(Y')$, that is, we have  $f_* \colon D_c(Y) \to D_c(Y')$.
 \end{enumerate}
\end{lemma}
\begin{proof}
 Part \one\ follows from \cite[Tag 08E0]{stacks-project}. For part \two, since $Y'$ is separated over $\kk$ and $F$ has proper support over $\kk$, \cite[Tag 01W6]{stacks-project} implies that $F$ has proper support over $Y'$, so $f_* F \in D(Y')$ by part \one. To see that $f_*F$ has proper support, it is enough by Lemma~\ref{lem:easycompact}\three\ to show that $f(\supp(F))$ is proper. This follows from \cite[Tag~03GN]{stacks-project}.
\end{proof}

An object $E\in D(Y\times S)$ is \emph{$S$-perfect} if, locally over $S$, it is quasi-isomorphic to a bounded complex of $q^{-1}(\cO_S)$-flat coherent sheaves. In this context, the definition of $S$-perfect introduced by Illusie~\cite{SGA6} does not state explicitly that the $q^{-1}(\cO_S)$-flat sheaves must be coherent, but as Lieblich~\cite[Example~2.1.2]{Lie06} remarks, these notions of $S$-perfect are nevertheless equivalent.

\begin{proposition}
\label{prop:functorpair}
Let $S, Y$ be separated schemes of finite type, and let $E\in D(Y\times S)$. If $E$ has proper support over $S$, then it provides an integral functor
\[
\Phi_{E} \colon D_c(S)\longrightarrow D_c(Y), \, F\mapsto p_*(E\otimes q^*F).
\]
 If in addition $E$ is $S$-perfect, then we obtain a second integral functor 
 \[
\Psi_{E}\colon \Dperf(Y)\longrightarrow \Dperf(S), \, F\mapsto q_*(E\otimes p^*F).
\]
\end{proposition}
\begin{proof}
For the first claim, given $F \in D_c(S)$, we need to show that 
$p_*\left(E \otimes q^*F\right) \in D_c(Y)$. Since the projection $q$ is flat, we have $q^*F\in D(Y \times S)$. The support of $E \otimes q^*F \in D(Y \times S)$ is closed and contained in \[ q^{-1}(\supp(F)) \cap \supp(E) \]
which is proper. Therefore, $E\otimes q^*F \in D_c(Y \times S)$, and the claim follows from Lemma \ref{lem:pushforwardcompact}.

To construct the functor $\Psi_{E}$, let $F\in \Dperf(Y)$. Then $p^*F\in D(Y \times S)$ is perfect. By \cite[III, Proposition~4.5]{SGA6} (applied with $f = \id_{Y \times S}$ and $g = q$---note that \emph{perfect} is the same as ``of finite amplitude with respect to the identity morphism''), it follows that $E \otimes p^*F \in D(Y \times S)$ is $S$-perfect. The support of $E \otimes p^*F$ is proper over $S$,  hence
 $q_*(E \otimes p^*F) \in D(S)$ by Lemma~\ref{lem:pushforwardcompact}.
On the other hand, \cite[III, Proposition~4.8]{SGA6} shows that $q_*(E \otimes p^*F)$ is $S$-perfect, and therefore perfect.
\end{proof}

\subsection{On t-structures}
\label{sec:t-structures}
Let $\cD$ be a triangulated category. Recall that a \emph{t-structure} on $\cD$ is a pair of full subcategories $(\cD^{\leq 0}, \cD^{\geq 0})$ of $\cD$ such that if for $n\in \ZZ$ we write $D^{\leq n}:=D^{\leq 0}[-n]$ and $D^{\geq n}:=D^{\geq 0}[-n]$, then we have
\begin{enumerate}
\item $\cD^{\leq -1}\subseteq \cD^{\leq 0}$; 
\item $\Hom(F,G)=0$ for $F\in \cD^{\leq 0}$ and $G\in \cD^{\geq 1}$; and 
\item for $F\in \cD$ there exists an exact triangle $\tau^{\leq 0}F\ra F\ra\tau^{\geq 1}F$
in $D$, where $\tau^{\leq 0}F\in D^{\leq 0}$ and $\tau^{\geq 1}F\in D^{\geq 1}$.
\end{enumerate}
For $a, b\in \ZZ$ with $a\leq b$, write $D^{[a,b]}:= D^{\geq a}\cap D^{\leq b}$. More generally, we  write $D^{[-\infty,b]}:= D^{\leq b}$ and $D^{[a,\infty]}:= D^{\geq a}$, and refer to the subcategory $D^{[a,b]}$ for any interval $[a,b]$ that may be infinite on one side. The \emph{heart} of the t-structure is the abelian category $D^{[0,0]}$. The inclusions $D^{\leq n}\ra D$ and $D^{\geq n}\ra D$ admit right- and left-adjoints $\tau^{\leq n}\colon D\ra D^{\leq n}$ and $\tau^{\geq n}\colon D\ra D^{\geq n}$ respectively. For $F\in D$ and $n\in \ZZ$, the corresponding \emph{truncation triangle} is the exact triangle  
\[
\tau^{\leq n}F\lra F\lra\tau^{\geq n+1}F
\]
in $D$, where $\tau^{\leq n}F\in D^{\leq n}$ and $\tau^{\geq n+1}F\in D^{\geq n+1}$. For $i\in \ZZ$, the cohomology functor $H^i\colon D\ra D^{[0,0]}$ is given by $H^i(F)=\tau^{\leq 0}\tau^{\geq 0}(F[i])$. A t-structure $(\cD^{\leq 0}, \cD^{\geq 0})$ is \emph{nondegenerate} if $\cap_n \cD^{\leq n} =\cap_n \cD^{\geq n}=\{ 0 \}$, or equivalently, if the only object $F\in D$ satisfying $H^i(F)=0$ for all $i\in \ZZ$ is the zero object. A t-structure is \emph{bounded} if $\cup_n \cD^{\leq n} =\cup_n \cD^{\geq n}=\cD$, or equivalently, if it is nondegenerate and if each $F\in D$ satisfies $H^i(F)\neq 0$ for only finitely many values of $i\in \ZZ$. For proofs of these assertions and other facts about t-structures, see Mili\v{c}i\'{c}~\cite[Chapter~4]{Milicic03}.

\begin{remark}
\label{rem:nondegenerateAP}
Both \cite{AP06} and \cite{P07} say that a t-structure is `nondegenerate' if it is bounded in the sense defined above (and hence nondegenerate in the sense defined above).
\end{remark}

Given triangulated categories $D, D^\prime$ each equipped with a t-structure, a functor $\Phi\colon D\ra D^\prime$ is \emph{left t-exact} if $\Phi(D^{\geq 0})\subseteq (D^\prime)^{\geq 0}$,  and it's \emph{right t-exact} if $\Phi(D^{\leq 0})\subseteq (D^\prime)^{\leq 0}$. The functor is \emph{t-exact} if it's both left- and right t-exact. 

\begin{lemma}
\cite[Lemma 1.1.1]{P07}
\label{lem:zero-kernel}
Let $D_1$ and $D_2$ be a pair of triangulated categories equipped with t-structures and let $\Phi\colon D_1 \ra D_2$ be a t-exact functor with $\ker \Phi=0$. Then for any interval $[a,b]$ that may be infinite on one side, we have $D_1^{[a,b]} = \{F \in D_1 \mid \Phi(F) \in D_2^{[a,b]}\}$.
\end{lemma}

\subsection{Sheaves of t-structures over the base}
 Let $S, Y$ be Noetherian schemes. We continue to write $p\colon Y\times S\to Y$ and $q\colon Y\times S\to S$ for the two projections.
 
\begin{definition}
\label{def:sheaf}
A \emph{sheaf of t-structures on $\cdlc(Y\times S)$ over $S$} is a bounded t-structure on $\cdlc(Y\times U)$ for each open subset $U\subseteq S$, such that for every open subset $j\colon U'\hookrightarrow U$, the restriction functor $(\id_Y\times j)^*\colon \cdlc(Y\times U) \ra \cdlc(Y\times U')$ is t-exact.
\end{definition}

This notion generalises that of a `t-structure on $D(Y\times S)$ that is local over $S$' \cite{P07} and that of a `sheaf of t-structures on $Y$ over $S$' when $Y$ and $S$ are projective \cite{AP06}. To justify the terminology, we extend \cite[Lemma 2.3.4]{P07} to the setting of left-compact support.

\begin{lemma} \label{lem:gluing-property}
Let $S$ be a Noetherian scheme with a finite open cover $S= \bigcup_i U_i$. Assume that we are given a sheaf of t-structures on each $\cdlc(Y\times U_i)$ over $U_i$ that agree on all pairwise intersections $Y \times (U_i \cap U_j)$. Then there exists a unique sheaf of t-structures on $\cdlc(Y \times S)$ over $S$ that restricts to the given t-structure 
on $\cdlc(Y \times U_i)$ for each $i$. Moreover, it satisfies
\begin{equation}
\label{eqn:gluing-t-str}
\cdlc^{[a,b]}(Y \times S) = \big\{ F \in D(Y \times S) \mid F|_{Y \times U_i} \in \cdlc^{[a,b]}(Y \times U_i)\text{ for all } i \big\}.
\end{equation}
\end{lemma}

\begin{proof} 
We first show that there is a unique t-structure on $\cdlc(Y \times S)$ such that the restriction functors $\cdlc(Y \times S) \to \cdlc(Y \times U_i)$ are t-exact for the given t-structures for each $U_i$.

\emph{Construction.}
 The right hand side of \eqref{eqn:gluing-t-str} is contained in $\cdlc(Y \times S)$ by Lemma~\ref{lem:Flcsupp2}. We first prove that \eqref{eqn:gluing-t-str} defines a t-structure on $\cdlc(Y \times S)$. The definition gives $\cdlc^{\leq -1}(Y\times S)\subseteq \cdlc^{\leq 0}(Y\times S)$. To show that $F \in \cdlc^{\leq 0}(Y \times S)$ and $G \in \cdlc^{\geq 1}(Y \times S)$ satisfy $\Hom(F,G)=0$, choose a finite open affine cover $U_i = \bigcup_j U_{ij}$ to obtain a finite open affine cover $S= \bigcup_{ij}U_{ij}$. The t-structure on $\cdlc(Y \times U_{ij})$ induced from that on $\cdlc(Y \times U_i)$ satisfies 
 \[
 F|_{Y \times U_{ij}} \in \cdlc^{\leq 0}(Y \times U_{ij})\quad \text{ and }\quad G|_{Y \times U_{ij}} \in \cdlc^{\geq 1}(Y \times U_{ij}),
 \]
 so $\Hom^{\leq 0}(F|_{Y \times U_{ij}}, G|_{Y \times U_{ij}})=0$. For $q\colon Y\times S\to S$, consider $\shom_S(F,G):=q_*\shom(F,G) \in \DD(\Qcoh(S))$. The complex of sheaves $\shom_S(F,G)|_{U_{ij}} = \shom_{U_{ij}}(F|_{Y \times U_{ij}}, G|_{Y \times U_{ij}})$ is obtained from the complex of $H^0(\cO_{U_{ij}})$-modules $\Hom(F|_{Y \times U_{ij}}, G|_{Y \times U_{ij}})$ by sheafification, so with respect to the standard t-structures we have that $\shom_S(F,G)|_{U_{ij}} \in \DD^{\geq 1}(\Qcoh(U_{ij}))$ and hence $\shom_S(F,G) \in \DD^{\geq 1}(\Qcoh(S))$.
Since $\Hom^\bullet(F, G) = \Gamma(\shom(F, G)) = \Gamma(S, \shom_S(F, G))$, and since $\Gamma$ is left exact, it follows that $\Hom^0(F, G) = 0$. It remains to define the truncation functors. By boundedness of the t-structure on each $\cdlc(Y \times U_{i})$ and an induction argument, we need only prove that for any $F \in \cdlc^{\geq 0}(Y \times S)$, the left truncation $H^0(F)$ exists, as does a morphism $H^0(F) \ra F$ whose cone lies in $\cdlc^{\geq 1}(Y \times S)$. For this, we have $F|_{Y \times U_i} \in\cdlc^{\geq 0}(Y \times U_i)$ for each $i$, so the left truncation $H^0(F|_{Y \times U_i})$ exists with a morphism $H^0(F|_{Y \times U_i}) \lra F|_{Y \times U_i}$ whose cone lies in $\cdlc^{\geq 1}(Y \times U_i)$. By \cite[Theorem 2.1.9, Corollary 2.1.11]{AP06}, the objects $H^0(F|_{Y \times U_i}) \in \cdlc(Y \times U_i)$ glue to give an object in $D(Y \times S)$ which we define to be $H^0(F)$. For every $i\in \ZZ$, the object $H^0(F)|_{Y \times U_i} \cong H^0(F|_{Y \times U_i})$ has left-compact support, and hence so does $H^0(F)$ by Lemma~\ref{lem:Flcsupp2}. We may also glue the morphisms $H^0(F|_{Y \times U_i}) \ra F|_{Y \times U_i}$ into a global morphism $H^0(F) \ra F$ by \cite[Lemma 2.1.10]{AP06}. Since the cone of this global morphism restricts to the cone of each local morphism, which lies in $\cdlc^{\geq 1}(Y \times U_i)$, it follows from \eqref{eqn:gluing-t-str} that the cone of the global morphism lies in $D_c^{\geq 1}(Y \times S)$ as required. Thus  \eqref{eqn:gluing-t-str} defines a t-structure.

We now show that \eqref{eqn:gluing-t-str} induces the given t-structure on each $\cdlc(Y \times U_i)$. The restriction from $Y \times S$ to $Y \times U_i$ is t-exact, so we need only show for any $i$ and any $F_i \in \cdlc^{[a,b]}(Y \times U_i)$ that there exists $F \in \cdlc^{[a,b]}(Y \times S)$ such that $F|_{Y \times U_i} \cong F_i$. By Corollary~\ref{cor:essential-surjective}, there exists $G \in \cdlc(Y \times S)$ such that $G|_{Y \times U_i} \cong F_i$. We take the truncation $F:=\tau^{\geq a}\tau^{\leq b} G \in \cdlc^{[a,b]}(Y \times S)$. Both $\tau^{\leq a-1}G$ and $\tau^{\geq b+1}G$ restrict to trivial objects on $\cdlc(Y \times U_i)$, so $F|_{Y \times U_i} \cong F_i$ as required. 

\emph{Uniqueness.}
Next, we show that the t-structure \eqref{eqn:gluing-t-str} is the unique t-structure on $\cdlc(Y \times S)$ over $S$ which induces the given t-structures on $\cdlc(Y \times U_i)$. Let $\tilde{D}_{lc}^{[a,b]}(Y \times S)$ be another such t-structure. Any $F \in \tilde{D}_{lc}^{[a,b]}(Y \times S)$ satisfies $F|_{Y \times U_i} \in \cdlc^{[a,b]}(Y \times U_i)$ for all $i$, so $F \in \cdlc^{[a,b]}(Y \times S)$. On the other hand, if we truncate any $G \in \cdlc^{[a,b]}(Y \times S)$ with respect to the second t-structure, then $(\tilde{\tau}^{\leq a-1}G)|_{Y \times U_i} = \tau^{\leq a-1}(G|_{Y \times U_i}) = 0$ for each $i$. The uniqueness of gluing from \cite[Corollary 2.1.11]{AP06} implies that $\tilde{\tau}^{\leq a-1}G = 0$, and similarly we have $\tilde{\tau}^{\geq b+1}G = 0$. It follows that $G \in \tilde{D}_{lc}^{[a,b]}(Y \times S)$ and hence $D_{lc}^{[a,b]}(Y \times S) = \tilde{D}_{lc}^{[a,b]}(Y \times S)$ as required.

\emph{Sheafify.} It remains to show that our given t-structure on $\cdlc(Y \times S)$ extends uniquely to a sheaf of t-structures over $S$. To construct the associated t-structure over $U \subset S$, replace $S$ by $U$ and $U_i$ by $U \cap U_i$ in the construction and proof of uniqueness above. One easily verifies the sheaf property by applying \eqref{eqn:gluing-t-str} for $S = \bigcup_i U_i$ and analogously for $U = \bigcup_i U \cap U_i$ simultaneously, along with the sheaf property for the
given t-structures on each $U_i$.
\end{proof}

\begin{remark}
\label{rmk:degree-preserving}
The following rephrasing of the uniqueness result in the above theorem is useful in practice: Given a sheaf of t-structures on $\cdlc(Y \times S)$ over $S$ and an object $F \in \cdlc(Y \times S)$ satisfying $F|_{Y \times U_i}\in\cdlc^{[a,b]}(Y\times U_i)$ for each $i$, then $F\in \cdlc^{[a,b]}(Y\times S)$. 
\end{remark}

\begin{lemma}
\label{lem:tensorlinebundle}
For any schemes $Y$ and $S$, suppose that $\cdlc(Y\times S)$ has a sheaf of t-structures over $S$. For $L\in \Pic(S)$, the functor $\cdlc(Y \times S) \lra \cdlc(Y \times S)$ sending $F$ to $F\otimes q^*L$ is t-exact.
\end{lemma}
\begin{proof} 
 The functor is well-defined by Lemma~\ref{lem:cpt-supp}. Let $S= \bigcup_i U_i$ be an open cover such that $L\vert_{U_i}\cong \mathcal{O}_{U_i}$ for each $i$. Then for $F\in \cdlc^{[a,b]}(Y\times S)$, we have that $(F\otimes q^*L)\vert_{Y\times U_i} \cong F\vert_{Y\times U_i}$ lies in $\cdlc^{[a,b]}(Y\times U_i)$ for each $i$, because restriction to an open subset for a sheaf of t-structures is t-exact. The result follows from Remark~\ref{rmk:degree-preserving}.
\end{proof}

The main result of this section provides a partial converse (see \cite[Theorem~2.1.4]{AP06}):

\begin{theorem}
\label{thm:sheaf-criterion}
Let $S$ be a quasi-projective scheme, let $L$ be an ample line bundle on $S$, and let $(\cdlcn(Y \times S), \cdlcp(Y \times S))$ be a bounded t-structure on $\cdlc(Y\times S)$ such that the functor 
\[
\cdlc(Y \times S) \lra \cdlc(Y \times S), \quad F \longmapsto F\otimes q^*L
\]
is t-exact. Then we obtain by restriction a sheaf of t-structures on $\cdlc(Y\times S)$ over $S$.
\end{theorem}

Towards this goal, let $T \subseteq S$ be a closed subset. Consider the subcategory
\[
\cdlc(Y \times S)_T = \left\{E \in \cdlc(Y \times S)\mid \supp E \subseteq Y \times T \right\} = D(Y \times S)_T \cap \cdlc(Y \times S)
\]
of objects with left-compact support whose support lies over $T$. Our proof follows closely that of \cite[Theorem~2.1.4]{AP06}, beginning with two results on the category $\cdlc(Y\times S)_T$. First we recall the following Lemma (the proof of which does not require $Y$ to be smooth or projective):

\begin{lemma}[{\cite[Lemma 2.1.5]{AP06}}] \label{lem:supportviazeroemaps}
Let $f_1, \dots, f_n$ be sections of some line bundle $L$ on $S$ such that $T$ is the set of common zeroes of $f_1, \dots, f_n$, and let $F \in D(Y \times S)$. Then $F \in D(Y \times S)_T$ if and only if there exists $d > 0$ such that the morphisms
$f_{i_1}\cdot \dots \cdot f_{i_d} \colon F \to F \otimes q^*L^d$ are zero for all sequences $(i_1, \dots, i_d)$ of length $d$.
\end{lemma}

For any $i\in \ZZ$, we write $H^i_t$ for the $i$-th cohomology functor with respect to the t-structure on $\cdlc(Y\times S)$ listed as an assumption in Theorem~\ref{thm:sheaf-criterion}.

\begin{lemma}
\label{lem:coh-support}
Under the assumptions of Theorem~\ref{thm:sheaf-criterion}, let $T\subseteq S$ be a closed subset and let $F\in D(Y\times S)$. Then $F\in \cdlc(Y\times S)_T$ if and only if $H_t^i(F)\in \cdlc(Y\times S)_T$ for all $i\in \ZZ$.
\end{lemma}
\begin{proof}
Suppose first that $H_t^i(F)\in \cdlc(Y\times S)_T$ for all $i\in \ZZ$. Let $Z_i\subseteq Y$ be a proper subscheme satisfying $\supp H^i_t(F) \subseteq Z_i\times T$, where $Z_i$ is empty if $H_t^i(F)=0$. The given t-structure is bounded, so  $F$ is a finite extension of only finitely many cohomology sheaves $H_t^i(F)$.  Therefore $\supp(F)\subseteq (\bigcup_{i} Z_i)\times T$ where $\bigcup_{i} Z_i$ is a proper subscheme of $Y$, giving $F\in \cdlc(Y\times S)_T$.

For the opposite implication, let $L$ denote an ample line bundle on $S$. Replacing $L$ by a suitable power, we may assume that $T$ is the common zero-locus of sections $f_1, \cdots, f_n$ of $L$. We apply Lemma \ref{lem:supportviazeroemaps} to obtain $d$ for which 
$f=f_{i_1}\cdots f_{i_d}\colon F \lra F\otimes q^*L^d$ is the zero map for all such sequences. By assumption, tensoring with $q^*L$ is t-exact for the t-structure in question, so it commutes with taking cohomology $H^i_t$. Therefore, 
\[ H^i_t(f) \colon H^i_t(F) \to H^i_t(F) \otimes q^*L^d \]
is the zero map for all such sequences $i_1, \dots, i_d$. The reverse direction of Lemma \ref{lem:supportviazeroemaps} gives
$H^i_t(F) \in D(Y \times S)_T$. Combined with $H^i_t(F) \in \cdlc(Y \times S)$ by definition of a t-structure on $\cdlc(T \times S)$, this proves our claim.
\end{proof}

\begin{lemma}
\label{lem:quotient-closed}
Under the assumptions of Theorem~\ref{thm:sheaf-criterion}, let $\cA$ denote the heart of the given t-structure. Then for every closed subset $T\subseteq S$, the subcategory $\cdlc(Y\times S)_T\cap \cA$ of the abelian category $\cA$ is closed under subobjects, quotients and extensions.
\end{lemma}
\begin{proof}
Tensoring with $q^*L$ is a t-exact functor on $\cdc(Y \times S)$, so it is exact on $\cA$. Given a short exact sequence $0 \ra E \ra F \ra G \ra 0$ in $\cA$, a diagram chase shows that if all maps $F \ra F\otimes q^*L^d$ as in Lemma \ref{lem:supportviazeroemaps} vanish, then so do all such maps $E \ra E\otimes q^*L^d$ and $G \ra G\otimes q^*L^d$. Conversely, if all the maps $E \ra E\otimes q^*L^d$ and $G \ra G\otimes q^*L^d$ vanish, then so do all the maps $F \ra F\otimes q^*L^{2d}$ given by sequences of length $2d$. The result follows from Lemma~ \ref{lem:supportviazeroemaps}.
\end{proof}

\begin{proof}[Proof of Theorem \ref{thm:sheaf-criterion}]
Let $U\subseteq S$ be an open subset. For any interval $[a,b]$ that may be infinite on one side, consider the subcategory
\begin{equation}
\label{eqn:open-restriction-t-str} 
\cdlc^{[a,b]}(Y\times U) = \big\{ F_0\in \cdlc(Y\times U) \mid \exists \; F \in \cdlc^{[a,b]}(Y\times S)\text{ such that } F_0=j^*F \big\}
\end{equation}
 of $\cdlc(Y\times U)$. We proceed in two steps. 
 
 \smallskip
 \noindent \textsc{Step 1}: To verify that \eqref{eqn:open-restriction-t-str} defines a bounded t-structure, clearly $\cD^{\leq -1}(Y\times U)\subseteq \cD^{\leq 0}(Y\times U)$. We next check that there are no nontrivial morphisms between any $F_0\in \cdlcn(Y\times U)$ and $G_0\in \cdlc^{\geq 1}(Y\times U)$. Proposition~\ref{prop:localisation} implies that if a morphism $F_0\to G_0$ does exist then it's obtained from a diagram of the form 
 \begin{equation}
 \label{eqn:morphismdiagram}
 F\stackrel{f}{\longleftarrow} F^\prime \ra G,
 \end{equation}
 where $F\in \cdlcn(Y\times S)$ and $G\in \cdlc^{\geq 1}(Y\times S)$ satisfy $j^*F\cong F_0$ and $j^*G\cong G_0$, and where the cone $C$ of $f$ lies in $\cdlc(Y\times S)_T$ for $T:= S\setminus U$. The long exact sequence of cohomology for the exact triangle $F^\prime \ra F \ra C$ shows $H_t^0(C)\to H_t^1(F^\prime)$ is surjective, so $H_t^1(F)\in \cdlc(Y\times S)_T$ by Lemma~\ref{lem:quotient-closed}; similarly, $H_t^i(C)\cong H_t^{i+1}(F^\prime)$ for $i\geq 1$ implies $H_t^{i+1}(F^\prime)\in \cdlc(Y\times S)_T$. Thus all cohomology sheaves of $\tau_t^{\geq 1}(F^\prime)$ lie in $\cdlc(Y\times S)_T$, and hence so does $\tau_t^{\geq 1}(F^\prime)$ by Lemma~\ref{lem:coh-support}. This object is the cone of $g\colon\tau_t^{\leq 0}(F^\prime)\to F^\prime$ which then lies in the localising class, and therefore the diagram
 \[
 F\stackrel{f\circ g}{\longleftarrow} \tau_t^{\leq 0}(F^\prime) \ra G
 \]
 is equivalent to that from \eqref{eqn:morphismdiagram}. The t-structure on $\cdlc(Y\times S)$ shows that this map is zero, so the original morphism from $F_0$ to $G_0$ is zero as required. To check condition \three\ from the definition of a $t$-structure in Section~\ref{sec:t-structures}, let $E_0\in \cdlc(Y\times U)$ and apply Corollary~\ref{cor:essential-surjective} to obtain $E\in \cdlc(Y\times S)$ such that $j^*E=E_0$. The given t-structure on $\cdlc(Y\times S)$ provides $F\in \cdlcn(Y\times S)$ and $G\in \cdlc^{\geq 1}(Y\times S)$ such that $F \ra E \ra G$ is an exact triangle. Since $j^*$ is exact, the objects  $j^*F\in \cdlcn(Y\times U)$ and $j^*G\in \cdlc^{\geq 1}(Y\times U)$ fit into an exact triangle $j^*F \ra E_0 \ra j^*G$, so \eqref{eqn:open-restriction-t-str} defines a t-structure. Since the t-structure on $\cdlc(Y\times S)$ is bounded, there exist integers $a<b$ such that $E \in \dlabc(Y\times S)$. Hence $E_0 \in \dlabc(Y\times U)$, which shows the t-structure on $\cdlc(Y\times U)$ is also bounded.

\smallskip

\noindent \textsc{Step 2}: We verify that the t-structure from Step 1 defines a sheaf of t-structures over $S$. For any open subsets $U' \subseteq U \subseteq S$, we have the following commutative diagrams of pullback functors
\[
\xymatrix{
\cdc(Y\times S) \ar[r]^{j^*} \ar@/_5mm/[rr]_{j{''}^*} & \cdc(Y\times U) \ar[r]^{j'^*} & \cdc(Y\times U').
}
\]
To show that $j'^*$ is t-exact, consider $F_0\in \cdcn(Y\times U)$ and choose $F\in \cdcn(Y\times S)$ such that $j^*F=F_0$. Then $j'^*F_0=j'^*j^*F=j''^*F \in \cdcn(Y\times U')$, so $j'^*$ is right exact. Left-exactness is similar, so the t-structure on $\cdc(Y \times S)$ induces a sheaf of t-structures over $S$. 
\end{proof}

Setting $L = \cO_S$ in Theorem~\ref{thm:sheaf-criterion} immediately gives:
\begin{corollary}
\label{cor:affine-sheaf}
Let $S$ be an affine scheme. Then every bounded t-structure on $\cdlc(Y \times S)$ determines by restriction a sheaf of t-structures on $\cdlc(Y\times S)$ over $S$.
\end{corollary}

\section{Sheaf of t-structures on \texorpdfstring{\protect$\cdc(\ypr)\protect$}{Dc(YxPr)} over \texorpdfstring{\protect$\mathbb{P}^r\protect$}{Pr}}
In this section we construct a sheaf of t-structures on $\cdc(Y\times \mathbb{P}^r)$ over $\mathbb{P}^r$ following the work of Abramovich--Polishchuk~\cite[Theorem~2.3.6]{AP06}.

\subsection{On resolution of the diagonal}
Let $Y$ be any scheme. For $r\geq 0$, let $p\colon Y\times \mathbb{P}^r\to Y$ denote the first projection and $X:= (Y\times \mathbb{P}^r)\times_Y (Y\times \mathbb{P}^r)$ the fibre product with fibre square 
\[
\xymatrix{
 X \ar[r]^{\pi_2} \ar[d]_{\pi_1} & Y\times \mathbb{P}^r \ar[d]^{p} \\
 Y\times \mathbb{P}^r \ar[r]^{p} & Y.
}
\]
For $F, G\in \DD(\Qcoh(Y\times \mathbb{P}^r))$, write $F\boxtimes G:=\pi_1^*F\otimes \pi_2^*G\in \DD(\Qcoh(X))$.  Let $q\colon Y\times \mathbb{P}^r \to \mathbb{P}^r$ denote the second projection, $\cO(1)=q^*\cO_{\mathbb{P}^r}(1)$ the relative hyperplane bundle and $\Omega:=q^*\Omega_{\mathbb{P}^r}$ the relative cotangent bundle. For $F\in \DD(\Qcoh(Y\times \mathbb{P}^r))$ and $i\in \mathbb{Z}$, write $F(i):= F\otimes \cO(i)$. 

The relative version of resolution of the diagonal $\Delta \subseteq X$ by Orlov~\cite{Or92a} is the resolution
\begin{equation}
\label{eqn:Beilinson}
0 \lra \Omega^r(r)\boxtimes \cO(-r) \lra \cdots \lra\Omega^1(1)\boxtimes\cO(-1) \lra \cO\boxtimes\cO \lra \cO_\Delta \lra 0.
\end{equation}
For $j \in \ZZ$, tensoring by $\cO(-j) \bt \cO(j)$ gives a resolution
\[
0 \lra \Omega^r(r-j)\boxtimes \cO(j-r) \lra \cdots \lra\Omega^1(1-j)\boxtimes\cO(j-1) \lra \cO(-j)\boxtimes\cO(j) \lra \cO_\Delta \lra 0.
\]
For $0\leq i\leq r$ and $j\in \ZZ$, each sheaf $\Omega^i(i-j) \boxtimes \cO(j-i)$ on $X$ defines an integral transform that we denote  $\Phi^{i,j}\colon D(Y\times \mathbb{P}^r) \lra D(Y\times \mathbb{P}^r)$, where
\begin{eqnarray*}
\Phi^{i,j}(F) & = & (\pi_2)_*\Big(\pi_1^*(F)\otimes \big(\Omega^i(i-j)\boxtimes \cO(j-i)\big)\Big) \\
                  & \cong & p^*p_*\big(F\otimes \Omega^i(i-j)\big)\otimes \cO(j-i)
 \end{eqnarray*}
by the projection formula and flat base change. By Proposition \ref{prop:functorpair}, this functor restricts to an integral transform $\Phi_c^{i,j}\colon D_c(Y\times \mathbb{P}^r) \lra D_c(Y\times \mathbb{P}^r)$.

For any fixed $j\in \ZZ$, if we break the above resolution into short exact sequences as described in \cite[Proof of Corollary 8.29]{Hu06a}, any object $F\in D_c(Y\times \mathbb{P}^r)$ can be reconstructed by taking successive cone operations on the collection
\begin{equation}
\label{eqn:1st-cone-collection}
\big\{ \Phi_c^{r,j}(F), \Phi_c^{r-1,j}(F), \dots, \Phi_c^{1,j}(F), \Phi_c^{0,j}(F) \big\}.
\end{equation}

\begin{remark}
\label{rem:orderswitch}
In writing the resolution \eqref{eqn:Beilinson} we could equally well have written each term as $\cO(-i)\boxtimes \Omega^i(i)$, in which case the resulting integral functor $\Psi^{i,j}$ would satisfy
\begin{equation}
\label{eqn:orderswitch}
\Psi^{i,j}(F) \cong p^*p_*\big(F(j-i)\big)\otimes \Omega^i(i-j)
\end{equation}
for every $0\leq i\leq r$ and $j\in \ZZ$.
\end{remark}

The next two results record several useful consequences of these observations.

\begin{lemma}
\label{lem:res}
For any $m \geq r+1$, there is an exact sequence  
\begin{equation}
\label{eqn:beilinsonsequence}
0 \lra V^m_r \otimes \cO_{\PP^r} \lra V^m_{r-1} \otimes \cO_{\PP^r}(1) \lra \cdots \lra V^m_0 \otimes \cO_{\PP^r}(r) \lra \cO_{\PP^r}(m) \lra 0,
\end{equation}
where $V^m_i=H^0(\PP^r, \Omega^i(m-r+i))$ for $0 \leq i \leq r$. 
\end{lemma}
\begin{proof}
 The observations above for $Y=\spec \kk$ and $j=r$ show that the sheaf $\cO_{\PP^r}(m) \in D(\PP^r)$ can be reconstructed by taking successive cone operations using the objects
\[
\Phi_{i,r}(\cO(m)) =p^*p_*\big(\cO(m)\otimes \Omega^i(i-r)\big)\otimes \cO(r-i) \cong \mathbf{R}\Gamma\big(\Omega^i(m-r+i)\big) \otimes \cO(r-i)
\]
 for $0\leq i\leq r$. Our assumption on $m$ gives $m-r+i>0$, so Manivel~\cite{Manivel96} implies that the higher cohomology groups of $\Omega^i(m-r+i))$ vanish; hence
 $\Phi^{i,r}(\cO(m)) = V^m_i \otimes \cO(r-i)$. Substituting these sheaves for $0\leq i\leq r$ into the above cone operations yields \eqref{eqn:beilinsonsequence}.
\end{proof}

\begin{lemma}
\label{lem:semiorthog}
Let $Y$ be a scheme and $p\colon Y\times \mathbb{P}^r\to Y$ the first projection. Let $F\in D(Y\times \mathbb{P}^r)$. 
\begin{enumerate}
\item[\one] We have $F\in D_c(Y\times \mathbb{P}^r)$ if and only if there exists $j\in \ZZ$ such that $p_*(F(j-i))\in D_c(Y)$ for all $0\leq i\leq r$.
\item[\two] If $p_*F(-i) = 0$ for $0 \leq i \leq r$, then $F = 0$.
\item[\three] There is a semi-orthogonal decomposition
\begin{equation}
\label{eqn:semi-orthogonal}
D_c(Y\times \mathbb{P}^r )=\big\langle p^*D_c(Y)(-r), \cdots, p^*D_c(Y)(-1), p^*D_c(Y)\big\rangle.
\end{equation}
\end{enumerate}
\end{lemma}
\begin{proof}
For \one, the `only if' direction follows from Lemma~\ref{lem:cpt-supp}\one\ and Corollary~\ref{cor:cpt-supp}. Conversely, the object $F\in D(\ypr)$ can be reconstructed by taking successive cone operations on the collection $\big\{ \Psi^{r,j}(F), \Psi^{r-1,j}(F), \dots, \Psi^{1,j}(F), \Psi^{0,j}(F) \big\}$ from \eqref{eqn:orderswitch}. Lemma~\ref{lem:cpt-supp}\one\ and Corollary~\ref{cor:cpt-supp} imply that $\Psi^{i,j}(F)\in D_c(\ypr)$ for $0\leq i\leq r$, and hence $F\in D_c(Y\times \mathbb{P}^r)$. For \two, our assumption on $F$ ensures that the objects $\Psi^{i,0}(F)$ from \eqref{eqn:orderswitch} are trivial for $0\leq i\leq r$, so $F \cong 0$ after taking cones. For \three, Lemma~\ref{lem:cpt-supp}\one\ and Corollary~\ref{cor:cpt-supp} imply that we obtain a functor $\phi_i\colon \cdc(Y)\to \cdc(Y\times \mathbb{P}^r)$ for each $0\leq i\leq r$  by setting $\phi_i(F):= p^*(F)\otimes\cO(-i)$. Each $\phi_i$ is fully faithful
by the projection formula. The approach of Orlov~\cite[\S 2]{Or92a} shows that the sequence of subcategories on the right-hand side of \eqref{eqn:semi-orthogonal} is semiorthogonal. As for generation, consider $F\in D_c(Y\times \PP^r)$. For $0\leq i \leq r$, the object $\Phi_{i,0}(F)=p^*p_*(F\otimes\Omega^i(i))\otimes\cO(-i)$ from collection \eqref{eqn:1st-cone-collection} lies in $p^*D_c(Y)(-i)$, so after taking cones we have that $F$ is contained in the right side of \eqref{eqn:semi-orthogonal}, as required.
\end{proof}

\begin{proposition}
\label{prop:proj-cpt-supp}
Let $L$ be an ample bundle on a projective scheme $S$, and write $q\colon Y \times S\ra S$ for the second projection. For any $F\in D(Y\times S)$, we have 
\begin{equation}
\label{eqn:cats-proj}
F \in D_c(Y \times S)\iff p_*(F \otimes  q^*L^n) \in D_c(Y) \quad \text{ for all } n \gg 0.
\end{equation}
\end{proposition}
\begin{proof}
One direction is immediate from Lemma~\ref{lem:cpt-supp}\one\ and Corollary~\ref{cor:cpt-supp}. For the other direction, it is enough to show that
$\supp(F) \subseteq \supp(p_* (F \otimes q^*L^n)) \times S$ for $n \gg 0$. This follows from Lemma~\ref{lem:easycompact}\two\ if we choose $n$ large enough such that for each cohomology sheaf $H^i(F)$, the tensor product $H^i(F) \otimes q^*L^n$ is globally generated over $Y$, in other words, that the canonical map
$p^* p_* \left(H^i(F) \otimes q^*L^n\right) \to H^i(F) \otimes q^*L^n$ is surjective.
\end{proof}

\subsection{A family of t-structures}
From now on we work under the following assumption:

\begin{assumption}
\label{assumption}
Let $(\cdcn(Y), \cdcp(Y))$ be a Noetherian, bounded t-structure on $\cdc(Y)$. \end{assumption}

 For any interval $[a,b]$ that may be infinite on one side, we obtain a subcategory $\cdc^{[a,b]}(Y)$ of $\cdc(Y)$, where $\cA:= D_c^{[0,0]}(Y)$ is the heart.  For $i\in \ZZ$, we write $H^i_Y(-):=\tau_Y^{\geq 0}\tau_Y^{\leq 0}(-[i])$ for the $i$-th cohomology functor, where $\tau^{\leq 0}_Y, \tau^{\geq 0}_Y$ denote the truncation functors.

The following purely categorical result, which combines \cite[Lemma 3.1.1-3.1.2]{P07}, enables us to glue t-structures of subcategories arising in a semi-orthogonal decomposition. 

\begin{lemma}[\cite{P07}]
\label{lem:1st-t-str}
Let $\cD = \langle \cA_1, \cdots, \cA_n \rangle$ be a semi-orthogonal decomposition, and let $(\cA_i^{\leq 0}, \cA_i^{\geq 0})$ be a t-structure on $\cA_i$ for $1\leq i\leq n$. Assume in addition that each inclusion $\cA_i \hookrightarrow \cD$ has a right adjoint $\rho_i\colon \cD \ra \cA_i$, and that for each pair of indices $i<j$, the restriction functor
$\rho_i | _{\cA_j}\colon \cA_j \ra \cA_i$ is right t-exact. Then we obtain a t-structure on $\cD$ by setting 
\[
\cD_\rho^{[a,b]} = \big\{ F \in \cD \mid \rho_i(F) \in \cA_i^{[a,b]} \text{ for all } i=1, \cdots, n \big\}
\]
 for any interval $[a,b]$ that may be infinite on one side.
\end{lemma}

\begin{proposition}
\label{prop:t-str-n} 
 For $n\geq 0$, there exists a bounded t-structure on $\cdc(\ypr)$, where for any interval $[a,b]$ that may be infinite on one side, we have
\[
D_c^{[a,b]}(\ypr)_n := \big\{ F \in D(\ypr) \mid p_* F(n+i) \in \cdc^{[a,b]}(Y) \text{ for } 0\leq i\leq r \big\}.
\]
\end{proposition}
\begin{proof}
 Lemma~\ref{lem:semiorthog}\one\ implies that each $D_c^{[a,b]}(\ypr)_n$ is contained in $\cdc(\ypr)$. By tensoring with $\cO(-n)$, it suffices to prove the result for $n=0$. In this case, we show that Lemma~\ref{lem:1st-t-str} applies to semiorthogonal decomposition of $\cdc(\ypr)$ from \eqref{eqn:semi-orthogonal}. For $0\leq i\leq r$, identify $D_c(Y)$ with $p^*\cdc(Y)(-i)$ via the fully faithful functor $\phi_i\colon \cdc(Y)\to \cdc(\ypr)$ sending $F$ to $p^*(F)(-i)$. In particular, each $p^*\cdc(Y)(-i)$ inherits a t-structure. The right-adjoint to $\phi_i$ is $\rho_i\colon \cdc(\ypr) \ra \cdc(Y)$ where $\rho_i(F) = p_*(F(i))$. For any $i<j$ and $F \in \cdc(Y)$, we
obtain
\[
(\rho_j \circ \phi_i)(F) = F \otimes p_*\cO(j-i) = F \otimes H^0(\cO(j-i)).
\]
 Therefore $\rho_j \circ \phi_i$ is t-exact for any t-structure on $\cdc(Y)$, so Lemma~\ref{lem:1st-t-str} gives the t-structure on $\cdc(\ypr)$. To show boundedness, let $F\in \cdc(Y\times \mathbb{P}^r)$. For $0\leq i\leq r$ we have $p_*(F(n+i))\in \cdc(Y)$ by Corollary~\ref{cor:cpt-supp}. Boundedness of the t-structure on $\cdc(Y)$ gives $k_0, \dots, k_r$ such that $p_*(F(n+i))\in \cdc^{\leq k_i}(Y)$. Then $k=\max_{0 \leq i \leq r} k_i$ gives $p_*(F(n+i))\in \cdc^{\leq k}(Y)$ for $0\leq i\leq r$, and hence $F\in \cdc^{\leq k}(\ypr)_n$ as required.
\end{proof}

For $n\geq 0$ and $i\in \ZZ$, let $H^i_n(-):=\tau_n^{\geq 0}\tau_n^{\leq 0}(-[i])$ denote the $i$th cohomology functor of the t-structure from Proposition~\ref{prop:t-str-n}, where $\tau^{\leq 0}_n$ and $\tau^{\geq 0}_n$ denote the truncation functors. Let 
\[
D_n^{[a,b]}:=D_c^{[a,b]}(\ypr)_n
\]
denote the subcategory coming from the bounded t-structure of Proposition~\ref{prop:t-str-n}.

\begin{corollary}
\label{cor:push}
 Let $F \in \cdc(\ypr)$. 
 \begin{enumerate}
 \item[\one] If $m \geq r+1$, then $p_*F(m)$ can be reconstructed by taking successive cone operations using the objects $V^m_{r-i} \otimes p_*F(i)[r-i]$ for $0 \leq i \leq r$, where $V^m_i:=H^0(\PP^r, \Omega^i(m-r+i))$ as in Lemma~\ref{lem:res}.
 \item[\two] If $m\geq 0$ and $F\in D_0^{[0,0]}$, then $p_*F(m) \in D_c^{[-r,0]}(Y)$.
 \end{enumerate}
\end{corollary}

\begin{proof}
Pull the resolution of $\cO_{\PP^r}(m)$ from Lemma~\ref{lem:res} back along $q\colon Y\times \PP^r\to \PP^r$, tensor with $F$ and pushforward along $p\colon Y\times \PP^r\to Y$ to obtain a resolution of $p_*F(m)$ in terms of 
\[
p_*\Big(F\otimes q^*\big(V^m_{r-i}\otimes \cO_{\PP^r}(i)[r-i]\big)\Big) = V^m_{r-i}\otimes p_*F(i)[r-i] \quad\]
for $0\leq i\leq r$ which gives \one. By definition, if $F\in D^{[0,0]}_0$, then $p_*F(i)\in D_c^{[0,0]}(Y)$ for $0\leq i\leq r$, giving \two\ in case $0\leq m\leq r$ as $D_c^{[0,0]}(Y)\subset D_c^{[-r,0]}(Y)$. It follows that $p_*F(i)[r-i]\in D_c^{[-r,0]}(Y)$ for $0\leq i\leq r$, so if $m\geq r+1$, then \two\ follows from the resolution in part \one.
\end{proof}

\begin{lemma}
\label{lem:inclusion}
The t-structures $(D_n^{\leq 0},D_n^{\geq 0})$ from Proposition~\ref{prop:t-str-n} satisfy
\begin{align}
D^{\leq 0}_0 &\subseteq D^{\leq 0}_1 \subseteq D^{\leq 0}_2 \subseteq \cdots \subseteq D^{\leq r}_0; \label{eqn:1st-inclusion}\\
D^{\geq 0}_0 &\supseteq D^{\geq 0}_1 \supseteq D^{\geq 0}_2 \supseteq \cdots \supseteq D^{\geq r}_0. \label{eqn:2nd-inclusion}
\end{align}
In particular, for each $i\in \ZZ$ there is a morphism $H^i_n(-)\lra H^i_{n+1}(-)$ of cohomology functors.
\end{lemma}
\begin{proof}
For the two chains of inclusions, it is enough to prove $D^{\leq 0}_0 \subseteq D^{\leq 0}_n$ and $D^{\geq r}_0 \subseteq D^{\geq 0}_n$ for all $n \geq 0$. We begin with the former inclusion. Since the t-structures are bounded, it is enough to prove $D^{[0,0]}_0 \subseteq D^{\leq 0}_n$. But this claim follows from Corollary \ref{cor:push}\two\ and the definition of $D^{\leq 0}_n$. For the other inclusion, it is enough to show that $D^{[0,0]}_0 \subseteq D^{\geq -r}_n$, and the proof is similar.

To construct the morphism of cohomology functors, note that the inclusion $D_n^{\leq 0}\subseteq D_{n+1}^{\leq 0}$ implies that the morphism $\tau_n^{\leq 0}(F)\to F$ for each $F\in \cdc(\ypr)$ factors through $\tau_{n+1}^{\leq 0}(F)$, giving a transformation $\tau_n^{\leq 0}\to \tau_{n+1}^{\leq 0}$. The inclusion $D_{n+1}^{\geq 0}\subseteq D_n^{\geq 0}$ similarly gives $\tau_n^{\geq 0} \to \tau_{n+1}^{\geq 0}$. Then for all $i\in \ZZ$ we have a morphism $\tau_n^{\leq 0}(F[i])\to \tau_{n+1}^{\leq 0}(F[i])$, and hence a morphism
\[
H^i_n(F)=\tau_n^{\geq 0}\tau_n^{\leq 0}(F[i]) \lra \tau_n^{\geq 0}\tau_{n+1}^{\leq 0}(F[i])\lra \tau_{n+1}^{\geq 0}\tau_{n+1}^{\leq 0}(F[i]) = H^i_{n+1}(F)
\]
as required.
\end{proof}

\subsection{On graded $S$-modules in an abelian category}
For  $V=H^0(\PP^r, \cO(1))$, the symmetric algebra of $V$ is a graded $\kk$-algebra $S=\bigoplus_{m\geq 0}S^m$ generated by $r+1$ variables of degree one. We now recall several categorical notions and results from \cite[Section~2.2]{AP06}, where the abelian category of interest is the heart $\cA$ of the t-structure on $\cdc(Y)$ given by Assumption~\ref{assumption}. 

A \emph{graded $S$-module in $\cA$} is a collection $M=\{M_n \mid n\in \ZZ\}$ of objects in $\cA$ and a collection of  morphisms $\{\varphi_{m,n}\colon S_m\otimes M_n\to M_{m+n}\mid m,n\in \ZZ, m \geq 0\}$ satisfying the obvious associativity condition, such that $\varphi_{0,n}$ is the identity for each $n\in \ZZ$. We typically write $M=\oplus_{n\in \ZZ} M_n$. A \emph{morphism} of graded $S$-modules in $\cA$ is a collection of morphisms $\{f_n\colon M_n\to M'_n \mid n\in \ZZ\}$ satisfying $f_{m+n}\circ \varphi_{m,n}=\varphi^\prime_{m,n}\circ (\id_{S_m}\otimes f_n)$ for all $m,n\in \ZZ$. A \emph{free graded $S$-module of finite type in $\cA$} is a finite direct sum of graded $S$-modules in $\cA$ of the form $S\otimes M(i)$, where
\[
\big(S\otimes M(i)\big)_n = S_{i+n}\otimes M
\]
 for an object $M$ in $\cA$ and a fixed $i\in \ZZ$, and where the morphisms 
\[
S_m\otimes (S_{n+i}\otimes M)\lra S_{n+m+i}\otimes M
\]
for $i,m,n\in \ZZ$ are induced by multiplication in $S$, namely $S_m\otimes S_{n+i}\ra S_{m+n+i}$. A graded $S$-module $M$ in $\cA$ is of \emph{finite type} if there is a surjective map $P\to M$ for a free graded $S$-module $P$ of finite type in $\cA$. The main result we require is the following \cite[Theorem~2.2.2]{AP06}.

\begin{theorem}
\label{thm:abcats}
The category of graded $S$-modules of finite type in $\cA$ is abelian and Noetherian.
\end{theorem}

We record the following examples for later use.

\begin{examples}
\label{exa:fintype}
Let $F\in \cA$ be any object.
\begin{enumerate}
\item Let $\ell\geq 0$. If $\oplus_n M_n$ is a graded $S$-module of finite type, then so is $\bigoplus_{m\geq \ell} M_n$. It follows that the graded $S$-module $\bigoplus_{m\geq \ell} S_n\otimes F$ is of finite type in $\cA$.
\item For $m\geq 0$, tensor the Euler exact sequence on $\PP^r$ by $\cO_{\PP^r}(m)$ and apply the global sections functor to see that $H^0(\PP^r,\Omega^1(m))$ is the kernel of a map $S_{m-1}^{\oplus(r+1)}\ra S_m$. The direct sum of all such maps shows that $\bigoplus_{m\geq 0} H^0(\PP^r,\Omega^1(m))$ is the kernel of a homomorphism $S(-1)^{\oplus(r+1)}\ra S$ of free graded $S$-modules. Theorem~\ref{thm:abcats} implies that $\bigoplus_{m\geq 0} H^0(\PP^r,\Omega^1(m))\otimes F$ is a graded $S$-module of finite type in $\cA$, so for any $\ell\geq 0$, part (1) shows that $\bigoplus_{m\geq \ell} H^0(\PP^r,\Omega^1(m))\otimes F$ is a graded $S$-module of finite type in $\cA$. 
\end{enumerate}
\end{examples} 
  
\begin{lemma}
\label{lem:exact}
Let $M =\oplus_n M_n$ be a graded $S$-module of finite type in $\cA$. The complex 
\begin{equation}
\label{eqn:knm}
\textstyle{0 \lra \bigwedge^{r+1}V \otimes M_{d-(r+1)} \lra \cdots \lra \bigwedge^2V \otimes M_{d-2}\lra V\otimes M_{d-1} \lra M_{d} \lra 0}
\end{equation}
 obtained as the strand of the Koszul complex for $M$ in degree $d$  is exact for $d\gg 0$.
\end{lemma}

\begin{proof}
The proof is contained in \cite[Step 5 of Proof of Proposition 2.3.3]{AP06}.
\end{proof}

\subsection{A sheaf of t-structures on $\cdc(\ypr)$}
We now use the family of t-structures from Proposition~\ref{prop:t-str-n} to construct a `limiting' t-structure on $\cdc(\ypr)$ that is actually a sheaf of t-structures over $\PP^r$. We continue to work under Assumption~\ref{assumption}.

As a first step, we provide an application of the categorical results from the previous section by establishing a technical result that will be used in the proof of Proposition~\ref{prop:stab} to follow. Recall that $D_0^{[0,0]}$ denotes the heart of the 0th t-structure on $\cdc(\ypr)$ constructed in Proposition~\ref{prop:t-str-n}.

\begin{lemma}
\label{lem:technical}
Let $F\in D_0^{[0,0]}$. Then $G_n:=H^0_0(F(n))$ satisfies $p_*G_n(i)\in \cA$ for $0\leq i\leq r+1$ and $n\gg 0$.
\end{lemma}
\begin{proof}
Since $F, G_n\in D_0^{[0,0]}$, we have $p_*F(i), p_*G_n(i)\in \cA$ for $0\leq i\leq r$. It remains to prove that $p_*G_n(r+1)\in \cA$. We proceed in three steps.

\smallskip

\noindent \textsc{Step 1}: For $0\leq i\leq r$, show that $p_*G_n(i)\cong H^0_Y(p_*F(n+i))$. 

\smallskip

Since $F\in D_0^{\leq 0}\subseteq D_n^{\leq 0}$ by Lemma~\ref{lem:inclusion}, we have $F(n)\in D_0^{\leq 0}$ and hence $G_n=H^0_0(F(n))=\tau_0^{\geq 0}F(n)$. We therefore have an exact triangle $\tau_0^{\leq -1}F(n) \ra F(n) \ra G_n$. For $0 \leq i \leq r$, by tensoring with $\cO(i)$ and pushing forward, we get another exact triangle
\begin{equation}
\label{eqn:trianglep*Gni}
p_*\big(\tau_0^{\leq -1}F(n)(i)\big) \lra p_*F(n+i) \lra p_*G_n(i)\end{equation}
The left-hand object in \eqref{eqn:trianglep*Gni} lies in $D_c^{\leq -1}(Y)$ for $0 \leq i \leq r$ because $\tau_0^{\leq -1}F(n)\in D_0^{\leq -1}$. Since $p_*G_n(i) \in \cA\subseteq D_c^{\geq 0}(Y)$, the exact triangle \eqref{eqn:trianglep*Gni} is a standard truncation triangle and hence by uniqueness of objects in such triangles we have $p_*G_n(i)= \tau_Y^{\geq 0}p_*F(n+i)$. It remains to note that $p_*F(n+i) \in D_c^{[-r,0]}(Y)$ for $n\geq r+1$ and $0 \leq i \leq r$ by Corollary \ref{cor:push}, from which we obtain $H^0_Y(p_*F(n+i))\cong \tau_Y^{\geq 0}p_*F(n+i)\cong p_*G_n(i)$ as required. 

\smallskip

\noindent \textsc{Step 2}: For $0\leq i\leq r$, show $p_*G_n(i)$ is the $(n+i-r)$-graded piece of a specific $S(V)$-module of finite type in $\cA$.

\smallskip

For $m\geq r+1$ and $V_i^m=H^0(\PP^r,\Omega^i(m-r+i))$, Corollary~\ref{cor:push} shows $p_*F(m) \in D_c^{[-r,0]}(Y)$ is the cone of a morphism $V^{m}_1 \otimes p_*F(r-1) \ra V^{m}_0 \otimes p_*F(r)$. The direct sum of all such maps for $m\geq r+1$ is a graded $S(V)$-module homomorphism
\[
\phi\colon \bigoplus_{m\geq r+1} H^0(\PP^r,\Omega^1(m-r+1))\otimes p_*F(r-1) \lra \bigoplus_{m\geq r+1} H^0(\cO_{\PP^r}(m-r))\otimes p_*F(r).
\]
Examples~\ref{exa:fintype} and Theorem~\ref{thm:abcats} imply that the cokernel of this map $\cok(\phi)=\oplus_j C_j$ is a graded $S(V)$-module of finite type in $\cA$. Moreover, for any $n \geq r+1$ and $0 \leq i \leq r$, the equality $V_0^{n+i}\cong S(V)_{n+i-r}$ shows that the $(n+i-r)$-graded piece of this module satisfies
\[
C_{n+i-r}  =  \cok \left( V^{n+i}_1 \otimes p_*F(r-1) \lra V^{n+i}_0\otimes p_*F(r)\right)
\cong  H_Y^0(p_*F(n+i)) \cong p_*G_n(i)
\]
by Step 1 above. This completes Step 2.

\medskip

\noindent \textsc{Step 3}: Deduce that $p_*G_n(r+1)\in \cA$ using a resolution by $p_*G_n(i)=C_{n+i-r}$ for $0 \leq i \leq r$.

\smallskip

The twist of the Koszul complex associated to the space $V=H^0(\PP^r,\cO(1))$ is the resolution
\[
\textstyle{0\lra \bigwedge^{r+1}V \otimes \cO_{\PP^r} \lra \bigwedge^rV \otimes \cO_{\PP^r}(1) \lra \cdots \lra V \otimes \cO_{\PP^r}(r)\lra \cO_{\PP^r}(r+1)\lra 0.}
\]
As in the proof of Corollary~\ref{cor:push}, pull this resolution back along $q\colon Y\times \PP^r\to \PP^r$, tensor with $G_n$ and pushforward along $p\colon Y\times \PP^r\to Y$ to obtain a quasi-isomorphism from the complex
\begin{equation}
\label{eqn:Gncomplex}
\textstyle{\bigwedge^{r+1}V \otimes p_*G_n \lra \bigwedge^rV \otimes p_*G_n(1) \lra \cdots \bigwedge^2V\otimes p_*G_n(r-1)\lra V \otimes p_*G_n(r)}
\end{equation}
of objects in $\cA$ to $p_*G_n(r+1)$.  To show that $p_*G_n(r+1)\in \cA$, we need only show that the complex \eqref{eqn:Gncomplex} has nonzero cohomology only in the right-hand position. For this, Step 2 enables us to rewrite this complex as
\begin{equation*}
\textstyle{\bigwedge^{r+1}V \otimes C_{n-r}\lra \bigwedge^{r}V \otimes C_{n-r+1} \lra \cdots \lra \bigwedge^2V \otimes C_{n-1} \lra V \otimes C_n},
\end{equation*}
which we recognise from \eqref{eqn:knm} as forming part of the Koszul complex of degree $n+1$ associated to the graded $S(V)$-module $\oplus_n C_n$ in $\cA$. We need only show that this latter complex is exact for $n\gg 0$, but this is immediate from Lemma \ref{lem:exact} because $\oplus_n C_n$ is of finite type.
\end{proof}

The key observation in constructing the sheaf of t-structures is the following stabilisation result for the cohomology objects $H^i_n(F)$ associated to the family of t-structures $(D^{\leq 0}_n, D^{\geq 0}_n)$ on $\cdc(Y \times \PP^r)$ constructed in Proposition~\ref{prop:t-str-n}. Recall from Lemma~\ref{lem:inclusion} that there exist canonical morphisms $H^i_n(F) \lra H^i_{n+1}(F)$ for all $n\geq 0$ and $i\in \ZZ$. 

\begin{proposition}
\label{prop:stab}
 For every $F \in \cdc(Y \times \PP^r)$, there exists $N\in \ZZ$ such that for every $n\geq N$, the canonical morphism $H^i_n(F) \ra H^i_{n+1}(F)$ is an isomorphism for all $i \in \ZZ$. 
\end{proposition}

\begin{proof}
We claim that it suffices to prove that the highest nonzero cohomology group of $F$ stabilises. Indeed, boundedness of the t-structure $(D_0^{\leq 0},D_0^{\geq 0})$ gives $a,b\in \ZZ$ such that $F\in D_0^{[a,b]}$, so the inclusions \eqref{eqn:1st-inclusion} and \eqref{eqn:2nd-inclusion} imply that for all $n\geq 0$ we have $F\in D_n^{[a-r,b]}$. If we can find $N>0$ such that $H^b_n(F) \cong H^b_{n+1}(F)$ for $n\geq N$, then truncating at $b$ gives an isomorphism 
\[
\xymatrix{
 \tau_n^{\leq b-1}F \ar[r]\ar[d] & F \ar[d]^{=}\ar[r] & H^b_n(F) \ar[d]^{\cong} \\
\tau_{n+1}^{\leq b-1}F \ar[r] & F \ar[r] & H^b_{n+1}(F)\\
}
\]
of exact triangles and hence $\tau_N^{\leq b-1}F\in D_n^{[a-r,b-1]}$ for $n\geq N$. Now consider $\tau_N^{\leq b-1}F$, and apply induction on the length of the interval $[a-r,b]$ to obtain the statement of the proposition.

To simplify the claim, shift $F$ by $b$ to obtain $F\in D_0^{\leq 0}$ and hence $F \in D_n^{\leq 0}$ for $n \geq 0$ by \eqref{eqn:1st-inclusion}. Proving the claim is equivalent to proving that $\tau_n^{\geq 0}F \ra \tau_{n+1}^{\geq 0}F$ is an isomorphism for $n \gg 0$. Since $\tau_{n}^{\geq 0}$ annihilates all objects in $D_0^{\leq -1}\subseteq D_{n}^{\leq -1}$ by \eqref{eqn:1st-inclusion}, applying $\tau_{n}^{\geq 0}$ to the triangle $\tau_0^{\leq -1}F\to F\to \tau_0^{\geq 0}F$ gives $\tau_{n}^{\geq 0}F\cong \tau_{n}^{\geq 0}\tau_0^{\geq 0} F$. Thus, we may replace $F$ by $\tau_0^{\geq 0}(F)$ throughout, i.e., we may assume $F \in D_0^{[0,0]}$. Similarly, since $\tau_{n+1}^{\geq 0}$ annihilates objects in $D_n^{\leq -1}\subseteq D_{n+1}^{\leq -1}$ by \eqref{eqn:1st-inclusion}, we have $\tau_{n+1}^{\geq 0}F\cong \tau_{n+1}^{\geq 0}\tau_n^{\geq 0} F$. The claim now reduces to proving that $\tau_n^{\geq 0}(F) \in D_{n+1}^{\geq 0}$. This means $H^0_n(F) \in D_{n+1}^{\geq 0}$, or equivalently, $H^0_n(F)(n+1)\in D_0^{\geq 0}$. Tensoring by $\cO(n)$ defines an autoequivalence of $\cdc(\ypr)$ that takes the t-structure $(D_n^{\leq 0}, D_n^{\geq 0})$ to $(D_0^{\leq 0}, D_0^{\geq 0})$, so $H^0_n(F) = H^0_0(F(n))(-n)$. As a result, the claim is equivalent to requiring that each $F\in D_0^{[0,0]}$ satisfies $H^0_0(F(n))(1)\in D_0^{\geq 0}$ for $n\gg 0$. To complete the proof, it remains to show that for $F\in D_0^{[0,0]}$, the object $G_n:=H^0_0(F(n))$ satisfies $G_n(1) \in D_0^{\geq 0}$. This is equivalent to showing that $p_*G_n(i+1)\in D_c^{\geq 0}(Y)$ for $0\leq i\leq r$ and $n\gg 0$ which is immediate from Lemma~\ref{lem:technical}.
\end{proof}

\begin{theorem}
\label{thm:limitingtstructure}
Consider the t-structure on $\cdc(Y)$ from Assumption~\ref{assumption}. Then there is a sheaf of t-structures on $\cdc(\ypr)$ over $\PP^r$, where for any interval $[a,b]$ that may be infinite on one side, we have
\[
D_c^{[a,b]}(\ypr) = \{ F\in D(\ypr) \mid p_* F(n) \in D_c^{[a,b]}(Y) \text{ for all } n \gg 0 \}.
\]
 Moreover, this t-structure satisfies 
\begin{equation}
\label{eqn:relations}
D_c^{\leq 0}(\ypr) = \bigcup_{n \geq 0} D_c^{\leq 0}(\ypr)_n\quad \text{and}\quad D_c^{\geq 0}(\ypr) = \bigcap_{n \geq 0} D_c^{\geq 0}(\ypr)_n. 
\end{equation}
\end{theorem}

\begin{proof}
 Lemma~\ref{lem:semiorthog}\one\ implies that each $D_c^{[a,b]}(\ypr)$ is contained in $\cdc(\ypr)$. Write $D^{\leq 0}=D_c^{\leq 0}(\ypr)$ and $D^{\geq 0}=D_c^{\geq 0}(\ypr)$. Now \eqref{eqn:relations} follows from \eqref{eqn:1st-inclusion} and \eqref{eqn:2nd-inclusion}, and hence $D^{\leq -1}\subseteq D^{\leq 0}$. Moreover, $D^{\geq 1}$ is right-orthogonal to $D^{\leq 0}$, because for $F\in D^{\geq 1}=\bigcap_{n\geq 0} D_n^{\geq 1}$, we have $\Hom(E,F)=0$ for all $n\geq 0$ and $E\in D_n^{\leq 0}$, so $F$ is orthogonal to $\bigcup_{n\geq 0}D_n^{\leq 0}=D^{\leq 0}$.  To define the truncation functors, Proposition~\ref{prop:stab} associates to each $F\in \cdc(\ypr)$ an integer $N$ such that $H^i_N(F)\in D_n^{[0,0]}$ for all $n\geq N$ and $i\in \ZZ$. In particular, $\tau_N^{\leq 0}F\in  \bigcup_{n \geq N} D^{\leq 0}_n\subseteq D^{\leq 0}$ and $\tau_N^{\geq 1}F\in  \bigcap_{n \geq N} D^{\geq 1}_n\subseteq D^{\geq 1}$ by \eqref{eqn:1st-inclusion} and \eqref{eqn:2nd-inclusion} respectively. Define 
\begin{equation}
\label{eqn:limitingtruncation}
\tau^{\leq 0}\colon \cdc(\ypr)\lra D^{\leq 0}\quad\text{and}\quad\tau^{\geq 0}\colon \cdc(\ypr)\lra D^{\geq 0}
\end{equation}
by setting $\tau^{\leq 0}F=\tau_n^{\leq 0}F$ for $n\gg 0$, and $\tau^{\geq 0}F=\tau_n^{\geq 0}F$ for $n\gg 0$. For any $F\in \cdc(\ypr)$, the truncation triangle is $\tau_n^{\leq 0}F\ra F\ra \tau_n^{\geq 1}F$ for $n\gg 0$, so $(D^{\leq 0},D^{\geq 0})$ is a t-structure. For boundedness, let $F\in \cdc(\ypr)$. Since the t-structures $(D_n^{\leq 0},D_n^{\geq 0})$ are bounded for $n\geq 0$, there exists $a_n, b_n\in \ZZ$ such that $F\in D_n^{[a_n,b_n]}$ for $n\geq 0$. Proposition~\ref{prop:stab} gives $N\in \ZZ$ such that the cohomology groups stabilise for t-structures indexed by $n\geq N$ and hence $F\in D_n^{[a_N,b_N]}$ for $n\geq N$. For $a:=\min_{0\leq i\leq N} a_i$ and $b:=\max_{0\leq i\leq N} b_i$, we have $F\in D_n^{\geq a}\cap D_n^{\leq b}=D_n^{[a,b]}$ for $n\geq 0$. Therefore $F\in D_c^{[a,b]}(\ypr)$ as required. Finally, tensoring by $q^*\cO_{\PP^r}(1)$ preserves the heart, so it's t-exact, and the result follows from Theorem \ref{thm:sheaf-criterion}.
\end{proof}

For $i\in \ZZ$, let $H^i(-):=\tau^{\geq 0}\tau^{\leq 0}(-[i])$ denote the $i$th cohomology functor of the t-structure from Theorem~\ref{thm:limitingtstructure}, where the truncation functors are given by \eqref{eqn:limitingtruncation}.

\subsection{The Noetherian property}
Theorem~\ref{thm:limitingtstructure} provides a bounded t-structure on $\cdc(Y\times U)$ for each open subset $U\subseteq \PP^r$. We now show that every such t-structure has a Noetherian heart. For this, associate to each $F\in D^{[0,0]}$ a graded $S$-module $M(F)$ in the category $\cA$ by setting
\[
M(F)_n:=\left\{\begin{array}{cc} H^0_Y(p_*F(n)) & \text{for } n>0 \\ 0 & \text{otherwise,}\end{array}\right.
\]
where $H^i_Y(-)$ is the $i$th cohomology functor for the t-structure on $\cdc(Y)$ from Assumption~\ref{assumption}, and where the maps $S_1\otimes M(F)_n\lra M(F)_{n+1}$ giving the $S$-module structure are obtained by applying $H^0_Y(p_*F(-\otimes \cO(n+1)))$ to the right-hand map from the Euler sequence on $\PP^r$. Note that $M(-)_n\colon D^{[0,0]}\to \cA$ is a functor for each $n\in \ZZ$, and hence so is $M(-):=\oplus_n M(-)_n$.

\begin{lemma}
\label{lem:leftexactfunctor}
The functors $M(-)_n$ and $M(-)$ are left-exact, while $M(-)_n$ is exact for $n\gg 0$. Moreover, for $F\in D^{[0,0]}$ we have that:
\begin{enumerate}
\item[\one] the graded $S$-module $M(F)$ in $\cA$ is of finite type. 
\item[\two] if a subobject $F'\subseteq F$ satisfies $M(F')_n=M(F)_n$ for $n\gg 0$, then $F'=F$.
\end{enumerate}
\end{lemma}
\begin{proof}
For the first statement, consider an exact sequence
\[
0\lra F_1\lra F_2\lra F_3\lra 0
\]
in $D^{[0,0]}$. Apply $p_*(-\otimes \cO(n))$ and take the long exact sequence in cohomology for the t-structure on $\cdc(Y)$ from Assumption~\ref{assumption} to see that $M(-)_n$ and hence $M(-)$ is left-exact. Since $p_*F_1(n)\in \cA_Y$ for $n\gg 0$, we have $H^1_Y(p_*F_1(n))=0$ for $n\gg 0$, so
\[
0\lra M(F_1)_n\lra M(F_2)_n\lra M(F_3)_n\lra 0
\]
is an exact sequence in $\cA$ for $n\gg 0$. This completes the proof of the first statement. For $F\in D^{[0,0]}$, part \one\ is stated and proved as a claim in the course of \cite[Proof of Theorem~2.3.6(2)]{AP06}. For \two, let $f\colon F'\ra F$ denote the inclusion and set $F^{\prime\prime}:=\cok(f)\in D^{[0,0]}$. For $n\gg 0$, apply the exact functor $M(-)_n$ to the short exact sequence $0\ra F'\ra F\ra F^{\prime\prime} \ra 0$ in $D^{[0,0]}$ to obtain $M(F^{\prime\prime})_n = 0$. Since $F^{\prime\prime}\in D^{[0,0]}$, we have $p_*F^{\prime\prime}(n)= 0$ for $n\gg 0$, so we can twist to get $p_*F^{\prime\prime}(-i)=0$ for $0\leq i\leq r$. Lemma~\ref{lem:semiorthog}\one\ gives $F^{\prime\prime}=0$, so $F^\prime=F$ as required.
\end{proof}

\begin{proposition}
\label{prop:noetherian}
The sheaf of t-structures on $\cdc(\ypr)$ over $\PP^r$ from Theorem~\ref{thm:limitingtstructure} is Noetherian, i.e., for every open $U\subseteq \PP^r$, the heart of the t-structure on $\cdlc(Y\times U)$ is Noetherian. 
\end{proposition}
\begin{proof}
Consider first the case $U=\PP^r$. For $F\in D^{[0,0]}$, consider an increasing chain
\begin{equation}
\label{eqn:chainFs}
F_1 \subseteq F_2 \subseteq F_3 \subseteq \cdots \subseteq F
\end{equation}
of subobjects in $D^{[0,0]}$. Applying the left-exact functor $M(-)$ gives an increasing chain of subobjects in the category of graded $S$-modules of finitely type in $\cA$. This latter chain stabilises by Theorem~\ref{thm:abcats}, so Lemma~\ref{lem:leftexactfunctor}\two\ implies that \eqref{eqn:chainFs} also stabilises as required.

For arbitrary open $U\subseteq \PP^r$, suppose \eqref{eqn:chainFs} is an increasing chain of objects in the heart $D^{[0,0]}$ of the bounded t-structure on $\cdlc(Y\times U)$ given by Theorem~\ref{thm:limitingtstructure}. Each $F_i$ is a complex of sheaves, so an inclusion $F_i\subseteq F_{i+1}$ is equality if the restriction to each open subset in an open cover is equality. Thus, we may assume $U$ is affine with complement $D$ given by a section of $\cO_{\PP^r}(d)$ for some $d\in \ZZ$. Note that $F$ and hence each $F_i$ has left-compact support. Corollary~\ref{cor:essential-surjective} gives an extension $\bar{F}\in\cdc(\ypr)$ of $F$. Restriction to the open subset $Y\times U$ is t-exact, so $H^0(\bar{F})|_{Y\times U}=H^0(\bar{F}|_{Y\times U})=F$. Replacing $\bar{F}$ by $H^0(\bar{F})$ if necessary, we may assume $\bar{F}\in D^{[0,0]}$, and similarly for each $\bar{F_i}$. The result of \cite[Lemma 2.1.8]{AP06} extends verbatim to the case of left-compact support, so we may replace $\bar{F_i}$ by $\bar{F_i}(-k_iD)$ and hence assume that the injection $F_i \hookrightarrow  F$ extends to a morphism $\phi_i\colon \bar{F_i}\lra \bar{F}$. The restriction to $Y\times U$ is unchanged by this, as is the property of having left-compact support, so we obtain an increasing chain 
\begin{equation}
\label{eqn:2nd-increasing}
\bar{F_1} \subseteq \bar{F_2} \subseteq\bar{F_3}\subseteq  \cdots \subseteq \bar{F}
\end{equation}
of objects in $D^{[0,0]}_c(\ypr)$ satisfying $\overline{F_i}|_{Y\times U}=F_i$ for all $i>0$. The sequence \eqref{eqn:2nd-increasing} stabilises by the case $U=\PP^r$ above, and restricting this chain to $Y\times U$ shows that \eqref{eqn:chainFs} also stabilises.
\end{proof}

\section{Sheaf of t-structures over an arbitrary base}
In this section we follow closely the approach of Polishchuk~\cite{P07} in extending the construction of the sheaf of t-structures on $\cdc(\ypr)$ over $\mathbb{P}^r$ to an arbitrary base scheme $S$ that is separated and of finite type. We then extend the work of Abramovich--Polishchuk~\cite{AP06} to show these these t-structures satisfy the open heart property.

\subsection{Extending t-structures to the quasi-coherent setting}
 Let $D$ be a triangulated category. A full subcategory $P$ is a \emph{pre-aisle} if $P$ is closed under extensions and the shift functor $X \to X[1]$ for any $X \in P$. For any subcategory  $S \subseteq D$, the \emph{pre-aisle generated by $S$} is the smallest pre-aisle containing $S$, denoted $\pa_D[S]$. A full subcategory $P \subseteq D$ is an \emph{aisle} if $P = D^{\leq 0}$ for some t-structure $(D^{\leq 0}, D^{\geq 0})$ on $D$. Every aisle is a pre-aisle, but the converse is false in general; see \cite[Remark of Section~2.1]{P07}. If we assume further that $D$ is a triangulated category in which all small coproducts exist, then a pre-aisle $P$ is \emph{cocomplete} if it is closed under small coproducts. For any subcategory $S \subseteq D$, the \emph{cocomplete pre-aisle generated by $S$} is the smallest cocomplete pre-aisle containing $S$, denoted by $\pa_D[[S]]$. 

\begin{lemma}
\label{lem:t-str-qcy}
Let $Y$ and $S$ be Noetherian schemes. Any t-structure $(D^{\leq 0},D^{\geq 0})$ on $\cdlc(Y\times S)$ can be extended to a t-structure on $\DD(\Qcoh(Y\times S))$ that satisfies
\begin{align}
\DD^{\leq 0}(\Qcoh(Y\times S)) &= \pa_{\DD(\Qcoh(Y\times S))}[[D^{\leq 0}]]; \label{eqn:left-qcoh-U}\\
\DD^{\geq 0}(\Qcoh(Y\times S)) &= \big\{ F \in \DD(\Qcoh(Y\times S)) \mid\Hom (\DD^{\leq -1}, F)=0 \big\}.\label{eqn:right-qcoh-U}
\end{align}
Furthermore, for every interval $[a,b]$ that may be infinite on one side, we have that
\begin{equation}
\label{eqn:intersection}
\DD^{[a,b]}\big(\!\Qcoh(Y\times S)\big) \cap \cdlc(Y\times S) = \cdlc^{[a,b]}(Y\times S).
\end{equation}
\end{lemma}
\begin{proof}
All small coproducts exist in $\DD(\Qcoh(Y\times S))$ by \cite[Ex~1.3]{Ne96a}. Since $\cdlc(Y\times S)$ is an essentially small full subcategory of $\DD(\Qcoh(Y\times S))$, we're done by \cite[Lemma~2.1.1]{P07}.
\end{proof}

\begin{remark}
\label{rem:left-inf-sum}
Equation \eqref{eqn:intersection} implies that if two t-structures on $\cdlc(Y\times S)$ extend to the same t-structure on $\DD(\Qcoh(Y\times S))$, then the original t-structures are equal.
\end{remark}

\subsection{Sheaf of t-structures over an affine base}
\label{sec:affinebase}
 Given a Noetherian, bounded t-structure on $D_c(Y)$, Theorem~\ref{thm:limitingtstructure} gives a sheaf of t-structures on $\cdc(\ypr)$ over $\PP^r$. We now study the restriction of this t-structure to subcategories $\cdlc(Y\times U)$, where $U$ is an affine scheme of finite type over $\kk$. The main result of this section is the following:

\begin{proposition}
\label{prop:t-str-affine}
Let $U, Y$ be schemes of finite type, with $Y$ separated and $U$ affine. Extend the t-structure on $\cdc(Y)$ from Assumption~\ref{assumption} to a t-structure on $\DD(\Qcoh(Y))$ using Lemma~\ref{lem:t-str-qcy}. There exists a sheaf of Noetherian t-structures on $\cdlc(Y \times U)$ over $U$, such that for any interval $[a,b]$ that may be infinite on one side, we have
\begin{equation}
\label{eqn:t-str-U-formula}
\dlabc(Y \times U)=\left\{ F \in \cdlc(Y \times U) \mid p_*F \in \DD^{[a,b]}(\Qcoh(Y)) \right\},
\end{equation}
 where $p\colon Y\times U\to Y$ is the projection to the first factor. Moreover, $p^*\colon \cdc(Y) \ra \cdlc(Y \times U)$ is t-exact with respect to these t-structures.
\end{proposition}

 We prove this result in two stages. We first restrict along an open immersion $Y\times \AA^r\to Y\times \PP^r$ to prove the special case $U=\AA^r$; our proof runs parallel to that of \cite[Lemma 3.3.2-3.3.4]{P07}. We then restrict further along a closed immersion $Y\times U\to Y\times \AA^r$ to prove the general case following \cite[Proof of Theorem~3.3.6]{P07}.

\begin{proof}[Proof of Proposition~\ref{prop:t-str-affine} for $\AA^r$]
Let $j\colon\mathbb{A}^r\to \mathbb{P}^r$ be an open immersion, and let $p\colon Y\times \AA^r\to Y$ denote the first projection. The local nature of the sheaf of t-structures on $D_c(\ypr)$ over $\PP^r$ from Theorem~\ref{thm:limitingtstructure} (and Proposition~\ref{prop:noetherian}) induces a sheaf of Noetherian t-structures on $\cdlc(Y\times \mathbb{A}^r)$ over $\mathbb{A}^r$ such that $(\id_Y \times j)^*$ is t-exact. The projection formula shows that pullback along the first projection $p'\colon \ypr \to Y$ is t-exact, hence so is $p^*=(\id_Y \times j)^* \circ p'^*$. It remains to show that the t-structure on $\cdlc(Y\times \mathbb{A}^r)$ satisfies \eqref{eqn:t-str-U-formula}. We proceed in three steps:
 
\smallskip

\indent \textsc{Step 1}: We first claim that
\begin{equation}
\label{eqn:prep-Ar}
\cdlcn(Y \times \AA^r) = \pa [ p^*\cdcn(Y) ].
\end{equation}
For one inclusion, let $F \in \cdlcn(Y \times \AA^r)$ and write $F = (\id_Y \times j)^*G$ for some $G \in \cdcn(Y \times \PP^r)$. Let $N\in \NN$ satisfy $G_n:=p'_*G(n) \in \cdcn(Y)$ for $n \geq N$. Deduce from Remark~\ref{rem:orderswitch} that $G$ lies in the extension closure of $\{ p'^*G_{N+i} \otimes \Omega^{r-i}(-N-i)\}_{0\leq i\leq r}$. Pull back to $Y \times \AA^r$ to see that $F$ lies in the extension closure of $\{p^*G_{N+i}\otimes E_i\}_{0\leq i\leq r}$, where $E_i=(\id_Y \times j)^*(\Omega^{r-i}(-N-i))$ is a trivial bundle of finite rank on $Y \times \AA^r$ for all $0 \leq i \leq r$. Thus $F \in \pa [ p^*\cdcn(Y) ]$. For the opposite inclusion,  we need only show that $p^*\cdcn(Y)\subseteq \cdlcn(Y \times \AA^r)$, which follows from the t-exactness of $p^*$ shown above.

\smallskip

\indent \textsc{Step 2}: We deduce \eqref{eqn:t-str-U-formula} for $\AA^r$ by applying Lemma~\ref{lem:zero-kernel} to $p_*\colon \cdlc(Y \times \AA^r) \ra \DD(\Qcoh(Y))$, and for this we must check that $p_*$ is t-exact and has trivial kernel. For left t-exactness, let $F \in \cdcn(Y)$. We know $\DD^{\leq 0}(\Qcoh(Y))$ is closed under small coproducts by \eqref{eqn:left-qcoh-U}. The projection formula gives $p_*p^*F = F \otimes_\kk H^0(\AA^r, \cO_{\AA^r})\in \DD^{\leq 0}(\Qcoh(Y))$ and hence $p_*p^*\cdcn(Y) \subseteq \DD^{\leq 0}(\Qcoh(Y))$. Since the image of a pre-aisle under $p_*$ is a pre-aisle and since $\DD^{\leq 0}(\Qcoh(Y))$ is itself a pre-aisle, we deduce from Step 1 above that
\begin{equation*}
p_*\cdlcn(Y \times \AA^r) = p_* \pa [p^*\cdcn(Y)] \subseteq \DD^{\leq 0}(\Qcoh(Y))
\end{equation*}
as required. For right t-exactness, use Lemma~\ref{lem:t-str-qcy}, adjunction and Step 1 to obtain
\begin{align}
\cdlc^{\geq 0}(Y \times \AA^r) &= \{ F \mid \Hom( \cdlc^{\leq -1}(Y \times \AA^r), F)=0 \} &\nonumber \\
&= \{ F \mid \Hom( \pa [ p^*D_c^{\leq -1}(Y) ], F)=0 \} & \nonumber\\
&= \{ F \mid \Hom( p^*D_c^{\leq -1}(Y), F)=0 \} & \nonumber \\
&= \{ F \mid \Hom( D_c^{\leq -1}(Y), p_*F)=0 \}. &\nonumber 
\end{align}
 Now $p_*\cdlcp(Y\times \AA^r)\subseteq \DD^{\geq 0}(\Qcoh(Y))$ by \eqref{eqn:right-qcoh-U}, so $p_*$ is right t-exact. Finally, if $F \in\cdlc(Y \times \AA^r)$ satisfies $p_*F = 0$, then $F = 0$ by \cite[Tag 08I8]{stacks-project} because $p\colon Y \times \AA^r \ra Y$ is affine. 

\smallskip

\noindent This establishes Proposition~\ref{prop:t-str-affine} for $U=\AA^r$.
\end{proof}

 The proof of the general case relies on the following result which extends to our setting the statement and proof of \cite[Theorem~2.3.5]{P07}.

\begin{lemma}
\label{lem:t-str-closed}
Let $Y, S$ and $S'$ be separated schemes of finite type,  and let $g\colon S' \ra S$ be a finite morphism of finite Tor dimension. Let $(\cdlc^{\leq 0}(Y \times S), \cdlc^{\geq 0}(Y \times S))$ be a Noetherian t-structure on $\cdlc(Y \times S)$. There exists a Noetherian t-structure on $\cdlc(Y \times S')$, where for any interval $[a,b]$ that may be infinite on one side, we have
\begin{equation}
\label{eqn:finTordim}
\dlabc(Y \times S') = \left\{ F \in D(Y \times S') \mid (\id_Y \times g)_*F \in \dlabc(Y \times S) \right\}.
\end{equation}
\end{lemma}
\begin{proof}
 Any object $F \in D(Y \times S')$ satisfying $(\id_Y \times g)_*F \in \cdlc(Y \times S)$ has left-compact support by Lemma~\ref{lem:Flcsupp1}, so the right-hand side of \eqref{eqn:finTordim} is contained in $\cdlc(Y\times S^\prime)$. The functors 
\[
\Phi:=(\id_Y \times g)_* \colon \cdlc(Y \times S') \lra \cdlc(Y \times S) \text{ and }\Psi:=(\id_Y \times g)^* \colon \cdlc(Y \times S) \lra \cdlc(Y \times S')
\]
 are well-defined by Lemma~\ref{lem:cpt-supp}. We check that the hypotheses of \cite[Theorem 2.1.2]{P07} hold:
 \begin{enumerate}
\item $\Phi$ is obtained by restriction from $(\id_Y \times g)_*\colon \DD(\Qcoh(Y\times S')\to \DD(\Qcoh(Y \times S))$, and $\Psi$ by restriction from its left adjoint $(\id_Y \times g)^*\colon\DD(\Qcoh(Y \times S))\to \DD(\Qcoh(Y\times S'))$.
\item $(\id_Y \times g)_* \colon \DD(\Qcoh(Y \times S')) \ra \DD(\Qcoh(Y \times S))$ commutes with small coproducts by \cite[Lemma 1.4]{Ne96a} because $Y\times S^\prime$ is quasi-compact and separated.
\item $\Phi\circ \Psi$ is right t-exact with respect to the given t-structure. For this, let $F\in \cdlc(Y\times S)$. The projection formula gives $(\Phi\circ \Psi)(F)\cong F\otimes (\id_Y \times g)_*\cO_{Y \times S'}$, so we need only show that tensoring by $(\id_Y \times g)_*\cO_{Y \times S'}$ is right t-exact. Since $g$ has finite Tor dimension, $(\id_Y \times g)_*\cO_{Y \times S'}$ is quasi-isomorphic to a bounded complex of flat sheaves, and the result follows as in the proof of \cite[Theorem 2.3.5]{P07}.
\item if $F \in \DD(\Qcoh(Y \times S'))$ satisfies $(\id_Y \times g)_*F \in \cdlc(Y \times S)$, then we have $F \in\cdlc(Y \times S')$ by Lemma~\ref{lem:Flcsupp1}.
\item if $F \in \cdlc(Y \times S')$ satisfies $(\id_Y \times g)_*F = 0$, then we have $F = 0$ by \cite[Tag 08I8]{stacks-project}. 
\end{enumerate}
The result now follows from \cite[Theorem 2.1.2]{P07}. 
\end{proof}

\begin{remark} \label{rem:applylemtstrclosed}
We will apply Lemma \ref{lem:t-str-closed} when $S$ is smooth, or when $g$ is the inclusion of an
effective Cartier divisor $T \subset S$. In the former case,
the morphism is of finite Tor dimension as every bounded complex of coherent sheaves on a smooth scheme admits a locally free resolution by \cite[Corollary 19.8]{Eis95}, and in the latter case as $\cO_S(-T) \to \cO_S$ is a finite locally free resolution of $\cO_T$.
\end{remark}

\begin{proof}[Proof of Proposition~\ref{prop:t-str-affine} in general] Let $U$ be an affine scheme of finite type. Choose a closed immersion $i\colon U \hookrightarrow \AA^r$, and write $p\colon Y\times U\ra Y$ and $p^\prime\colon Y\times\AA^r\ra Y$ for the first projections, so $p=p^\prime\circ (\id_Y\times i)$.
 
 We first show that \eqref{eqn:t-str-U-formula} defines a sheaf of Noetherian t-structures on $\cdlc(Y\times U)$ over $U$. Applying Lemma~\ref{lem:t-str-closed} (see Remark \ref{rem:applylemtstrclosed}) to the t-structure on $\cdlc(Y\times \AA^r)$ constructed in the special case above gives a t-structure on $\cdlc(Y \times U)$, satisfying 
\begin{align*}
\dlabc(Y \times U) &= \big\{ F \in \cdlc(Y \times U) \mid (\id_Y \times i)_*(F) \in \dlabc(Y \times \AA^r) \big\} \\
&= \big\{ F \in \cdlc(Y \times U) \mid p_*(F) \in \DD^{[a,b]}(\Qcoh(Y))\big\}
\end{align*}
 for any interval $[a,b]$ that may be infinite on one side. This gives \eqref{eqn:t-str-U-formula}. Corollary~\ref{cor:affine-sheaf} and Proposition~\ref{prop:noetherian} imply that we obtain a sheaf of Noetherian t-structures on $\cdlc(Y\times U)$ over $U$.
  
  To prove that $p^*$ is t-exact, let $F \in \cdc^{[a,b]}(Y)$. By \eqref{eqn:t-str-U-formula} and the projection formula, we must show that $p_* p^* F = F \otimes_\kk H^0(U, \cO_U) \in \DD^{[a,b]}(\Qcoh(Y))$. For this, the $\kk$-vector space $V$ underlying $H^0(U, \cO_U)$ has a countable basis. If $V$ has a finite $\kk$-basis, then clearly $F \otimes_\kk V \in \dabc(Y) \subseteq \DD^{[a,b]}(\Qcoh(Y))$. Otherwise, $V$ has a countably infinite $\kk$-basis. The trick \cite[Lemma~3.3.5]{P07} is to observe that $V \cong H^0(\AA^1, \cO_{\AA^1})$ as a $\kk$-vector space. Proposition~\ref{prop:t-str-affine} holds for $U=\AA^1$, so pullback along the first projection $p^{\prime\prime}\colon Y \times \AA^1\ra Y$ is t-exact and hence ${p''}^*(F) \in \dlabc(Y \times \AA^1)$.  The projection formula and \eqref{eqn:t-str-U-formula} for $U=\AA^1$ give 
\[
F \otimes_\kk V = F \otimes_\kk H^0(\AA^1, \cO_{\AA^1}) \cong {p''}_* {p''}^*(F) \in \DD^{[a,b]}(\Qcoh(Y)).
\]
This completes the proof that $p^*$ is t-exact and hence concludes the proof of Proposition~\ref{prop:t-str-affine}.
\end{proof}

\subsection{Construction over an arbitrary base}
We are now in a position to establish the first main result of this section following \cite[Theorem 3.3.6]{P07}. 

\begin{theorem}
\label{thm:t-str-any}
Let $S, Y$ be separated schemes of finite type. Extend the t-structure on $\cdc(Y)$ from Assumption~\ref{assumption} to a t-structure on $\DD(\Qcoh(Y))$ using Lemma~\ref{lem:t-str-qcy}. Then there is a sheaf of Noetherian t-structures on $\cdlc(Y \times S)$ over $S$ such that for any interval $[a,b]$ that may be infinite on one side, we have
\begin{equation}
\label{eqn:t-str-S-formula}
\dlabc(Y \times S)=\left\{ F \in \cdlc(Y \times S) \mid \begin{array}{cc} p_*(F|_{Y \times U})  \in \DD^{[a,b]}(\Qcoh(Y))\\ \text{for any open affine }U\subseteq S\end{array}\right\}
\end{equation}
where we abuse notation by writing $p$ for the first projection from $Y \times U$. Moreover
\begin{enumerate}
\item[\one] for any finite open affine covering $S=\bigcup_i U_i$, we have that 
\[
F \in \dlabc(Y \times S)\iff p_*(F|_{Y \times U_i}) \in \DD^{[a,b]}(\Qcoh(Y))\text{ for every }i; 
\]
\item[\two] the functor $p^*\colon \cdc(Y) \ra \cdlc(Y \times S)$ is t-exact with respect to these t-structures.
 \item[\three] Assume in addition that $S$ is projective. Then this t-structure satisfies 
\begin{equation}
\label{eqn:t-str-proj}
F \in D_c^{[a,b]}(Y \times S)\iff p_*(
F \otimes  q^*L^n) \in \dabc(Y) \quad \text{ for } n \gg 0,
\end{equation}
where $L$ is any ample bundle on $S$ and $q\colon Y \times S\ra S$ is the second projection.
\end{enumerate}
\end{theorem}

 Before the proof we present a compatibility result for open immersions of affine schemes. This extends to our setting a statement from the proof of \cite[Theorem~3.3.6]{P07}.

\begin{lemma}
\label{lem:compatibility}
Let $j\colon U_1 \hookrightarrow U_2$ be an open immersion of affine schemes of finite type. Then 
\begin{equation*}
(\id_Y \times j)^*\colon \cdlc(Y \times U_2) \lra \cdlc(Y \times U_1)
\end{equation*}
is t-exact with respect to the t-structures on both sides given by \eqref{eqn:t-str-U-formula}.  
\end{lemma}
\begin{proof}
 We have all the elements in place to reproduce the proof of this statement from \cite[Proof of Theorem~3.3.6]{P07} so we provide only an outline. Extend the t-structures on $\cdc(Y)$ and $\cdlc(Y\times U_i)$ to t-structures on $\DD(\Qcoh(Y))$ and $\DD(\Qcoh(Y\times U_i))$ by Lemma~\ref{lem:t-str-qcy}. 

\smallskip

\indent \textsc{Step 1}: For any interval $[a,b]$ that may be infinite on one side, we prove an analogue of \eqref{eqn:t-str-U-formula}, namely that
\begin{equation*}
\label{eqn:t-str-qcoh-U}
\DD^{[a,b]}(\Qcoh(Y \times U_2)) = \big\{ F \in \DD(\Qcoh(Y \times U_2)) \mid {p_2}_*F \in \DD^{[a,b]}(\Qcoh(Y)) \big\},
\end{equation*}
 where $p_i\colon Y\times U_i\to Y$ is the projection to the first factor. We have $\ker {p_2}_* =0$ by \cite[Tag 08I8]{stacks-project}, so it suffices by Lemma~\ref{lem:zero-kernel} to show that ${p_2}_*\colon \DD(\Qcoh(Y \times U_2)) \ra \DD(\Qcoh(Y))$ is t-exact. This follows exactly as in the proof of \cite[Theorem~3.3.6]{P07}.

\smallskip

\indent \textsc{Step 2}: We now claim that
\begin{equation}
\label{eqn:pullbackQcohAffine}
\DD^{\leq 0}(\Qcoh(Y \times U_1))=(\id_Y \times j)^* \DD^{\leq 0}(\Qcoh(Y \times U_2)).
\end{equation}
Indeed, for $F \in \DD^{[a,b]}(\Qcoh(Y \times U_1))$, we have ${p_2}_*((\id_Y \times j)_* F) = {p_1}_*F \in \DD^{[a,b]}(\Qcoh(Y))$, and hence $(\id_Y \times j)_* F \in \DD^{[a,b]}(\Qcoh(Y \times U_2))$ by \textsc{Step 1}.
Equation \eqref{eqn:right-qcoh-U} and adjunction now imply
\[
\DD^{\leq 0}(\Qcoh(Y \times U_2)) \subseteq \{ F \mid \Hom((\id_Y \times j)^*F, \DD^{\geq 1}(\Qcoh(Y \times U_1)))=0 \},
\]
so the right-hand side of \eqref{eqn:pullbackQcohAffine} is contained in the left. The opposite inclusion follows because each $F\in \DD^{\leq 0}(\Qcoh(Y \times U_1))$ satisfies $F = (\id_Y \times j)^*(\id_Y \times j)_*F\in (\id_Y \times j)^* \DD^{\leq 0}(\Qcoh(Y \times U_2))$.

\smallskip

\indent \textsc{Step 3}: On the one hand, Lemma~\ref{lem:t-str-qcy} shows that $\cdlc^{\leq 0}(Y \times U_1)$ is an aisle in $\cdlc(Y \times U_1)$ which extends to a t-structure on $\DD(\Qcoh(Y\times U_1))$ that satisfies
\begin{equation}
\label{eqn:1st-gen}
\DD^{\leq 0}(\Qcoh(Y \times U_1)) = \pa_{\DD(\Qcoh(Y \times U_1))} [[ \cdlc^{\leq 0}(Y \times U_1) ]].
\end{equation}
On the other hand, since pulling back along $\id_Y \times j$ commutes with extensions and left shifts, and respects coproducts (see \cite[Proposition~1.21]{Neemanbook}), \textsc{Step 2} combined with the analogue of \eqref{eqn:1st-gen} for $U_2$ gives
\begin{equation}
\DD^{\leq 0}(\Qcoh(Y \times U_1)) = \pa_{\DD(\Qcoh(Y \times U_1))} [[ (\id_Y \times j)^* \cdlc^{\leq 0}(Y \times U_2) ]]. \label{eqn:2nd-gen}
\end{equation}
Equation~\eqref{eqn:open-restriction-t-str} shows that $(\id_Y \times j)^* \cdlc^{\leq 0}(Y \times U_2)$ is also an aisle in $\cdlc(Y \times U_1)$ which by \eqref{eqn:2nd-gen} extends to the same t-structure on $\DD(\Qcoh(Y\times U_1))$. 
Comparing \eqref{eqn:1st-gen} and \eqref{eqn:2nd-gen}, Remark~\ref{rem:left-inf-sum} implies that these t-structures coincide as desired.
\end{proof}

\begin{proof}[Proof of Theorem~\ref{thm:t-str-any}]
 Given a finite open affine cover $S=\bigcup_i U_i$, Proposition~\ref{prop:t-str-affine} determines a sheaf of t-structures on $\cdlc(Y\times U_i)$ over $U_i$ for each $i$. For any pair $i, j$ and for any open affine $V \subseteq (U_i \cap U_j)$ of finite type, restricting from $\cdlc(Y \times U_i)$ and $\cdlc(Y \times U_j)$ gives two a priori different sheaves of t-structures on $\cdlc(Y \times V)$. These agree by Lemma~\ref{lem:compatibility}, and  
Lemma~\ref{lem:gluing-property} gives a unique sheaf of t-structures on $\cdlc(Y\times S)$ over $S$ that, by \eqref{eqn:t-str-U-formula}, is characterised by the condition from Theorem~\ref{thm:t-str-any}\one. This t-structure is independent of the choice of open cover because property \eqref{eqn:t-str-S-formula} follows from Theorem~\ref{thm:t-str-any}\one, Proposition~\ref{prop:t-str-affine}, Lemma~\ref{lem:compatibility} and Lemma~\ref{lem:gluing-property}. 

To see that this t-structure is Noetherian, the restriction of any ascending chain of subobjects in $\cdlc^{[0,0]}(Y \times S)$ is an ascending chain of subobjects in $\cdlc^{[0,0]}(Y \times U_i)$. This latter chain stabilises by Proposition~\ref{prop:t-str-affine}, and since the open cover is finite, the original chain stabilises. 

Finally, to show that $p^*$ is t-exact, write $p_i\colon Y\times U_i\to Y$ and $p\colon Y\times S\to Y$ for the first projections. If $F \in \cdc^{[a,b]}(Y)$, then Proposition~\ref{prop:t-str-affine}  gives $p^*(F)\vert_{Y\times U_i} = p_i^*F \in \dabc(Y \times U_i)$ for each $i$, and hence $p^*F \in \cdlc^{[a,b]}(Y \times S)$ because we have a sheaf of t-structures over $S$.

For the final statement, let $S$ be projective. Recall from Proposition~\ref{prop:proj-cpt-supp} that $F\in D(Y\times S)$ has compact support if and only if $p_*(F\otimes q^*L^n)\in D_c(Y)$ for $n\gg 0$. First, reduce to the case when $L$ is very ample using Lemma~\ref{lem:tensorlinebundle}. In this case, let $i\colon S\to \PP^r$ be the closed immersion such that $L=i^*\cO(1)$. Given the sheaf of t-structures on $\cdc(Y\times \PP^r)$ over $\PP^r$ from Theorem~\ref{thm:limitingtstructure}, pullback along the closed immersion $\id_Y\times i$ using Lemma~\ref{lem:t-str-closed} and Remark \ref{rem:applylemtstrclosed} to obtain a sheaf of Noetherian t-structures on $\cdc(Y\times S)$ over $S$. To see that this coincides with the above t-structure, take an open affine cover of $S$, confirm that Theorem \ref{thm:t-str-any}\one\ holds, and apply Lemma~\ref{lem:gluing-property}.
 \end{proof}
 
\subsection{The open heart property}
Let $Y$ and $S$ be a separated schemes of finite type, and consider the sheaf of Noetherian t-structures on $\cdlc(Y \times S)$ over $S$ from Theorem~\ref{thm:t-str-any}. For any subset $V \subseteq S$ that is either open or closed,  let $\mathcal{A}_V$ denote the heart of the induced t-structure on $\cdlc(Y \times V)$ and let $H^i_V(F)$ denote the $i$-th cohomology of an object $F \in \cdlc(Y \times V)$. 

\begin{lemma}
\label{lem:-1-0-exact}
Let $T\subset S$ be an effective Cartier divisor. Any object $F \in \cdlc^{[a,b]}(Y \times S)$ satisfies $F|_{Y \times T} \in \cdlc^{[a-1,b]}(Y \times T)$.
\end{lemma}

\begin{proof}
Let $i\colon T \hookrightarrow S$ denote the closed immersion, and let $f\in H^0(\mathcal{O}_S(T))$ be a defining section for $T$. Lemma~\ref{lem:tensorlinebundle} implies that $F\otimes q^*\mathcal{O}_S(-T)\in \cdlc^{[a,b]}(Y \times S)$. Since 
\begin{equation}
\label{eqn:timesf}
F \otimes q^*\cO_S(-T) \stackrel{\cdot f}{\longrightarrow} F\longrightarrow (\id\times i)_*(F\vert_{Y\times T})
\end{equation}
is an exact triangle in $D_{lc}(Y \times S)$, it follows from the cone construction that $(\id\times i)_*(F\vert_{Y\times T})\in \cdlc^{[a-1,b]}(Y \times S)$. The result follows from Lemma~\ref{lem:t-str-closed}, see again Remark \ref{rem:applylemtstrclosed}.
\end{proof}

Recall from \cite[Definition 3.1.1]{AP06} that an object $F\in \cA_S$ is \emph{$S$-torsion} if it is the pushforward of an object $E\in D(Y\times T)$ for some closed subscheme $T$ in $S$. Equivalently, for every effective Cartier divisor $D\subset S$ containing $T$ with defining section $f\in H^0(\mathcal{O}_S(D))$, there is an integer $k$ such that the morphism $f^k\colon F\to F\otimes q^*\mathcal{O}_S(kD)$ is zero. We say that $F\in \cA_S$ is \emph{$S$-torsion-free} if it contains no nonzero $S$-torsion subobject.  The next result follows \cite[Lemma 3.3.4]{AP06}.

\begin{lemma}
\label{lem:heart-restrict-zero}
Let $T\subset S$ be an effective Cartier divisor, and let $E \in \cA_S$. If $H^0_T(E|_{Y \times T})=0$ then there is an open neighborhood $T \subseteq U \subseteq S$ such that $E|_{Y \times U}=0$.
\end{lemma}

\begin{proof}
The support of $E$ is closed, so it suffices to prove that $E|_{Y \times T}=0$. Let $f\in H^0(\cO_S(T))$ be a defining section for $T$. Since the abelian category $\cA_S$ is Noetherian, there is a maximal $S$-torsion subobject $E^{tor}\subset E$ supported in $T$ and a short exact sequence 
$$ 0 \lra E^{tor} \lra E \lra F \lra 0 $$
in $\cA_S$, where $F$ has no torsion subobject with support in $T$. By restricting to $Y\times T$ and applying Lemma~\ref{lem:-1-0-exact}, we obtain an exact triangle 
\begin{equation}
\label{eqn:EtorEF}
 E^{tor}|_{Y \times T} \lra E|_{Y \times T} \lra F|_{Y \times T}, 
 \end{equation}
all of whose terms lie in $\cdlc^{[-1,0]}(Y\times T)$. It suffices to prove that $F|_{Y \times T}= 0 = E^{tor}|_{Y \times T}$.

 First consider $F|_{Y\times T}$. Since $H^0_T(E|_{Y \times T})=0$, the cohomology sequence for \eqref{eqn:EtorEF} implies that $H^0_T(F|_{Y \times T})=0$. We claim that $H^{-1}_T(F|_{Y\times T})=0$, in which case $F|_{Y\times T}=0$. For this, the morphism $F\otimes q^*\cO_S(-T)\ra F$ in the category $\cA_S$ is injective because $F$ has no torsion subobject  with support in $T$. It follows from \eqref{eqn:timesf} that $(\id\times i)_*(F|_{Y\times T})$ is the cokernel and hence also lies in $\cA_S$. We obtain $F\vert_{Y\times T}\in \cA_S$ by Lemma~\ref{lem:t-str-closed}, so $H^{-1}_T(F|_{Y\times T})=0$ which proves the claim.

 It remains to show that $E^{tor}|_{Y \times T}=0$. Since $F|_{Y\times T}=0$, we may assume from the beginning that $E$ is $S$-torsion with support in $T$. Let $k \geq 0$ be the minimal value such that $E$ is annihilated by $f^k$. Lemma \ref{lem:-1-0-exact} and the assumption $H^0_T(E|_{Y \times T})=0$ imply that $E|_{Y \times T} \in \cdlc^{[-1,-1]}(Y \times T)$, so $(\id \times i)_*(E|_{Y \times T}) \in \cdlc^{[-1,-1]}(Y \times S)$ by Lemma~\ref{lem:t-str-closed}. In particular, $H_S^0((\id \times i)_*(E|_{Y \times T})) =0$, so the cohomology sequence for the exact triangle
\[
 E \otimes q^*\cO_S(-T) \stackrel{\cdot f}{\lra} E \lra (\id \times i)_*(E|_{Y \times T}) 
 \]
 shows that the morphism $E \otimes q^*\cO_S(-T) =H^0_S(E \otimes q^*\cO_S(-T))\stackrel{\cdot f}{\lra} H^0_S(E) = E$ is surjective. If $k>0$, then $f^{k-1}$ annihilates $f(E\otimes q^*\cO_S(-T))=E$ which contradicts minimality of $k$, so $k=0$ and hence $E|_{Y\times T}=0$.
\end{proof}

The next result extends \cite[Proposition 3.3.2]{AP06} and \cite[Proposition 2.3.7]{P07} to our setting.

\begin{proposition}[The open heart property]
\label{prop:open-heart}
Let $Y$ and $S$ be separated schemes of finite type, and let $T\subset S$ be a local complete intersection. Let $F \in \cdlc(Y \times S)$. If $F|_{Y \times T} \in \cA_T$, then there is an open neighborhood $T \subseteq U \subseteq S$ such that $F|_{Y \times U} \in \cA_U$.
\end{proposition}

\begin{proof}
It suffices to prove the statement under the additional assumption that $T$ is an effective Cartier divisor in $S$, as an induction on the codimension of $T$ in $S$ proves the general case. Let $a,b\in \ZZ$ be such that $F\in \cdlc^{[a,b]}(Y\times S)$ with $b>0$. We proceed in two steps:

\smallskip

\noindent \textsc{Step 1}:  We find an open neighbourhood $T \subseteq U \subseteq S$ such that $F|_{Y \times U} \in \cdlc^{[a,b-1]}(Y\times U)$, and hence by induction, shrinking $U$ at each step if necessary, we deduce that $F|_{Y \times U} \in \cdlc^{[a,0]}(Y\times U)$. For this, restrict to $Y\times T$ a truncation exact triangle for $F$ to obtain an exact triangle
\begin{equation}
\label{eqn:b-1restrict}
\tau^{\leq b-1}F|_{Y \times T} \lra F|_{Y \times T} \lra H^b_S(F)[-b]|_{Y \times T}.
\end{equation}
Lemma \ref{lem:-1-0-exact} gives $\tau^{\leq b-1}F|_{Y \times T} \in \cdlc^{[a-1,b-1]}(Y \times T)$, so the long exact sequence in cohomology for \eqref{eqn:b-1restrict} gives
\[
H^0_T\big(H^b_S(F)|_{Y \times T}\big) = H^b_T\big(H^b_S(F)|_{Y \times T}[-b]\big) = H^b_T(F|_{Y \times T}) = 0
\]
because $b>0$. Applying Lemma~\ref{lem:heart-restrict-zero} gives an open neighbourhood $T\subseteq U\subseteq S$ such that $H^b_S(F)\vert_{Y\times U}=0$, and hence $F|_{Y \times U} \in \cdlc^{[a,b-1]}(Y\times U)$ as required. 

\smallskip

We have proven $F|_{Y \times U} \in \cdlc^{[a,0]}(Y \times U)$. If $a=0$ the claim is proved. Otherwise, we can replace $S$ by $U$ and hence have $F \in \cdlc^{[a,0]}(Y \times S)$ for some $a < 0$.

\smallskip

\noindent \textsc{Step 2}: We show that, in this case, we also have $\tau^{\leq -1}F|_{Y \times T}=0$. On one hand, we have $\tau^{\leq -1}F|_{Y \times T} \in \cdlc^{[a-1,-1]}(Y \times T)$ by Lemma \ref{lem:-1-0-exact}. On the other hand, restricting a truncation exact triangle for $F$ to $Y\times T$ gives
\begin{equation}
\label{eqn:-1restrict}
\tau^{\leq -1}F|_{Y \times T} \lra F|_{Y \times T} \lra H^0_S(F)|_{Y \times T}.
\end{equation}
Applying Lemma \ref{lem:-1-0-exact} to $H^0_S(F)$ gives $H^0_S(F)|_{Y \times T} \in \cdlc^{[-1,0]}(Y \times T)$, and since $F|_{Y \times T} \in \cA_T$ holds by assumption, the exact triangle \eqref{eqn:-1restrict} shows that the object
\[
 \tau^{\leq -1}F|_{Y \times T} = \cone(F|_{Y \times T} \to H^0_S(F)|_{Y \times T})[-1]
 \]
 lies in $\cdlc^{[0,1]}(Y \times T)$. These two statements force $\tau^{\leq -1}F|_{Y \times T} = 0$ as claimed. 
 
 \smallskip
 
 To conclude, the support of $\tau^{\leq -1}F$ is closed in $S$, so there exists an open neighbourhood $T \subseteq U \subseteq S$ such that $\tau^{\leq -1}F|_{Y \times U} = 0$. Thus $F|_{Y\times U} = H^0_S(F)|_{Y\times U} \in \cA_U$ as required.
\end{proof}

\section{Numerical Bridgeland stability conditions for compact support}
\label{sect:stabilityconditions}

The goal of this section is to provide the right setting for stability conditions for objects with compact support on a non-compact quasi-projective variety $Y$. Note that the $K$-group of $D_c(Y)$ almost always has infinite rank (for example, skyscraper sheaves of points that do not lie on proper subvarieties of positive dimension have linearly independent classes), and yet the numerical $K$-group of $D_c(Y)$ is not defined when $Y$ is singular. Even when $Y$ is smooth, the class of skyscraper sheaves of points is 0, so $D_c(Y)$ is unlikely to admit numerical stability conditions (where
one requires that the central charge factors via the numerical Grothendieck group). To get around these problems, many authors (see Section \ref{subsec:moremotivation}) have instead considered stability conditions on $D_Z(Y)$, the derived category with objects supported on a proper subvariety $Z \subset Y$. However, this does not lead to moduli spaces of finite type: even the moduli space of skyscraper sheaves of points would be the completion of $Y$ at $Z$.

We therefore propose to use a variant of the numerical Grothendieck group of $D_c(Y)$, defined via the Euler pairing with perfect complexes.

\subsection{Numerical Grothendieck groups} \label{subsec:numericalKgroups}
Let $Y$ be a separated scheme of finite type over $\kk$. 
For any objects $E\in \Dperf(Y)$ and $F\in D_c(Y)$, the vector space $\bigoplus_i \Hom_{D(Y)}(E,F[i])$ is of finite dimension over $\kk$. The Euler form  
\[
\chi_Y\colon K(\Dperf(Y))\times K(D_c(Y))\to \ZZ
\]
between the Grothendieck groups of these categories is the bilinear form given by 
\begin{equation}
\label{eqn:eulerform}
\chi_Y(E,F)= \sum_{i\in \ZZ} (-1)^i\dim_\kk \Hom(E,F[i]).
\end{equation}
The quotient of $K(\Dperf(Y))$ and $K(D_c(Y))$ with respect to the kernel of $\chi$ on each factor defines the \emph{numerical Grothendieck groups} $\Knumperf(Y)$ and $\Knumc(Y)$ respectively, and we use the same notation  
\[
\chi_Y\colon \Knumperf(Y)\times \Knumc(Y)\to \ZZ
\]
 for the induced perfect pairing. Our interest lies in studying the category $D_c(Y)$ when $\Knumc(Y)$ has finite rank. Here we present a sufficient condition for this to hold.
 
\begin{lemma}
\label{lem:finiterank}
Let $Y$ be a normal, quasi-projective scheme of finite type over a field $\kk$ of characteristic zero. Then $\Knumc(Y)$ has finite rank.
\end{lemma}
\begin{proof}
First assume that $Y$ is smooth. We may choose a smooth projective completion $\bar{Y}$ of $Y$, hence we have $\Dperf(\bar{Y}) = D(\bar{Y})$. Write $j\colon Y \rightarrow \bar{Y}$ for the open immersion.  Let $E, E' \in D_c(Y)$ satisfy $[E] = [E'] \in \Knumc(Y)$. For any $P \in \Dperf(\bar{Y})$, we have   
$$ \chi_{\bar{Y}}(P, j_*E) = \chi_Y(j^*P, E) = \chi_Y(j^*P, E') = \chi_{\bar{Y}}(P, j_*E') $$
by adjunction, so the map $j_*\colon \Knumc(Y) \ra \Knum(\bar{Y})$ given by $j_*([E]) = [j_*E]$ is well-defined. We claim that $j_*$ is injective. Indeed, let $E, E' \in D_c(Y)$ satisfy $[j_*E] = [j_*E'] \in \Knum(\bar{Y})$. Each $P \in \Dperf(Y)$ is of the form $j^*\bar{P} = P$ for some $\bar{P} \in \Dperf(\bar{Y})$ by \cite[Lemma 2.3.1]{P07}, so
$$ \chi_Y(P, E) = \chi_Y(j^*\bar{P}, E) = \chi_{\bar{Y}}(\bar{P}, j_*E) = \chi_{\bar{Y}}(\bar{P}, j_*E') = \chi_Y(j^*\bar{P}, E') = \chi_Y(P, E'). $$
It follows $[E] = [E'] \in \Knumc(Y)$ as required. Since $\bar{Y}$ is smooth and projective, the numerical Grothendieck group $\Knum(\bar{Y}) = \Knumc(\bar{Y})$ has finite rank by Hirzebruch-Riemann-Roch.

Now assume that $Y$ has normal singularities. Let $\pi \colon Y' \to Y$ be a resolution of singularities. Write $\overline{\pi_*} \colon \Coh Y' \to \Coh Y$ and $\underline{\pi^*} \colon \Coh Y \to \Coh Y'$ for the underived pushforward and pullback respectively. Since $Y$ is normal, we have $\overline{\pi_*} \cO_{Y'} = \cO_Y$. Pushforward and pullback give a pair of adjoint functors $\cdc(Y') \to \cdc(Y)$ and $\Dperf(Y) \to \Dperf(Y')$, and thus an induced map $\Knumc(Y') \to \Knumc(Y)$. By the (underived) projection formula, the functor $\overline{\pi_*}$ is essentially surjective. Since $\Knumc(Y)$ is generated by classes of coherent sheaves, it follows that the map $\Knumc(Y') \to \Knumc(Y)$ is surjective.
\end{proof}

\begin{remark} \label{remark:finiterank}
From now on, we assume for simplicity that $\Knumc(Y)$ has finite rank. All our arguments work equally well if we assume instead that the central charge of each stability condition on $\cdc(Y)$ factors through a fixed finite rank lattice  $\Lambda$ via a homomorphism $\alpha \colon \Knumc(Y) \to \Lambda$.
\end{remark}

\subsection{Stability conditions for compact support}
We assume that the reader is familiar with the notion of stability condition as introduced in \cite{Bri07}, in particular the notion of \emph{slicing}. We note that typically
the category $\cdc(Y)$ is decomposable into infinitely many factors; indeed, any closed point $y \in Y$ that does not lie on a proper subvariety of positive dimension of $Y$ gives rise to such a factor. Hence, instead of applying the notion of stability condition verbatim to the category $\cdc(Y)$, we restrict to the situation where $\Knumc(Y)$ is a finite rank lattice and we allow only central charges that factor through $\Knumc(Y)$.

\begin{definition} \label{def:stabilitycondition}
Assume that $\Knumc(Y)$ has finite rank. A \emph{numerical stability condition for compact support} on $Y$ is a pair $(Z, \cP)$, where $Z \colon \Knumc(Y) \to \CC$ is a group homomorphism and $\cP$ is a slicing of $\cdc(Y)$, such that the following properties hold:
\begin{enumerate}
\item For any $\phi\in \RR$ and any non-zero $E\in \cP(\phi)$, we have $Z([E]) \in \RR_{>0}\cdot e^{\pi i \phi}$; and
\item There exists a quadratic form $Q$ on $\Knumc(Y) \otimes \RR$ such that:
\begin{itemize}
\item for any $\phi\in \RR$ and any $E\in \cP(\phi)$, we have
$Q([E]) \ge 0$; and
\item $Q$ is negative definite on $\mathrm{Ker} Z \subseteq \Knumc(Y) \otimes \RR$.
\end{itemize}
\end{enumerate}
 Let $\Stab(D_c(Y))$ denote the space of numerical stability conditions for compact support on $Y$.
\end{definition}

 The deformation results of Bridgeland~\cite{Bri07} extend to this setting (see e.g. \cite[Appendix A]{BMS} for a discussion under the assumptions as formulated above);  in particular, $\Stab(D_c(Y))$ is a complex manifold of dimension equal to the rank of $\Knumc(Y)$. 

Moreover, the results of \cite[Section 9]{Bri08} carry over completely
to give a wall-and-chamber structure on $\Stab(D_c(Y))$ for any given class $\vv \in \Knumc(Y)$. More precisely, there exists a locally finite set of walls (real codimension one submanifolds) such that the set of $\sigma$-semistable objects of class $\vv$ does not change as $\sigma$ varies within a connected component of the complement of walls (called a \emph{chamber}),
and such that on every wall there exist strictly semistable objects that become unstable on one side of the wall.

\section{The linearisation map with compact support}
 \label{sec:linearisation}
 The goal of this section is to prove the main result.
Let $Y$ and $S$ be separated schemes of finite type, and write $p\colon Y\times S\to Y$ and $q\colon Y\times S\to S$ for the first and second projection respectively.

 \subsection{The linearisation map}
 For any closed point $s\in S$, define $Y_s:= Y\times \{s\}$ and write $i_{Y \times \{s\}}\colon Y_s \to Y \times S$ for the closed immersion. For $\cE\in D(Y\times S)$, we identify $Y_s\cong Y$ and let 
\[
\cE_s:=i^*_{Y \times \{s\}}\cE\in D(Y)
\]
 denote the derived pullback of $\cE$ to $Y_s$.

\begin{definition}
\label{def:lc-flat}
 Let $\cA\subset D_c(Y)$ be the heart of a bounded t-structure on $D_c(Y)$. An $S$-perfect object $\cE \in D(Y \times S)$ is a \emph{flat family over $S$ with respect to $\cA$} if there exists $\vv\in \Knumc(Y)$ such that $\cE_s\in \cA$ and $[\cE_s]=\textbf{v}$ for every closed point $s\in S$. When we wish to make the reference to $\vv$ explicit, we call this a \emph{flat family of class} $\textbf{v}$ \emph{over $S$ with respect to $\cA$}.
\end{definition}

\begin{remark}
When $Y$ is proper and $S$ is connected, the condition that $[\cE_s]=\textbf{v}$ for every closed point $s\in S$ and some $\vv \in \Knumc(Y)$ is superfluous, because the class $[\cE_s]\in \Knumc(Y)$ is constant over $S$. Thus, our notion of flat family generalises the standard one, see \cite[Definition~3.1]{BM12}.
\end{remark}

\begin{definition} Let $S, Y$ be separated and of finite type over $\kk$. Let $(Z, \cP)$ be a numerical Bridgeland stability condition for compact support on $Y$ in the sense of Definition \ref{def:stabilitycondition}. We say that $\cE \in D(Y \times S)$ is a \emph{family of semistable objects of class $\vv\in \Knumc(Y)$} if $\cE$ is a flat family over $S$ with respect to $\cP((\phi, \phi+1])$ of class $\vv$ for some $\phi \in \RR$, and if in addition each object $\cE_s$ is semistable with respect to $(Z, \cP)$. 
\end{definition}

Assume that $S$ is separated and of finite type over $\kk$. Let $N^1(S)$ denote the vector space of real Cartier divisor classes modulo numerical equivalence; here numerical equivalence is taken with respect to proper curves $C \subset S$. Dually, $N_1(S)$ denotes the space of proper 1-cycles in $S$ modulo numerical equivalence (with respect to Cartier divisor classes on $S$). Let $[C]\in N_1(S)$ denote the class of a 1-cycle.
 
 \begin{theorem} 
 \label{thm:linearisation}
 Let $S, Y$ be separated schemes of finite type. Assume that $\Knumc(Y)$ has finite rank. Let $\sigma = (Z_\sigma, \cP_\sigma)$ be a numerical Bridgeland stability condition for compact support on $Y$, and let $\cE$ be a family of $\sigma$-semistable objects of class $\vv \in \Knumc(Y)$. Assume that the support of $\cE$ is proper over $S$. There is a nef Cartier divisor class $\ell_{\cE, \sigma} \in N^1(S)$ on $S$, defined dually by
 \begin{equation}
 \label{eqn:ellC}
 \ell_{\cE,\sigma}\big([C]\big) = \Im\left(\frac{Z_{\sigma}(\Phi_{\cE}(\cO_C)\big)}{-Z_\sigma(\vv)}\right) \quad \in \Hom(N_1(S), \RR) \cong N^1(S).
 \end{equation}
 Moreover, $\ell_{\cE, \sigma}([C]) = 0$ if and only if for two general closed points $c, c^\prime\in C$, the corresponding objects $\cE_c, \cE_{c'}\in \cdc(Y)$ are $S$-equivalent.
 \end{theorem}
 
 In fact, for a fixed family $\cE$, equation \eqref{eqn:ellC} defines a numerical Cartier divisor on $S$ for \emph{any} numerical stability condition for compact support on $Y$. The resulting map
 \[
 \ell_{\cE}\colon \Stab(\cdc(Y))\to N^1(S)
 \]
 obtained by sending a stability condition $\sigma^\prime$ to the divisor class $\ell_{\cE,\sigma^\prime}$ is the \emph{linearisation map} of the family $\cE$. 
 
We present the proof of Theorem~\ref{thm:linearisation} in two stages. We first prove that the linearisation map is well-defined, postponing until the next subsection the proof of the positivity statements.

 \begin{lemma} \label{lem:welldefined}
The assignment of \eqref{eqn:ellC} defines a numerical Cartier divisor class $\ell_{\cE,\sigma}\in N^1(S)$.
\end{lemma}

\begin{lemma}
\label{lem:EulerCharComparison}
 Let $S, Y$ be schemes of finite type. Let $\cE \in D_{lc}(Y \times S)$ be $S$-perfect and let $F \in \Dperf(Y)$. For any proper subscheme $i\colon T\hookrightarrow S$, we have
\begin{equation}
\label{eqn:Eulercomparison}
\chi_Y\big(F,p_*(\cE\otimes q^*i_*\cO_T)\big) = \chi_T\big(i^*q_*(\cE\otimes p^*F^\vee)\big).
\end{equation}
\end{lemma}
\begin{proof}
 Use the projection formula repeatedly to obtain 
 \begin{align*}
\chi_{Y}(F,p_*(\cE\otimes q^*i_*\cO_T) & = \chi_Y(F^\vee\otimes p_*(\cE\otimes q^*i_*\cO_T)) \\
 & = \chi_Y(p_*(p^*F^\vee\otimes \cE\otimes q^*i_*\cO_T)) &   \\
  & = \chi_{S} (q_*(p^*F^\vee\otimes \cE\otimes q^*i_*\cO_T)) &  \\
& = \chi_{S} (q_*(\cE\otimes p^*F^\vee)\otimes i_*\cO_T) & \\
& = \chi_{S} (i_*(i^*q_*(\cE\otimes p^*F^\vee))) & \\
   & = \chi_{T} (i^*q_*(\cE\otimes p^*F^\vee)) & 
 \end{align*}
 as required.
 \end{proof}

 \begin{proof}[Proof of Lemma~\ref{lem:welldefined}]
 Note first that the integral functor $\Phi_{\cE}\colon D_c(S)\to D_c(Y)$ is well-defined by Proposition~\ref{prop:functorpair}. Since the stability condition is assumed to be numerical, we can choose $P_i\in \Dperf(Y)$ and $a_i\in \RR$ for $1\leq i\leq m$ such that 
\[
\sum_{i=1}^m a_i \chi_Y(P_i,-) = \Im \frac{Z_\sigma(-)}{-Z_\sigma(\vv)}\in \Hom(\Knumc(Y),\RR) \quad \text{and} \quad \chi_Y(P_i, \vv) = 0 \ \text{for all $i$}.
\]
It is sufficient to show that for each $i$, there exists a Cartier divisor class $L_i$ on $S$, such that
 \begin{equation}
 \label{eqn:numericalequivalence}
 \chi_{Y}(P_i,\Phi_{\cE}(\cO_C)) =  L_i.C
 \end{equation}
 for all projective curves $C \subseteq S$.
 Proposition~\ref{prop:functorpair} gives $\Psi_{\cE}(P_i^\vee):=q_*(\cE \otimes p^*P_i^\vee)\in \Dperf(S)$. We claim that the object $\Psi_{\cE}(P_i^\vee)$ has rank zero. Indeed, for any closed point $s\in S$, apply Lemma~\ref{lem:EulerCharComparison} to the closed immersion $i\colon \spec \kk(s)\hookrightarrow S$ to obtain  
\[
\rk\big(\Psi_{\cE}(P_i^\vee)\big) = \chi_{\spec\kk(s)}(i^*q_*(\cE \otimes p^*P_i^\vee)) = \chi_Y(P_i,p_*(\cE\otimes q^*i_*\kk(s))) =  \chi_Y(P_i,\vv)
= 0.
\]
 Now apply Lemma~\ref{lem:EulerCharComparison} to the closed immersion $i\colon C\hookrightarrow S$ and deduce from Riemann--Roch that
\[
 \chi_{Y}(P_i,\Phi_{\cE}(\cO_C)) =\chi_C(C,i^*\Psi_{\cE}(P_i^\vee)) = 0\cdot(1-g(C)) + \deg \Psi_{\cE}(P_i^\vee)\vert_C = \deg \Psi_{\cE}(P_i^\vee)\vert_C.
\]
Since $\Psi_{\cE}(P_i^\vee)$ is perfect, it has a determinant line bundle $L_i$ by \cite{KnudsenMumford}. By the compatibility of the determinant construction
with restriction to $C$ we conclude $L_i.C = \deg \Psi_{\cE}(P_i^\vee)\vert_C$ and thereby also equation \eqref{eqn:numericalequivalence}.
\end{proof}

 \subsection{Positivity}
 We now establish the positivity statements from Theorem~\ref{thm:linearisation}, and for this we follow closely the approach of \cite[Section~3]{BM12}.
 
 We continue to work under the assumptions of Theorem~\ref{thm:linearisation}. In particular, $\cdc(Y)$ carries a Noetherian bounded t-structure with heart $\cA$. For any proper curve $C\subseteq S$, we obtain a sheaf of Noetherian t-structures on $\cdlc(Y\times C)$ over $C$ by Theorem~\ref{thm:t-str-any}. Write $\cA_C$ for the heart of this t-structure.

\begin{lemma}
\label{lem:openheartgeom}
 Given the assumptions of Theorem~\ref{thm:linearisation}, let $C\subseteq S$ be a proper curve and let $\cE_C$ denote the derived restriction of $\cE$ to $Y\times C$. Then $\cE_C\in \cA_C$. Moreover, for any line bundle $L$ of sufficiently high degree on $C$, we have $\Phi_{\cE}(L)\in \cA$.
\end{lemma}
\begin{proof}
The support of $\cE_C$ is proper over $C$, so the object $\cE_C\in D(Y\times C)$ has left-compact support by Lemma~\ref{lem:lcvsproper} because $C$ is proper. In particular,  the object $\cE_C\in \cdlc(Y\times C)$ satisfies the open heart property from Proposition~\ref{prop:open-heart}. The first statement now follows as in \cite[Lemma~3.5]{BM12}. The projection formula and flat base change give that $\Phi_{\cE}((i_C)_*F)=\Phi_{\cE_C}(F)$ for any $F\in D(C)$, where $i_C: C \to S$ is the inclusion.
The proof of \cite[Lemma~3.6]{BM12} applies verbatim to give the second statement.  
\end{proof}

\begin{proof}[Proof of Theorem~\ref{thm:linearisation}]
 For any $\sigma\in \Stab(\cdc(Y))$, we may assume that $Z_\sigma(\vv)=-1$ using the $\CC$-action on $\Stab(\cdc(Y))$.  
 
 We first prove that the numerical divisor class $\ell_{\cE,\sigma}\in N^1(S)$ is nef. Let $C$ be a proper curve in $S$. As in \cite[Proposition~3.2]{BM12}, it is straightforward to show that the value of $\ell_{\cE,\sigma}([C])$ from \eqref{eqn:ellC} is unchanged if we replace $\cO_C$ by any line bundle $L$ on $C$. In particular, if $L$ is of sufficiently high degree on $C$, Lemma~\ref{lem:openheartgeom} gives $\Phi_{\cE}(L)\in \cA$ and hence   
 \[
 \ell_{\cE,\sigma}\cdot C = \Im\big(Z_{\sigma}(\Phi_{\cE}(\cO_C))\big) = \Im\big(Z_{\sigma}(\Phi_{\cE}(L)))\big) \geq 0
 \]
 as required, because $Z_\sigma$ sends objects of $\cA$ to the semi-closed upper half plane. 
 
 To prove the second statement, suppose first that $\ell_{\cE,\sigma}\cdot C = 0$. For any smooth point $c\in C$ and for any $L\in \Pic(C)$ of sufficiently high degree, applying $\Phi_{\cE}$ to the short exact sequence 
 \[
 0\lra L(-c)\lra L\lra \kk(c)\lra 0
 \]
 and invoking Lemma~\ref{lem:openheartgeom} gives a short exact sequence
  \[
 0\lra \Phi_{\cE}(L(-c))\lra \Phi_{\cE}(L)\lra \cE_c\lra 0.
 \]
 of objects in $\cA$. We have $0=\ell_{\cE,\sigma}\cdot C = \Im Z_\sigma(\Phi_{\cE}(L))$ and $Z_\sigma(\vv)=-1$, so both $\Phi_{\cE}(L)$ and $\cE_c$ have phase 1. Since $\cE_c$ is a quotient of $\Phi_{\cE}(L)$ in $\cA$, each Jordan--H\"{o}lder factor of $\cE_c$ is a Jordan--H\"{o}lder factor of $\Phi_{\cE}(L)$. The latter factors don't depend on the choice of the smooth point $c\in C$. Since $\kk$ is an infinite field, \cite[Lemma~3.7]{BM12} implies that $\cE_c$ is $S$-equivalent to $\cE_{c'}$ for any $c, c'\in C$. For the other direction, assume $\cE_c$ is $S$-equivalent to $\cE_{c'}$ for any two general closed points $c, c'\in C$. The analogue of \cite[Lemma~3.9]{BM12} gives a filtration of $\cE|_{Y\times C}$ of length $n$, say, whose successive quotients are of the form $p^*F_i\otimes q^*L_i$, where each $L_i\in \Pic(C)$ and each $F_i\in \cA_Y$ has phase 1. The projection formula and flat base change give
 \[
 Z_\sigma(\Phi_{\cE}(\cO_C)) = \sum_{i=1}^n Z_\sigma(F_i\otimes p_*q^*L_i) =  \sum_{i=1}^n Z_\sigma(F_i\otimes \mathbf{R}\Gamma (L_i)) =   \sum_{i=1}^n \chi_C(L_i) Z_\sigma(F_i),
 \]
which lies on the real axis. Therefore $\ell_{\cE,\sigma}([C])= 0$ as required. 
\end{proof}

\begin{proof}[Proof of Theorem~\ref{thm:positivity}]
This is immediate from Lemma~\ref{lem:finiterank} and Theorem~\ref{thm:linearisation}.
\end{proof}

\subsection{A geometric condition to ensure proper support} 

The goal of this subsection is to show that one of the assumptions of Theorem \ref{thm:linearisation}, namely that the universal family has proper support over $S$, holds for moduli spaces of simple objects when $Y$ is semi-projective.

We continue to assume that all our schemes are separated and of finite type over $\kk$.

\begin{proposition} \label{prop:simpleimpliesproper}
Assume that $Y$ admits a proper morphism $\tau \colon Y \to X$ to an affine scheme $X$. Choose a nonzero class $\vv \in \Knumc(Y)$, and let $\cE$ be a flat family of class $\vv$ over $S$ with respect to some bounded t-structure on $\cdc(Y)$.  
Assume that for all closed points $s \in S$, the object $\cE_s$ satisfies
$\Hom(\cE_s, \cE_s) = \kk$.
Then $\cE$ has proper support over $S$.
\end{proposition}

We begin with two Lemmas, for which we make the same assumptions as in Proposition \ref{prop:simpleimpliesproper}.
\begin{lemma}
\label{lem:one-point}
Let $s\in S$ be any closed point. Then $\supp(\cE_s)$ is connected and $\tau(\supp(\cE_s))$ is a single closed point in $X_s:=X\times \{s\}$.
\end{lemma}

\begin{proof}
If $\supp(\cE_s)$ is disconnected, we can write $\cE_s = \cE_s' \oplus \cE_s''$ where the summands have disjoint support; this contradicts the assumption that $\cE_s$ has only $\kk$ as endomorphism. Similarly, assume that $\tau(\supp(\cE_s))$ contains more than one point. Since
$X$ is affine, there exists a function on $X$ whose pullback to $Y$ is non-constant on the support of $\cE_s$. Multiplication with this function would give a non-scalar endomorphism of $\cE_s$
which is absurd.
\end{proof}

 Let $\tau_S:=\tau\times \id_S\colon Y \times S \to X \times S$,
and consider $W := \tau_S(\supp(\cE))$ as a topological subspace of $X \times S$. 
Note that by Lemma~\ref{lem:easycompact}\two, the formation of
$W$ commutes with base change. The induced map of topological spaces $q \colon W \to S$ is bijective on closed points by Lemma \ref{lem:one-point}. 

\begin{lemma} \label{lem:Wirred}
Assume additionally that $S$ is irreducible. 
Then $W$ is irreducible.  
\end{lemma}
\begin{proof}
Assume that $W$ is reducible. Since $S$ is irreducible, there has to be an affine curve in $S$, intersecting the images of at least two irreducible components of $W$ under $q$.
Without loss of generality, we may therefore assume that $S$ itself is an affine curve.
After base change to the normalisation, we may assume further that $S$ is smooth.

Since $q$ is injective, there is an irreducible component of $W$ that maps to a point $s_0 \in S$. It follows that this component is a point, and is therefore a connected component of $W$; consequently, $\cE = \cE_0 \oplus \cE'$ where the support of $\cE_0$ is contained in $Y \times \{s_0\}$, and the support of $\cE'$ is disjoint from $Y \times \{s_0\}$. 

Let $i_{s_0} \colon Y \times \{s_0\} \to Y \times S$ be the inclusion. We claim that $\left[i_{s_0}^* \cE \right] = 0$ in $\Knumc(Y)$, in contradiction to our assumption. Replacing $S$ by an open subset if necessary, we may assume that $s_0 \in S$ is the scheme-theoretic zero locus of a regular function $f \in H^0(\cO_S)$. Each cohomology sheaf $H^j(\cE_0)$ has a filtration
$0 \subset \mathrm{ker} f \subseteq \mathrm{ker} f^2 \subseteq \dots$ whose successive quotients are isomorphic to the pushforward
$(i_{s_0})_* \cF$ of a coherent sheaf $\cF$ on $Y$. Restricting the short exact sequence
\[ \cF \boxtimes \cO_S \xrightarrow{\cdot f} \cF \boxtimes \cO_S \to (i_{s_0})_* \cF
\]
to $Y \times \{s_0\}$ 
shows that the class of $i_{s_0}^* (i_{s_0})_* \cF$ in the $K$-group of $\cdc(Y)$ vanishes. Since $\cE_0$ is a successive extension of its (finitely many nonzero) cohomology sheaves $H^j(\cE_0)$, the same holds for the class of $i_{s_0}^* \cE_0$. However, this is absurd because $\left[ i_{s_0}^* \cE_0 \right] = \left[ i_{s_0}^* \cE \right]  = \vv \neq 0$ in $\Knumc(Y)$.
\end{proof}

\begin{proof}[Proof of Proposition \ref{prop:simpleimpliesproper}]
We first claim that the bijective morphism $q \colon W \to S$ is a homeomorphism, for which it only remains to prove that $q$ is closed. This can be checked after base change via any proper and surjective map $\tilde S \to S$. Hence we may assume that $S$ is normal and, by restricting to one of its connected components, also irreducible. By Lemma \ref{lem:Wirred}, $W$ is irreducible. Let $\underline{W}$ be the reduced subscheme $\underline{W} \subseteq X \times S$; then the induced morphism $\underline{q}\colon \underline{W} \to S$ is a bijective map of varieties over $\kk$, with $S$ normal. Since $\kk$ is algebraically closed of characteristic zero, the dominant morphism $\underline{q}$ is birational. The original form of Zariski's main theorem implies that $\underline{q}$ is an open immersion, so it's an isomorphism, and hence $q$ is a homeomorphism. 

Since the same arguments apply after base change, it follows that $q$ is universally closed, and thus proper. 
 Since $\tau_S$ is proper, and $\supp(\cE)$ is a closed subset of
$\tau_S^{-1}(W)$, it follows that the support of $\cE$ is proper over $S$.
\end{proof}

\section{On schemes admitting a tilting bundle} \label{sec:tilting}

The goal of this section is to prove Theorem~\ref{thm:pullbackample}. To this end, we modify slightly the standard set-up (see Section \ref{subsec:moremotivation} for references) for stability conditions for quiver algebras of finite global dimension: rather than working with the category of nilpotent representations, we work with representations that
are finite-dimensional over $\kk$, but insist that the central charge factors via a variant of the numerical Grothendieck group, see Section \ref{subsec:stabquiver}; this is analogous to our set-up in Section \ref{sect:stabilityconditions}. 
In fact, when $Y$ admits a tilting bundle, we show that these notions yield a compatible notion of stability conditions in Section \ref{subsec:tilting}, a compatible notion of flat families in Section \ref{subsec:flatfamilies}, and finally compatible nef and semiample line bundles in the sense of Theorem \ref{thm:pullbackample} in Section \ref{subsec:comparison}.

\subsection{Stability conditions for quiver algebras} \label{subsec:stabquiver}
 We first recall notation and some standard facts from the representation theory of quivers; see, for example, \cite[Chapters~II-III]{ASS06} or \cite[Section~4.2]{10authors}.
 
 Let $Q$ be a connected quiver where both the vertex set $Q_0$ and the arrow set $Q_1$ are finite. The path algebra $\kk Q$ is graded by path length, and the part in degree zero has a basis of orthogonal idempotents $e_i$ for $i\in Q_0$. We do not require that $Q$ is acyclic, so $\kk Q$ may be infinite-dimensional as a $\kk$-vector space.  
 
 Our interest lies with associative $\kk$-algebras $A$ that can be presented in the form $A\cong \kk Q/I$, where $Q$ is a quiver and $I\subset \kk Q$ is a two-sided ideal generated by linear combinations of paths of length at least one; we refer to any such algebra $A$ as a \emph{quiver algebra}. For each vertex $i\in Q_0$, there is an indecomposable projective $A$-module $P_i:=Ae_i$ corresponding to paths in $Q$ emanating from vertex $i$. In addition, our assumption on the ideal $I$ ensures that each vertex of the quiver also gives rise to a one-dimensional simple $A$-module $S_i$ on which the class in $A$ of every arrow of the quiver acts as zero. Examples of quiver algebras include finite-dimensional algebras \cite{ASS06}, finitely generated graded algebras whose degree zero part is finite-dimensional semisimple \cite[Appendix~A]{BridgelandStern}, and algebras whose ideal of relations is defined in terms of a superpotential \cite{BocklandtSchedlerWemyss}.
 
 For a quiver algebra $A$, let $D_{\text{perf}}(A)$ and $D_{\text{fin}}(A)$ denote the bounded derived categories of perfect $A$-modules and finite-dimensional $A$-modules respectively, and let $K_{\text{perf}}(A)$ and $K_{\text{fin}}(A)$ respectively denote the Grothendieck groups of these categories. The Euler form $\chi_A\colon K_{\text{perf}}(A)\times K_{\text{fin}}(A)\to \ZZ$ is 
\[
\chi_A(E,F)=\sum_{i\in \ZZ} (-1)^i \dim_{\kk} \Ext_A^i(E,F).
\]
We write $\Knum(A)$ for the quotient $K_{\text{fin}}(A)/K_{\text{perf}}(A)^\perp$; by abuse of language, we call it the \emph{numerical Grothendieck group} of $D_{\text{fin}}(A)$. Note that the Euler form descends to a perfect pairing $\chi_A\colon K_{\text{perf}}(A)/K_{\text{fin}}(A)^\perp\times \Knum(A)\to \ZZ$, and we write $\Knum(A)^\vee = K_{\text{perf}}(A)/K_{\text{fin}}(A)^\perp$.

\begin{lemma}
\label{lem:perfectpair}
Let $A$ be a quiver algebra. The following are equivalent:
\begin{enumerate}
\item[\one] there is an isomorphism $\Knum(A) \cong \bigoplus_{i\in Q_0} \ZZ [S_i];$
\item[\two] the classes $[S_i]$ for $i\in Q_0$ generate $\Knum(A);$ 
\item[\three] the classes $[P_i]$ for $i\in Q_0$ generate $\Knum(A)^\vee;$
\item[\four] there is an isomorphism $\Knum(A)^\vee \cong \bigoplus_{i\in Q_0} \ZZ [P_i]$.
\end{enumerate}
\end{lemma}
\begin{proof}
 The vertex simple $A$-modules define classes $[S_i]\in \Knum(A)$ for $i\in Q_0$, and the indecomposable projective $A$-modules define classes $[P_i]\in \Knum(A)^\vee$. Since 
\begin{equation}
\label{eqn:ExtPS}
\Ext^k_A(P_j,S_i) = \left\{\begin{array}{cr} \kk & \text{for }k=0\text{ and }i=j; \\ 0 & \text{otherwise}\end{array},\right.
\end{equation}
 it follows that $\bigoplus_{i\in Q_0} \ZZ [S_i]$ is a subgroup of $\Knum(A)$ and that $\bigoplus_{i\in Q_0} \ZZ [P_i]$ is a subgroup of $\Knum(A)^\vee$. The statements \one-\four\ are now clearly equivalent. 
 \end{proof}

\begin{remarks}
\begin{enumerate}
\item The assumptions of Lemma~\ref{lem:perfectpair} hold when a quiver algebra $A$ satisfies the Krull--Schmidt theorem (see \cite[I.4.10]{ASS06} for the statement when $A$ is a finite dimensional $\kk$-algebra). Proposition~\ref{prop:suffgeometric} presents a geometric situation where the assumptions of Lemma~\ref{lem:perfectpair} hold. 
 \item To simplify notation, write $\theta(v):= \chi_A(\theta, v)$ for $\theta\in \Knum(A)^\vee$, $v\in \Knum(A)$.
\end{enumerate}
\end{remarks}

From now one we assume that any of the equivalent conditions in Lemma~\ref{lem:perfectpair} holds. Fix once and for all a dimension vector $\vv:= \sum_{i\in Q_0} v_i [S_i]\in \Knum(A)$. Consider the vector subspace of $\Knum(A)^\vee\otimes_\ZZ \RR$ given by  
\[
\Theta_{\vv}:= \vv^\perp=\big\{\theta\in \Hom(\Knum(A),\RR) \mid \theta(\vv) = 0\big\}.
\]
For $\theta\in \Theta_{\vv}$, an $A$-module $M$ of class $\vv\in \Knum(A)$ is \emph{$\theta$-semistable} if $\theta(N) \geq 0$ for every nonzero proper $A$-submodule $N$ of $M$. The notion of $\theta$-stability is defined by replacing  $\geq$ with $>$; if $\vv$ is primitive, we say $\theta\in \Theta_{\vv}$ is \emph{generic} if every $\theta$-semistable $A$-module is $\theta$-stable. The choice of dimension vector $\vv\in \Knum(A)$ therefore determines a wall and chamber structure on the space $\Theta_{\vv}$ of stability parameters,  where two generic parameters $\theta, \theta^\prime\in \Theta_{\vv}$ lie in the same chamber if and only if the notions of $\theta$-stability and $\theta^\prime$-stability coincide. 

For any integral parameter $\theta\in \Theta_{\vv}$, King~\cite{King94}, and more generally, Van den Bergh~\cite{VdB04}, constructs the coarse moduli space $\overline{\mathcal{M}_A}(\vv,\theta)$ of S-equivalence classes of $\theta$-semistable $A$-modules of dimension vector $\vv$ as a GIT quotient 
\[
\overline{\cM_A}(\vv,\theta) = X/\!\!/\!_{\chi_\theta} G,
\]
where $X$ is an affine scheme,  $G=\big(\prod_{i\in Q_0} \mathrm{GL}(v_i)\big)/\kk^\times$, and 
$\chi_\theta\in G^\vee$ is a character determined by $\theta$. In particular,  $\overline{\mathcal{M}_A}(\vv,\theta)$ is projective over an affine scheme, where the polarising ample line bundle $L(\theta)$ on $\overline{\cM_A}(\vv,\theta)$ descends from the linearisation of $\cO_X$ by $\chi_\theta$. Note that $\overline{\mathcal{M}_A}(\vv,\theta)$ is projective when the quiver $Q$ is acyclic.

 If the dimension vector $\vv\in \Knum(A)$ is primitive and if $\theta\in \Theta_{\vv}$ is generic, then $\overline{\mathcal{M}_A}(\vv,\theta)$ coincides with the fine moduli space $\mathcal{M}_A(\vv,\theta)$ of isomorphism classes of $\theta$-stable $A$-modules of class $\vv$. The universal family on $\mathcal{M}_A(\vv,\theta)$ is a flat family of $\theta$-stable $A$-modules of dimension vector $\vv$, that is, a locally free sheaf $T=\bigoplus_{i\in Q_0} T_i$ with $\rk(T_i)=v_i$ such that the fibre of $T$ at any closed point of $\mathcal{M}_A(\vv,\theta)$ is a $\theta$-stable $A$-module of dimension vector $\vv$.

\smallskip

Let $\mathcal{A}$ denote the abelian category of finite-dimensional $A$-modules, so $\cA$ is the heart of the standard t-structure on $D_{\text{fin}}(A)$. Define 
\[
\Lambda := \big\{ \lambda \in \Hom(\Knum(A), \RR) \mid \lambda([S_i])>0 \text{ for all } i \in Q_0 \big\}.
\]
For $M\in \cA$, we have $\lambda(M) \geq 0$ for all $\lambda \in \Lambda$, where equality holds iff $M=0$. The next results extend the observation of Bridgeland~\cite[Example~5.5]{Bri07} on finite-dimensional algebras. 

\begin{lemma}
\label{lem:stab-cond}
For $\theta \in \Theta_{\vv}$, $\lambda \in \Lambda$ and $\xi \in \RR$, define $Z_{\theta, \lambda, \xi}\colon \Knum(A)\to \CC$ by setting 
\[
Z_{\theta, \lambda, \xi}(M):= \theta(M)+\left(\sqrt{-1} + \xi\right) \cdot\lambda(M).
\]
Then $\sigma_{\theta,\lambda, \xi}:=(Z_{\theta,\lambda, \xi}, \cA)$ is a Bridgeland stability condition on $\Dfin(A)$, satisfying the support property with respect to $\Knum(A)$.
\end{lemma}
\begin{proof}
For $\theta\in \Theta_{\vv}$, the image under $Z_{\theta,\lambda, \xi}$ of any nonzero object of $\cA$ lies in the upper half plane, so $Z_{\theta,\lambda, \xi}$ is a stability function on $\cA$. Objects of $\cA$ have finite dimension over $\kk$, so Harder--Narasimhan filtrations exist \cite[Proposition~2.4]{Bri07} and hence $\sigma_{\theta,\lambda}$ is a stability condition on the bounded derived category of $\cA$, that is, on $\Dfin(A)$.  

By the original definition of the support property \cite[Lemma~A.4]{BMS}, we must exhibit $C > 0$ such that
\[ \left| Z_{\theta,\lambda, \xi}(E) \right| \geq C \| [E] \|
\]
for all semistable objects $E$, and
with respect to some norm $\| \cdot \|$ on $\Knum(A) \otimes \RR \cong \RR^{Q_0}$. We may choose the supremum norm on $\RR^{Q_0}$. Up to shift,
any semistable object lies in the heart $\cA \subset \Dfin(A)$, so its class in $\Knum(A)$ is a non-negative linear combination of the classes
$[S_i]$ for the simple objects for $i \in Q_0$. Setting $C := \min_{i \in Q_0} \lambda([S_i])$, the claim becomes evident.
\end{proof}

\begin{remark}
 Since stability conditions are characterised by their heart and central charge \cite[Lemma 3.5]{Bri08}, the set of stability conditions of the form $\sigma_{\theta,\lambda,\xi}$ can be identified with the interior of the set of stability conditions whose heart is the category $\cA$ of finite-dimensional $A$-modules.
 \end{remark}

\begin{lemma}
\label{lem:thetavsbridgeland}
For $\theta \in \Theta_{\vv}, \lambda \in \Lambda$, $\xi\in \RR$,
an object $E \in \Dfin(A)$ of class $\vv$ is $\sigma_{\theta, \lambda, \xi}$-(semi)stable and of phase in $(0,1]$ if and only if it is 
a $\theta$-(semi)stable $A$-module.
\end{lemma}
\begin{proof}
An object $E\in \Dfin(A)$ of class $\vv$ is $\sigma_{\theta,\lambda,\xi}$-semistable of phase in $(0, 1]$ if and only if $E$ lies in the heart $\cA$, and the phase of  $Z_{\theta, \lambda, \xi}(F)$ is smaller than the phase of $Z_{\theta, \lambda, \xi}(E)$ for every proper nonzero submodule $F\subset E$. Since $\theta(\vv)=0$, we have $Z_{\theta, \lambda, \xi}(E) \in \RR_{>0} \cdot \left(\sqrt{-1} + \xi\right)$, so this is equivalent to $\theta(F) \geq 0$. Thus, the $\sigma_{\theta,\lambda, \xi}$-(semi)stable objects in $\Dfin(A)$ of class $\vv$ are precisely the $\theta$-(semi)stable $A$-modules of class $\vv$.
\end{proof}

Let $\Stab(\Dfin(A))$ denote the space of numerical stability conditions on 
$\Dfin(A)$ that satisfy the support property with respect to $\Knum(A)$. Combining the above results gives the following picture.  

\begin{proposition} 
\label{prop:thetatostab}
Let $\vv \in \Knum(A)$. Then there is a continuous map
\begin{equation} \label{thetatostab}
f\colon \Theta_{\vv} \times \Lambda \times \RR \to \Stab(\Dfin(A)), \quad \left(\theta, \lambda, \xi\right) \mapsto \sigma_{\theta, \lambda, \xi}
\end{equation}
 such that for any fixed $\lambda\in \Lambda$, $\xi\in \RR$, the wall-and-chamber structure on $\Theta_{\vv}$ is obtained by pulling back the wall-and-chamber structure on $\Stab(\Dfin(A))$ with respect to $\vv$. Moreover, for each triple $(\theta, \lambda, \xi)$, the moduli stack of $\sigma_{\theta, \lambda, \xi}$-(semi-)stable 
objects gets identified with the moduli stack of $\theta$-(semi)stable quiver representations.
\end{proposition}

When $\vv$ is primitive and $\theta$ generic, the map \eqref{thetatostab} gives an identification of fine moduli spaces; otherwise, the moduli stack of $\sigma_{\theta, \lambda, \xi}$-semistable objects of class $\vv$ has $\overline{\cM_A}(\vv,\theta)$ as coarse moduli space, which, as noted above, is projective over an affine.

\subsection{On schemes with a tilting bundle} \label{subsec:tilting}
Let $Y$ be a smooth scheme that admits a projective morphism $\tau\colon Y \to X = \Spec \:R$, and let $E$ be a locally-free sheaf of finite rank on $Y$. 


 We begin with a few comments about our conventions concerning left- and right-modules; in this paragraph all of our functors are underived. For any coherent sheaf $F$ on $Y$, the space $\Hom(E,F)$ is a right $\End(E)$-module and therefore a left $\End(E)^{\operatorname{op}}$-module; equivalently, and more geometrically, $\Hom(E,F)$ is a left module over the algebra 
\[
A:= \End(E^\vee).
\]
 Also, since $E$ is a left $\End(E)$-module and hence a right $A$-module, so $E\otimes_A M$ is well-defined for any left $A$-module $M$. 

Recall that a \emph{tilting bundle} on $Y$ is a locally-free sheaf $E$ of finite rank such that $\Ext^k(E,E)=0$ for $k>0$, and such that if $F\in \cD(Y)$ satisfies $\Hom(E,F)=0$, then $F = 0$.

\begin{theorem}
\label{thm:hilleVdB}
Let $Y$ be a smooth scheme that is projective over an affine scheme, and let $E$ be a tilting bundle on $Y$. Then $A:=\End(E^\vee)$ is an algebra of finite global dimension, and the derived $\Hom$ functor gives an exact equivalence
\begin{equation}
\label{eqn:Ddiagram}
\Hom(E,-)\colon D(Y) \longrightarrow\cD(A)
\end{equation}
 with quasi-inverse $E\otimes_A -$. Moreover, the restriction of this equivalence is an exact equivalence between $\cdc(Y)$ and $\Dfin(A)$, and there is an isomorphism $\Knumc(Y)\cong \Knum(A)$. 
\end{theorem}
\begin{proof}
 Since $Y$ is smooth and projective over a Noetherian affine scheme of finite type, applying the result of Hille--Van den Bergh~\cite[Theorem~7.6]{HilleVandenBergh04} and then the equivalence from right $\End(E)$-modules to left $A$-modules gives the equivalence \eqref{eqn:Ddiagram}, together with the fact that $A$ has finite global dimension. Composing the right/left-equivalence $(-)^{\operatorname{op}}$ with the quasi-inverse of the equivalence from [ibid.] shows that \eqref{eqn:Ddiagram} has quasi-inverse $(-)^{\operatorname{op}}\otimes_{A^{\operatorname{op}}} E$, which we express more succinctly as $E\otimes_A -$. Once we establish the equivalence between $\cdc(Y)$ and $\Dfin(A)$, the isomorphism on numerical Grothendieck groups follows from the fact that \eqref{eqn:Ddiagram} preserves the Euler forms. 
 
 It remains to prove the equivalence between $\cdc(Y)$ and $\Dfin(A)$. For any coherent sheaf $F$ on $Y$ with proper support, the $\kk$-vector space $\Hom^i(E,F)=\Ext^i(E,F)$ has finite-dimension, so \eqref{eqn:Ddiagram} restricts to a functor from $\cdc(Y)$ to $\Dfin(A)$. To show this functor is essentially surjective, we need only show that for any $F \in D(Y)$ with non-proper support, there exists $k \in \ZZ$, such that $\dim\Hom^k(E, F)=\infty$. For simplicity, the object $G:= E^\vee \otimes F$ also has non-proper support and satisfies $\Hom^k(E,F) = H^k(Y,G)$.
 
 When $G$ is a sheaf, we claim that $\dim H^0(Y,G)=\infty$. Indeed, $H^0(Y, G) = H^0(X, \tau_*G)$, where $\tau_*G$ also has non-proper support. Let $Z$ be a non-proper irreducible component of $\supp\tau_*G$. Since $Z$ is a closed subscheme of $X$, we have $Z = \Spec\: R/I$ for some ideal $I$ satisfying where $\dim_\kk (R/I) = \infty$. Since all higher cohomology vanishes on an affine scheme, there is a surjective map $ H^0(X,\tau_*G) \twoheadrightarrow H^0(Z,\tau_*G|_Z)$ and we need only prove $\dim H^0(Z,\tau_*G|_Z) = \infty$. Assume $s \in H^0(Z,\tau_*G|_Z)$ does not vanish at the generic point of $Z$. Then the $\kk$-linear map $R/I \to H^0(Z,\tau_*G|_Z)$ given by $f \mapsto fs$ is injective and has image equal to a subspace of $H^0(Z,\tau_*G|_Z)$ of infinite dimension, as desired.
 
 For the general case, let $G \in D(Y)$ have non-proper support and let $\cH^i(G)$ denote the $i$th cohomology sheaf of $G$. By \cite[p.56, (2.6)]{Hu06a} we have the spectral sequence
\[
E_2^{p,q} = H^p(Y, \cH^q(G)) \Longrightarrow H^{p+q}(Y, G).
\]
Since $G$ is bounded, there exists $a,b \in \ZZ$ such that $\cH^q(G)=0$ unless $a \leq q \leq b$. Let $k$ be the smallest index such that $\supp\cH^k(G)$ is non-proper. We show by induction for $r \geq 2$ that in the $r$-th page of the spectral sequence, we have
\begin{itemize}
\item $E_r^{p,q} = 0$ if $q < a$ or $p < 0$;
\item $\dim E_r^{p,q} < \infty$ if $a \leq q \leq k-1$;
\item $\dim E_r^{0,k} = \infty$.
\end{itemize}
 These statement hold by the previous paragraph when $r=2$. Assume they hold for $r$. Since $E_{r+1}^{p,q}$ is a subquotient of $E_r^{p,q}$, we have $\dim E_{r+1}^{p,q} \leq \dim E_r^{p,q}$, and the first two statements follow. For the third statement, notice that the arrow pointing towards $E_r^{0,k}$ is zero while the arrow emanating from it has head at a vector space of finite dimension, so $\dim E_{r+1}^{0,k}=\infty$. This completes the induction, giving $\dim E_\infty^{0,k} = \infty$ and hence $\dim H^k(Y, G)=\infty$ as desired.
\end{proof}

 From now on we work under the assumptions of Theorem~\ref{thm:hilleVdB}. In this case, it is well known (see for example Karmazyn~\cite[Section~2.4]{Karmazyn14}) that:
 \begin{enumerate}
 \item after replacing $A$ by a Morita equivalent algebra if necessary, we may assume that $E$ admits a decomposition of the form  $E=\bigoplus_{1\leq i\leq k} E_i$, where each locally-free summand is indecomposable and where $E_i$ and $E_j$ are non-isomorphic for $i\neq j$; and
  \item the algebra $A=\End(E^\vee)$ can be presented as the quotient $A\cong \kk Q/I$ of the path algebra of a quiver with relations. 
 \end{enumerate}
 However, in order to apply the results of the previous subsection, we require that the equivalent conditions of Lemma~\ref{lem:perfectpair} are satisfied. Here we present a sufficient geometric condition:

\begin{proposition}
\label{prop:suffgeometric}
 Under the assumptions of Theorem \ref{thm:hilleVdB}, suppose that the locus $Z$ in $Y$ contracted by $\tau\colon Y \to X$ is proper and connected. Assume in addition that $E = \bigoplus_{1\leq i\leq m} E_i$ admits a splitting with the following properties:
\begin{enumerate}
\item[\one] $A:=\End(E)$ is a quiver algebra generated by $\id_{E_1}, \dots, \id_{E_m}$ and $\bigoplus_{i \neq j} \Hom(E_i, E_j);$
\item[\two] the restrictions $E_i|_Z$ are simple sheaves that are pairwise non-isomorphic.
\end{enumerate}
Then $\Knum(A) = \bigoplus_{1\leq i\leq m} \ZZ[S_i]$, where $S_1, \dots, S_m$ are the vertex simple $A$-modules.
\end{proposition}
\begin{proof}
 Note that $\Knumc(Y)$ is generated by simple sheaves. For any such sheaf $\cF$, the image $\tau(\supp \cF) \subset X$ is a point, so either $\cF = \cO_y$ for some $y \notin Z$, or $\supp \cF \subset Z$. In the former case, the numerical class $[\cO_y]$ does not depend on $y \in Y$; thus there is a finite set of sheaves $\cF_1, \dots, \cF_n$ supported on $Z$ whose classes generate $\Knumc(Y)$. After a further filtration, we may assume that each $\cF_j$ is scheme-theoretically supported on $Z$. After tensoring each by a sufficiently high power of an ample line bundle on $Y$, we may assume that $\Ext^i(E, \cF_j) = H^i(E^\vee \otimes \cF_j) = 0 $ for all $i>0$ and $1\leq j\leq n$; in other words, each object $M_j := \Hom(E,\cF_j) \in D(A)$ is actually an $A$-module, and the classes $[M_1],\dots, [M_n]$ generate $\Knum(A)$.

 Present $A=\kk Q/I$ as a quiver algebra. We claim that every nontrivial cycle in $Q$ acts as zero on each $M_j$. Indeed, assumption \one\ ensures that every nontrivial cycle starting at vertex $i$ corresponds to an endomorphism of $E_i$ that factors via at least one other summand $E_k$; assumption \two\ ensures that any such endomorphism acts as zero on $E_i|_Z$. It therefore acts as zero on $\Hom(E_i|_Z, \cF_j) = \Hom(E_i, \cF_j) = M_j$ as claimed. It follows that each $A$-module $M_j$ is pulled back from a representation of the finite-dimensional quotient
$A/\langle \text{nontrivial cycles} \rangle$ of $A$. In particular, $M_j$ is a nilpotent representation of $A$, so $M_j$ can be written as an extension of the vertex simple $A$-modules $S_1,\dots, S_m$. The classes $[S_1],\dots, [S_m]$ therefore generate $\Knum(A)$, and the result follows from Lemma~\ref{lem:perfectpair} and Theorem~\ref{thm:hilleVdB}.
\end{proof}

\begin{remark}
If the morphism $\tau$ from Proposition~\ref{prop:suffgeometric} is not birational, then $Y$ is forced to be projective, in which case the fact that any tilting bundle admits a splitting such that the equivalent assumptions of Lemma~\ref{lem:perfectpair} hold was well known \cite{Kingpreprint}. Otherwise, the typical situation where Proposition~\ref{prop:suffgeometric} applies is to resolutions of an isolated singularity. 
\end{remark}

The tilting equivalence identifies the space $\Stab(\Dfin(A))$ with the space of numerical stability conditions on $Y$ for compact support
in the sense of Definition \ref{def:stabilitycondition}. For any class $\vv\in \Knumc(Y)$ and for any $\theta\in \Theta_{\vv}$, $\lambda\in \Lambda$ and $\xi\in \RR$, we abuse notation and also write $\sigma_{\theta,\lambda,\xi}$ for the resulting stability condition on $\cdc(Y)$. 

 Assume now that the equivalent assumptions of Lemma~\ref{lem:perfectpair} hold. We now compute explicitly the image of $\sigma_{\theta,\lambda,\xi}$ under the linearisation map $\ell_{\cE}$ determined by any flat family $\cE$. For this, let $S$ be any separated scheme of finite type, and for any numerical Bridgeland stability condition $(Z, \cP)$ for compact support on $Y$, let $\cE \in D(Y \times S)$ be a family of semistable objects of class $\vv\in \Knumc(Y)$.

 \begin{lemma}
 \label{lem:ellsigmatheta}
  For any $\vv\in \Knumc(Y)$, assume that the flat family $\cE$ of semistable objects of class $\vv$ is proper over $S$. Then for any $\theta=\sum_{} \theta_i [P_i] \in \Theta_{\vv}$, $\lambda\in \Lambda$ and $\xi\in \RR$, we have that 
  \begin{equation}
\label{eqn:ellE}
\ell_\cE(\sigma_{\theta,\lambda,\xi}) = \frac{1}{(\xi^2+1)\lambda(\vv)} \cdot \sum_{1\leq i\leq k} \theta_i \cdot \det \left( \Psi_\cE(E_i^\vee) \right).
\end{equation}
 \end{lemma}
 \begin{proof}
The central charge of the stability condition $f(\theta, \lambda, \xi) = \sigma_{\theta,\lambda,\xi}$ on $\cdc(Y)$ is identified with the central charge $Z_{\theta,\lambda,\xi}$ from Lemma~\ref{lem:stab-cond} under the isomorphism from $\Knumc(Y)^\vee_{\CC}$ to $\Knum(A)^\vee_{\CC}$ that sends $[E_i]$ to $[P_i]$ for $1\leq i\leq k$. A simple calculation shows that
\[
\Im\left( \frac{Z_{\theta,\lambda,\xi}(-)}{-Z_{\theta,\lambda,\xi}(\vv)} \right) = \frac{\theta(-)}{(\xi^2+1)\lambda(\vv)}.
\]
 Now, $\theta = \sum_{1\leq i\leq k} \theta_i[P_i] \in \Knum(A)^\vee_{\RR}$ is identified with  $\sum_{1\leq i\leq k} \theta_i [E_i]\in \Knumc(Y)^\vee_{\RR}$. The result follows from the proof of Lemma \ref{lem:welldefined}. 
\end{proof}

\subsection{Comparison of flat families} \label{subsec:flatfamilies}
 From now on we assume that $Y$ is a smooth scheme, projective over an affine scheme, that carries a tilting bundle $E$ such that $A:=\End(E^\vee)$ is a quiver algebra satisfying $\Knum(A)\cong \bigoplus_{i\in Q_0} \ZZ [S_i]$.
 
 Let $S$ be any separated scheme of finite type. Our next goal is to extend the functor from \eqref{eqn:Ddiagram} to obtain a natural correspondence between flat families of certain Bridgeland-semistable objects over $S$ on one hand, and flat families of King-semistable objects over $S$ on the other. 
 
 First, let $P\to A$ denote the minimal projective resolution of $A$ as an $(A,A)$-bimodule. Following Butler--King~\cite{ButlerKing}, the term of $P$ in degree $l\in \ZZ$ is  
  \[
  P^l=\bigoplus_{1\leq i,j\leq k} Ae_i\otimes V_{i,j}^l \otimes e_jA,
  \]
  where $e_1,\dots,e_k$ are the orthogonal idempotents corresponding to the summands $E_1^\vee,\dots, E_k^\vee$ of $E^\vee$, and $V_{i,j}^l = \mathrm{Tor}^l_A(S_i,S_j)$ is a finite dimensional $\kk$-vector space. Set $d_{i,j}^l:= \dim_\kk V_{i,j}^l$.
  
  Let $F$ be a locally-free sheaf on $S$ that is also a left $A$-module, and write $F=\bigoplus_i F_i$ for the idempotent decomposition as an $A$-module. The left $\End(E)$-module $E=\bigoplus_i E_i$ becomes a right $A$-module, and the derived tensor product $E\otimes_A F$ is represented by the complex $E\otimes_A P\otimes_A F$, whose term in degree $l\in \ZZ$ is the locally-free sheaf on $Y\times S$ given by
 \begin{align}
 E\otimes_A P^l\otimes_A F & = \bigoplus_{1\leq i,j\leq k} E_i\otimes_\kk V_{i,j}^l \otimes_\kk F_j & \nonumber \\
 & = \bigoplus_{1\leq i,j\leq k} \big(E_i\otimes_\kk F_j\big)^{\oplus d_{i,j}^l} & \nonumber \\
 & = \bigoplus_{1\leq i,j\leq k} \big(p^*E_i\otimes_{\cO_{Y\times S}} q^*F_j\big)^{\oplus d_{i,j}^l}, & \label{eqn:boxtimesterms}
 \end{align}
where $p\colon Y\times S\to Y$ and $q\colon Y\times S\to S$ are the first and second projections. The maps of the complex $E\otimes_A P\otimes_A F$ are morphisms of sheaves induced by the maps of the complex $P$. Note that $E\otimes_A F\in D(Y\times S)$ because $A$ has finite global dimension.  
  
 \begin{proposition}
 \label{prop:flatfamAmods}
 For any $\vv\in \Knumc(Y)$, let $\theta\in \Theta_{\vv}$, and consider $\lambda\in \Lambda, \xi\in \RR$. 
 \begin{enumerate}
 \item[\one] Let $\cF\in D(Y\times S)$ be a flat family of $\sigma_{\theta,\lambda,\xi}$-(semi)stable objects  of class $\vv$ with respect to $\cA$. If $\cF$ is proper over $S$, then $q_*(\cF\otimes p^*(E^\vee))\in D(S)$ is a flat family of $\theta$-(semi)stable $A$-modules of dimension vector $\vv$. 
 \item[\two] Conversely, let $F\in D(S)$ be a flat family of $\theta$-(semi)stable $A$-modules of dimension vector $\vv$ over $S$. Then $E\otimes_A F\in D(Y\times S)$ is a flat family of $\sigma_{\theta,\lambda,\xi}$-(semi)stable objects  of class $\vv$ with respect to $\cA$.
 \end{enumerate}
  Moreover, when the output from \two\ has proper support over $S$, these operations are mutually inverse (up to quasi-isomorphism).
 \end{proposition}
\begin{proof}
 Let $s\in S$ be a closed point. Write $i\colon \{s\}\hookrightarrow S$ and $i_s\colon Y\times \{s\}\to Y\times S$ for the closed immersions, and let $q_s\colon Y\times \{s\}\to \{s\}$ denote the second projection.  
 
 For \one, note first that $q_*((p^*E)^\vee\otimes \cF) = \Psi_{\cF}(E^\vee)$, and that the derived pullback to $s\in S$ is
\begin{equation}
\label{eqn:flatfamilyfibreone}
i^*q_*(p^*(E^\vee)\otimes \cF) = (q_s)_*(i_s)^*\big(p^*(E^\vee)\otimes \cF\big) = (q_s)_*(E^\vee\otimes \cF_s) = \Hom_Y(E,\cF_s).
\end{equation}
 We have that $\cF\in D(Y\times S)$ is $S$-perfect and proper over $S$. Since $E$ is locally-free, Proposition~\ref{prop:functorpair} implies that $\Psi_{\cF}(E^\vee)\in \Dperf(S)$. By \eqref{eqn:flatfamilyfibreone}, the derived restriction of $\Psi_{\cF}(E^\vee)$ to each closed point of $S$ is concentrated in degree zero, hence so is $\Psi_{\cF}(E^\vee)$. Thus, we've shown that $q_*(\cF\otimes p^*(E^\vee))$ is a locally-free sheaf on $S$ whose fibre over each closed point $s\in S$ is the $\theta$-semistable $A$-module $\Hom_Y(E,\cF_s)$ of dimension vector $\vv$. 

 To complete the proof of \one, it remains to show that the $A$-module structure on each fibre comes from a $\kk$-algebra homomorphism $A\to \End(q_*(\cF\otimes p^*(E^\vee)))$. For any open subset $U\subseteq S$, the space of sections of $q_*(\cF\otimes p^*(E^\vee))$ over $U$ is $\Gamma(Y\times U, p^*E^\vee\otimes \cF)$. Note that $A=\End(E^\vee)$ acts on the first factor whose restriction to any closed point $s\in U$ recovers the $A$-module structure on the fibre over $s$ by \eqref{eqn:flatfamilyfibreone}.
 
  For \two, the locally-free sheaf $F$ has a fibrewise left $A$-module structure, so $E\otimes_A F\in D(Y\times S)$ as above. Since $p^*E_i$ and $q^*F_j$ are $S$-perfect for $1\leq i,j\leq k$, we have that $E\otimes_A F$ is $S$-perfect by \eqref{eqn:boxtimesterms} and \cite[III, Proposition~4.5]{SGA6}. For a closed point $s\in S$, we have that 
  \[
  i_s^*(p^*E_i\otimes_{\cO_{Y\times S}} q^*F_j) = i_s^*p^*E_i\otimes_{\cO_{Y}} i_s^*q^*F_j = E_i\otimes_{\cO_{Y}} \big(\cO_Y\otimes_\kk(F_j)_s\big) = E_i\otimes_\kk (F_j)_s
  \]
  for all $1\leq i,j\leq k$, where $(F_j)_s$ denotes the fibre of $F_j$ over $s\in S$. The functors commute with direct sums, so just as in \eqref{eqn:boxtimesterms} above, for each $l\in \ZZ$, the $l$-th terms of $i_s^* (E\otimes_A F)$ and $E\otimes_A F_s$ coincide, where $F_s$ is the fibre of $F$ over $s\in S$. Since the maps in each complex derive from those of $P$, it follows that
  \[
  i_s^* (E\otimes_A F) = E\otimes_A F_s.
  \]
  Since each $F_s$ is a $\theta$-semistable $A$-module of dimension vector $\vv$, Lemma~\ref{lem:thetavsbridgeland} and Theorem~\ref {thm:hilleVdB} imply that $i_s^* (E\otimes_A F)$ is $\sigma_{\theta,\lambda,\xi}$-semistable of class $\vv$.  
 
 The proof that these operations are mutually inverse requires the fact that $E\otimes_A E^\vee\cong \cO_{\Delta}$ for the diagonal $\Delta\subset Y\times Y$ as in King~\cite{Kingpreprint}; we leave the details to the reader.
\end{proof}

\begin{remark}
\label{rem:stableflatfamilies}
The assumption in Proposition~\ref{prop:flatfamAmods} that $\cF$ is proper over $S$ is superfluous for a flat family of $\sigma_{\theta,\lambda,\xi}$-stable objects by Proposition~\ref{prop:simpleimpliesproper}.
\end{remark}

\begin{example}
 The flat family $E\otimes_A F$ of Bridgeland-stable objects was first studied by King~\cite{Kingpreprint} in the case when $S=Y$ and $F=E^\vee$; [ibid.] would write our $E\otimes_A F$ as $F\boxtimes_{A^{\operatorname{op}}} E$. 
\end{example}

\subsection{Comparison of line bundles} \label{subsec:comparison}
 For any class $\vv\in \Knum(A)$ and any integral parameter $\theta \in \Theta_{\vv}$, the GIT construction produces an ample line bundle $L(\theta)$ on the coarse moduli space
$\overline{\cM_{\cA}}(\vv, \theta)$. Given a family of $\theta$-semistable $A$-modules of dimension vector $\vv$ over a scheme $S$, the induced morphism $f \colon S \to \overline{\cM_{\cA}}(\vv, \theta)$ produces a semi-ample line bundle $f^* L(\theta)$. We now provide an alternative description of this line bundle using the linearisation map.

 Given a flat family $\cE\in D(Y\times S)$ of $\sigma_{\theta,\lambda,\xi}$-semistable objects of class $\vv$ with respect to $\cA$ that is proper over a separated scheme $S$ of finite type, we obtain by Proposition~\ref{prop:flatfamAmods} a flat family $q_*(\cE\otimes p^*E^\vee)$ of $\theta$-semistable $A$-modules of dimension vector $\vv$ over $S$, and hence a morphism
 \[
 f\colon S\to \overline{\cM_A}(\vv,\theta)
 \]
 to the coarse moduli space. Recall that the polarising ample line bundle $L(\theta)$ on $\overline{\cM_A}(\vv,\theta)= X/\!\!/_{\!\chi_\theta} G$ descends from the linearisation of $\cO_X$ by the character $\chi_\theta\in G^\vee$.

 \begin{theorem}
\label{thm:polarisingbundle}
Suppose that a flat family $\cE \in D(Y \times S)$ of $\sigma_{\theta,\lambda,\xi}$-semistable objects of class $\vv$ has proper support over a separated scheme $S$ of finite type. Then the numerical divisor class $\ell_{\cE}(\sigma_{\theta,\lambda,\xi})$ on $S$ and the polarising ample line bundle $L(\theta)$ on $\overline{\cM_A}(\vv,\theta)$ satisfy
 \[
 \ell_{\cE}(\sigma_{\theta,\lambda,\xi}) = c_{\lambda, \xi}\cdot  f^*L(\theta)\in N^1(S),
 \]
 where $f\colon S \to \overline{\cM_A}(\vv,\theta)$ is the classifying morphism and where $c_{\lambda, \xi}:= 1/(\xi^2+1)\lambda(\vv)\in \RR$.

\end{theorem}

 \begin{proof}
 In light of Lemma~\ref{lem:ellsigmatheta}, it suffices to show that 
\begin{equation}
\label{eqn:LCnew}
 f^*L(\theta) = \bigotimes_{1\leq i\leq k} \det\big(\Psi_{\cE}(E_i^\vee)\big)^{\otimes \theta_i},
\end{equation}
 where $\theta = \sum_{1\leq i\leq k} \theta_i [P_i]$. The GIT construction of $\overline{\cM_A}(\vv,\theta)= X\git G$ shows that the $\theta$-semistable locus $X^{\textrm{ss}}$ in $X$ carries a universal family $V$ of framed $\theta$-semistable $A$-modules of dimension vector $\vv$, equipped with an idempotent decomposition $V = \bigoplus_{1\leq i\leq k} V_i$, such that 
 \begin{equation} \label{eqn:onXss}
 \pi^*L(\theta) = \bigotimes_{1\leq i\leq k} \det(V_i)^{\theta_i}
 \end{equation}
 holds $G$-equivariantly on $X^{\textrm{ss}}$,
 where $\pi\colon X^{\textrm{ss}} \to \overline{\cM_A}(\vv,\theta)$ is the quotient map.
 Proposition~\ref{prop:flatfamAmods} shows that $\Psi_{\cE}(E^\vee)$ is a flat family of $\theta$-semistable $A$-modules of dimension vector $\vv$ on $S$. Let $\pi_S \colon \overline{S} \to S$ be the principal $G$-bundle corresponding to a choice of framing (up to a common rescaling) of each summand $\Psi_{\cE}(E_i^\vee)$. By the universality of $V$, it comes with a 
$G$-equivariant map $\overline{f} \colon \overline{S}\to X^{\textrm{ss}}$ that induces the map $f$ between the corresponding quotients, and that satisfies $\overline{f}^* V_i \cong \pi_S^* \bigl(N \otimes \Psi_\cE(E_i^\vee)\bigr)$ for all $i$ and a fixed line bundle $N 
 \in \Pic(S)$. 
 Pulling back \eqref{eqn:onXss} along this map gives the following identity of $G$-equivariant line bundles on $\overline{S}$:
 \begin{align*}
 \pi_S^* f^* L(\theta) & = \overline{f}^* \pi^* L(\theta) = 
 \overline{f}^* \bigotimes_{1\leq i\leq k} \det(V_i)^{\theta_i} \\ &
 =  \pi_S^* \left( N^{\otimes \sum_{1\leq i\leq k}\theta_i \rk(V_i)} \otimes
\bigotimes_{1\leq i\leq k} \det\big(\Psi_{\cE}(E_i^\vee)\big)^{\otimes \theta_i}
\right) 
   =  \pi_S^* \bigotimes_{1\leq i\leq k} \det\big(\Psi_{\cE}(E_i^\vee)\big)^{\otimes \theta_i},
 \end{align*}
where the last identity used $\sum_{1\leq i\leq k}\theta_i \rk(V_i) = \sum_{1\leq i\leq k}\theta_i v_i=0$. This descends
to the identity \eqref{eqn:LCnew} on $S$, as required.
 \end{proof}

 When $\vv$ is primitive and $\theta\in \Theta_{\vv}$ is generic, let $C\subseteq \Theta_{\vv}$ denote the GIT chamber containing $\theta$, let $\cM:=\cM_A(\vv,\theta)$ denote the fine moduli space and write $T=\bigoplus_{1\leq i\leq k}T_i$ for its universal bundle.  Consider the map $L_C\colon \Theta_{\vv}\rightarrow \Pic(\cM)_{\RR}$ given by sending $\eta=\sum_{1\leq i\leq k} \eta_i[P_i]$ to
\[
L_C(\eta):=\bigotimes_{1\leq i\leq k} \det(T_i)^{\otimes \eta_i}.
\]
Note that $L_C(\theta)$ is the polarising ample line bundle on $\mathcal{M}$ determined by the GIT construction.

\begin{corollary}
 For $\vv$ primitive, for $\theta\in C\subset \Theta_{\vv}$ generic, and for any $\lambda\in \Lambda, \xi\in \RR$, let $\cE\in D(Y\times \cM)$ denote the universal family of $\sigma_{\theta, \lambda,\xi}$-stable objects of class $\vv$ with respect to $\cA$. For $c_{\lambda, \xi}:= 1/(\xi^2+1)\lambda(\vv)\in \RR$, the following diagram commutes,
\begin{equation*}
\label{eqn:ample}
\xymatrix{
\Theta_{\vv} \ar[rr]^-{f(-,\lambda,\xi)} \ar[d]_{L_C} & &\Stab(D_c(Y)) \ar[d]^{\ell_\mathcal{E}} \\
\Pic(\cM)_{\RR} \ar[rr]^{c_{\lambda,\xi}\cdot [-]} & & N^1(\cM),
}
\end{equation*}
 where the top horizontal arrow is determined by \eqref{thetatostab} and where the lower horizontal map sends a line bundle to $c_{\lambda, \xi}$ times its numerical divisor class.
 \end{corollary}
 \begin{proof}
 Proposition~\ref{prop:flatfamAmods} (see also Remark~\ref{rem:stableflatfamilies}) implies that
 \begin{equation}
 \label{eqn:TvscE}
 \cE=E\otimes_A T \quad\text{ and }\quad T=\Psi_{\cE}(E^\vee).
 \end{equation}
 Since integral functors commute with direct sum, we have that $T_i=\Psi_{\cE}(E_i^\vee)$ for all $1\leq i\leq k$ because each $T_i$ is indecomposable.  Thus, for any $\eta = \sum_{1\leq i\leq k} \eta_i[P_i]\in \Theta_{\vv}$, we have 
\begin{equation}
\label{eqn:LC}
L_C(\eta) = \bigotimes_{1\leq i\leq k} \det \left( \Psi_\cE(E_i^\vee) \right)^{\otimes \eta_i}.
\end{equation}
 The result follows by comparing this with the numerical divisor class $\ell_{\cE}(\sigma_{\eta,\lambda,\xi})$ from \eqref{eqn:ellE}.
 \end{proof}
 
 \begin{remark}
 When $\cM\cong Y$ and $\cE=\cO_{\Delta}$, equation~\eqref{eqn:TvscE} gives $T=E^\vee$; see Karmazyn~\cite{Karmazyn14} and references therein for many examples where this is known to hold. 
 \end{remark}

\begin{remark}
Given their identification in Theorem \ref{thm:polarisingbundle}, it is instructive to compare the
strengths of two constructions
of the nef divisor class. The GIT construction produces a semiample line
bundle, and consequently a projective coarse moduli space parameterising
S-equivalence classes of semistable objects. On the other hand, the construction via Theorem
\ref{thm:linearisation} works uniformly across the entire space $\Stab(D_c(Y))$ of stability conditions (not just
on the subset corresponding to one particular heart of a t-structure), and gives
a moduli-theoretic interpretation of the class of this line bundle. In particular, this can give better
control of the behaviour of this line bundle at wall-crossings. For example, if (outside a subset of
sufficiently high codimension) a wall-crossing just induces stable objects $E$ to be replaced by
$\Phi(E)$ for some auto-equivalence $\Phi$ of $D_c(Y)$, then the induced action of $\Phi$ on
$\Knumc(Y)$ completely controls the effect of the wall-crossing on the linearisation map.
\end{remark}


\newcommand{\etalchar}[1]{$^{#1}$}

\end{document}